\documentclass[11pt]{article}

\usepackage{amsmath}
\usepackage{geometry}
\usepackage{indentfirst}
\usepackage{bm}
\usepackage{accents}
\usepackage{changepage}
\usepackage{color}
\usepackage{bigints}
\usepackage{yhmath}
\usepackage{dirtytalk}
\usepackage{tensor}
\usepackage{multirow}
\usepackage{enumerate}
\usepackage{enumitem}
\usepackage{mathtools}
\usepackage{latexsym,amsthm,amsfonts,amssymb,amsmath,amsxtra,mathrsfs,stmaryrd,bm}
\usepackage{figsize,subfigure,graphicx,ifthen,calc,psfrag}
\usepackage{upgreek}
\usepackage[greek,english]{babel}
\usepackage[numbers,sort&compress]{natbib}

\usepackage{comment}
\usepackage[backref=page]{hyperref}
\usepackage{caption}
\usepackage{float}
\usepackage{cleveref}
\usepackage{svg}

\newtheorem{thm}{Theorem}[section]
\newtheorem{prop}[thm]{Proposition}
\newtheorem*{defi*}{Definition}
\newtheorem{defi}[thm]{Definition}

\newtheorem{rem}[thm]{Remark}
\newtheorem{coro}[thm]{Corollary}
\newtheorem{ex}{Example}[section]

\newtheorem{Thm}{Theorem}

\makeatletter
\newcommand{\mylabel}[2]{#2\def\@currentlabel{#2}\label{#1}}
\makeatother

\numberwithin{equation}{section}

\allowdisplaybreaks

\usepackage{nomencl}
\makenomenclature

\usepackage{etoolbox}
\providetoggle{nomsort}
\settoggle{nomsort}{true} 

\makeatletter
\iftoggle{nomsort}{%
    \let\old@@@nomenclature=\@@@nomenclature        
        \newcounter{@nomcount} \setcounter{@nomcount}{0}%
        \renewcommand\the@nomcount{\two@digits{\value{@nomcount}}}
        \def\@@@nomenclature[#1]#2#3{
          \addtocounter{@nomcount}{1}%
        \def\@tempa{#2}\def\@tempb{#3}%
          \protected@write\@nomenclaturefile{}%
          {\string\nomenclatureentry{\the@nomcount\nom@verb\@tempa @[{\nom@verb\@tempa}]%
          \begingroup\nom@verb\@tempb\protect\nomeqref{\theequation}%
          |nompageref}{\thepage}}%
          \endgroup
          \@esphack}%
      }{}
\makeatother

\usepackage{multicol}
\def\div{\mathop{\operatorname{div}}\nolimits}

\usepackage{algorithm}
\usepackage[noend]{algpseudocode}

\Crefname{coro}{Corollary}{Corollaries}
\Crefname{thm}{Theorem}{Theorems}

\title{A probabilistic approach to spectral analysis of Cauchy-type inverse problems: Convergence and stability analysis \\[0.5 cm]
}
\author{
Iulian C\^{i}mpean$^{1,3,}$\thanks{Corresponding author. E-mail: \texttt{iulian.cimpean@unibuc.ro;} \href{https://orcid.org/0000-0002-3239-6834}{ORCID ID: 0000-0002-3239-6834}},
Andreea Grecu$^{2,}$\thanks{E-mail: \texttt{andreea.grecu@ismma.ro;} \href{https://orcid.org/0000-0001-5829-7765}{ORCID ID: 0000-0001-5829-7765}},
Liviu Marin$^{1,2,}$\thanks{E-mail: \texttt{liviu.marin@fmi.unibuc.ro;} \href{https://orcid.org/0000-0003-4009-1181}{ORCID ID: 0000-0003-4009-1181}}
}

\date{\small
$^1$Department of Mathematics, Faculty of Mathematics and Computer Science, University of Bucharest,
14~Academiei, 010014~Bucharest, Romania\\
$^2$``Gheorghe Mihoc~--~Caius Iacob" Institute of Mathematical Statistics and Applied Mathematics of the Romanian Academy, 13~Calea 13 Septembrie, 050711~Bucharest, Romania\\
$^3$``Simion Stoilow" Institute of Mathematics of the Romanian Academy, 21~Calea Grivi\c{t}ei,  010702~Bucharest, Romania
}
\begin{document}

\maketitle
\begin{abstract}
A comprehensive convergence and stability analysis of some probabilistic numerical methods designed to solve Cauchy-type inverse problems is performed in this study. 
Such inverse problems aim at solving an elliptic partial differential equation (PDE) or a system of elliptic PDEs in a bounded Euclidean domain, subject to incomplete boundary and/or internal conditions, and are usually severely ill-posed. 
In a very recent paper \cite{CiGrMaI}, a probabilistic numerical framework has been developed by the authors, wherein such inverse problems could be analysed thoroughly by simulating the spectrum of some corresponding direct problem and its singular value decomposition based on stochastic representations and Monte Carlo simulations. 
Herein a full probabilistic error analysis of the aforementioned methods is provided, whereas the convergence of the corresponding approximations is proved and explicit error bounds are provided. 
This is achieved by employing tools from several areas such as spectral theory, regularity theory for elliptic measures, stochastic representations, and concentration inequalities.
\end{abstract}

\noindent \textbf{Keywords:} inverse boundary value problem; Cauchy data; elliptic operator; elliptic measure; harmonic density; H\"older elliptic regularity; probabilistic representation; Monte Carlo methods; walk-on-spheres.

\medskip
\noindent \textbf{Mathematics Subject Classification:} 65N12, 65N15, 65N21, 65N25, 65N75, 35J25, 65C05, 60J65, 65C40.



\section{Introduction and brief presentation of the main results} 
\label{section:intro}
\subsection{Formulation of the inverse problem}
\label{ss:formulation}
Consider a bounded domain (open and connected) $D \subset \mathbb{R}^d$, $d \geq 2$,   which is occupied by a material characterised by an inhomogeneous and anisotropic conductivity tensor $\mathbf{K} \coloneqq \big( K_{ij}(\mathbf{x}) \big)_{i,j=\overline{1,d}} \in \mathbb{R}^{d \times d}$, $\mathbf{x} \in \overline{D}$, such that
\begin{enumerate}[label=({\bf H.\arabic*})]
\item \label{Hyp_D} $\partial{D} = \Gamma_0 \cup \Gamma_1$ with $\Gamma_1 \cap \Gamma_0 = \varnothing$, where $\Gamma_0$ and $\Gamma_1$ are regarded as the accessible and inaccessible parts of the boundary $\partial{D}$, respectively;
\item \label{Hyp_K} $\mathbf{K} \in \mathbb{R}^{d \times d}$ is a symmetric and strictly elliptic matrix of bounded and measurable coefficients; as a matter of fact, this work is mostly devoted to the case of homogeneous coefficients, namely $\mathbf{K}$ is independent of $x\in D$.
\end{enumerate}
If $\Gamma_0 \subset \partial D$ is an {\it open Lipschitz portion} of $\partial D$ in the sense of \cite{Al09}, then the classical {\it inverse Cauchy problem} ({\it ICP}), see \Cref{Fig00a}, reads as:
\begin{eqnarray}  \label{eq:ICP}
\left|
\begin{array}{l}
\textrm{Given } u_0 \in H^{1/2}(\Gamma_0) \textrm{ and } q_0 \in H^{-1/2}(\Gamma_0), \textrm{ find } u \in H^1(D) \textrm{ the weak solution to }\\[5pt]
\begin{array}{lll}
\qquad
\div\big(\mathbf{K} \nabla{u}\big) = 0~\textrm{ in } D, \quad &
u = u_0~\textrm{ on } \Gamma_0, \quad &
\mathbf{n} \cdot \big(\mathbf{K} \nabla{u}\big) = q_0~\textrm{ on } \Gamma_0,
\end{array}\\[5pt]
\textrm{where } \mathbf{n} \textrm{ denotes the outward unit normal at } \Gamma_0.
\end{array}
\right.
\end{eqnarray}
Note that recovering $u$ is essentially equivalent to finding the unknown Dirichlet data $u_1 \coloneqq u\big|_{\Gamma_1}$, assuming that the trace of $u$ on $\Gamma_1$ has a consistent meaning. 
The inverse problem \eqref{eq:ICP} only stands as a particular continuous version of a discrete inverse problem of interest that is actually addressed herein, see problem \eqref{eq:insidemeasurements}--\eqref{eq:u_0=0} and \Cref{Fig00b}.

The literature devoted to inverse problems for partial differential equations and, in particular, related to \eqref{eq:ICP}, is vast. To name a few monographs on inverse problems related to the present work, we only refer the reader to, e.g., \cite{gr93,sn99,go02,ta05,is06,st10,ki11,ne12,al12,Le21}. 
The applications of such inverse problems are also various, see, for example, \cite{al12,bo98,andrieux2005data,alessandrini1993stable,AnAb96,BrHaPi01,liu2008modified,We15,al19,fr05,co08,ma20,clerc2007cortical,koshev2020fem,colli1985mathematical,FrMa79}.

\subsection{Instability of the inverse problem} 
\label{ss:instability}
Firstly, note that based on the unique continuation property for elliptic operators, or, more precisely, by assuming $\mathbf{K}$ is Lipschitz for $d \geq 3$ and using \cite[Theorem 1.7]{Al09}, it follows there exists at most one weak solution $u$ to problem \eqref{eq:ICP} for the prescribed pair of data $\big( u_0, q_0 \big) \in H^{1/2}(\Gamma_0) \times H^{-1/2}(\Gamma_0)$.
Moreover, it is well-known that reconstructing this unique solution, provided that such a solution exists, from the prescribed measurements on $\Gamma_0$ is a severely unstable problem, see \cite{Ha23}.
Consequently, solving numerically problem \eqref{eq:ICP} is typically very challenging, whilst various regularising techniques have been proposed, explored, and compared in the literature.
There is a very high number of papers providing numerical evidence that one could, in principle, tame the instability issue for the inverse Cauchy problem for several benchmark examples. 
However, only very few studies contain a comprehensive convergence analysis that provides a complete understanding of this issue and/or reliable quantitative error estimates.

A classical illustration of the curse of instability that affects the inverse Cauchy problem \eqref{eq:ICP} was given by Hadamard \cite{Ha23} through the following ubiquitous example. Let $u^{(n)} \colon (0, \infty) \times \mathbb{R} \to \mathbb{R}$, $n \geq 1$, be the solution to the following problem
\begin{equation*}
\Delta u^{(n)} = 0 \mbox{ in } (0, \infty) \times \mathbb{R}, \quad 
u^{(n)}(0,y) = u_0^{(n)}(0,y) \coloneqq \dfrac{1}{n} \sin(ny), \quad 
u^{(n)}_x(0,y) = q_0^{(n)}(0,y) \coloneqq 0, \quad 
y \in \mathbb{R},
\end{equation*}
that is explicitly given by 
\begin{equation*}
u^{(n)}(x,y) = \dfrac{1}{n} \sin(ny)\left(\mathrm{e}^{nx} + \mathrm{e}^{-nx}\right), \quad 
(x,y) \in (0, \infty) \times \mathbb{R}, \quad 
n \geq 1. 
\end{equation*}
Clearly, it follows that, as $n \to \infty$, the Cauchy data $\big( u_0^{(n)}, q_0^{(n)} \big)$ converge to the boundary data $\big( u_0, q_0 \big) \equiv \big( 0, 0 \big)$, which are uniquely compatible to the solution $u^{(0)} \equiv 0$ in $(0, \infty) \times \mathbb{R}$. 
However, $u^{(n)}$ diverges exponentially fast as $n \to \infty$, hence the severe instability of the inverse problem.

Since Hadamard's example, the severe instability of the inverse Cauchy problem \eqref{eq:ICP} has been widely accepted as a folklore result and has been encountered in most of the numerical experiments or practical situations. 
To the best of our knowledge, a first theoretical proof of the severe instability has been given by Ben Belgacem \cite{Be07}. 
More precisely, it was shown that solving \eqref{eq:ICP} reduces to inverting an associated Steklov-Poincar\'{e} operator $\mathcal{S}$ acting on $L^2(\Gamma_1)$, provided that $D$ is $C^\infty$, $\mathbf{K} = \mathbf{I}_d$ and $\overline{\Gamma}_0 \cap \overline{\Gamma}_1 = \varnothing$, see also \cite{Be05}. 
It was further proved that the operator $\mathcal{S}$ is Hilbert-Schmidt and positive definite on $L^2(\Gamma_1)$ and, moreover, its eigenvalues $\lambda_k$, $k\geq 1$, decay to zero faster than any negative power of $k$.

A further analysis has been developed in \cite{Al09}, where $D$ is assumed to be a general bounded domain, such that $\Gamma_0$ is a Lipschitz portion of $\partial D$, $\mathbf{K}$ satisfies \ref{Hyp_K} together with the Lipschitz condition if $d \geq 3$, the elliptic operator is allowed to contain a bounded zeroth-order term, and the source term $f$ is not necessarily zero but from $L^2(D)$. 
In this setting, it is shown that a global reconstruction of the solution $u$ can be achieved, provided that some \textit{a priori} bound for $\Vert u \Vert_{H^1(D)}$ is given, such that
\begin{equation*}
\Vert u \Vert_{L^2(D)} \lesssim 
\big \vert \log \big( 
\Vert f \Vert_{L^2(D)} 
+ \Vert u_0 \Vert_{H^{1/2}(\Gamma_0)} 
+ \Vert q_0 \Vert_{H^{-1/2}(\Gamma_0)} 
\big) \big \vert^{-\mu}, \quad 
\mbox{ for some } \mu \in (0,1),
\end{equation*}
where the tacitly understood constant depends on $\Vert u \Vert_{H^1(D)}$ among other quantities; for more details, see \cite[Theorem 1.9]{Al09}. 
Thus, a theoretically stable global reconstruction can be achieved, however the above poor logarithmic stability emphasises once again that, in general, problem \eqref{eq:ICP} is difficult to be numerically solved globally in $D$ even if some \textit{a priori} bounds on the solution are available.
If the reconstruction of $u$ is considered only locally or strictly away from the inaccessible boundary $\Gamma_1$, then the stability can be enhanced. 
More precisely, by \cite[Theorem 1.7]{Al09}, if $G$ is an open subset of $D$ such that ${\rm dist}(G, \partial D) > 0$, then $u$ can be stably reconstructed in $D$ in the sense that
\begin{equation}\label{eq:local_stability}
\Vert u \Vert_{L^2(G)} \lesssim 
\left(\frac{|D|}{{\rm dist}(G,\partial D)^d}\right)^{1/2} \big( 
\Vert f \Vert_{L^2(D)} 
+ \Vert u_0 \Vert_{H^{1/2}(\Gamma_0)}
+ \Vert q_0 \Vert_{H^{-1/2}(\Gamma_0)} 
\big)^{\delta},
\end{equation}
where $\delta \coloneqq \alpha^{C |D|/{\rm dist}(G, \partial D)^d}$ for some $\alpha \in (0,1)$, $C = C(\mathbf{K}) > 0$, and the tacitly understood constant depends on the \textit{a priori} given $\Vert u \Vert_{L^2(D)}$ among other quantities.
Therefore, in $D$, one can hope for at least a H\"{o}lder $L^2-$stability for the locally reconstructed solution and, in general situations, this seems to be optimal. 
It should be noted that even for the local reconstruction problem subject to \textit{a priori} bounds for the unknown solution, the factor $\left(\dfrac{|D|}{{\rm dist}(G, \partial D)^d}\right)^{1/2}$ from \eqref{eq:local_stability} may be quite large, whilst the stability exponent $\delta$ may be close to zero. In such cases, there is only little hope that relation \eqref{eq:local_stability} can provide practical error bounds for a locally reconstructed numerical solution.

In spite of the aforementioned severe instability, there are numerous papers available in the literature and devoted to numerical methods aiming at solving stably the inverse problem \eqref{eq:ICP} - however, surprisingly successfully - at least for certain benchmark examples, see \cite[Introduction]{CiGrMaI} for a detailed discussion on these topics. 
This gap between the severe instability of the inverse problem or the relatively weak theoretical guarantees for its stability available, on the one hand, and the specific success of the various numerical algorithms proposed in the literature, on the other hand, has remained poorly understood, at least to us, and, at the same time, represents a real challenge.

\subsection{Description of the approach by C\^{i}mpean et al.~\texorpdfstring{\cite{CiGrMaI}}{CiGrMaI}} 
\label{ss:approach}
Filling the aforementioned gap by analysing and simulating efficiently the spectrum of some associated direct operator, using probabilistic representations of the elliptic densities and fast Monte Carlo simulations, has been one of the main aims of the approach by C\^{i}mpean et al.~\cite{CiGrMaI}. 
In this way, one can thoroughly understand and quantify the impact of the geometry of the domain and the structure of the conductivity coefficients on the (in)stability of the inverse problem under investigation. 
To clarify this perspective, we start off by considering the informal statement that the knowledge of the pair of Cauchy data $\big( u_0, q_0 \big) \in H^{1/2}(\Gamma_0) \times H^{-1/2}(\Gamma_0)$ is essentially equivalent to that of $u_0 \coloneqq u\big|_{\Gamma_0}$ together with the solution $u$ on the boundary $\Gamma_D \subset D$ of some thin shell of $\Gamma_0$.
More precisely, instead of studying \eqref{eq:ICP}, in \cite{CiGrMaI} it is considered the essentially equivalent inverse problem of finding $u$ which solves
\begin{eqnarray} \label{eq:insidemeasurementsf0fh}
\left|
\begin{array}{l}
\textrm{Given } u_0 \in H^{1/2}(\Gamma_0) \textrm{ and } u_D \in H^{1/2}(\Gamma_D), \textrm{ find } u \in H^1(D) \textrm{ the weak solution to }\\[5pt]
\begin{array}{lll}
\qquad
\div\big(\mathbf{K} \nabla{u}\big) = 0~\textrm{ in } D, \quad &
u = u_0~\textrm{ on } \Gamma_0 \subset \partial{D}, \quad &
u = u_D~\textrm{ on } \Gamma_D \subset D.
\end{array}
\end{array}
\right.
\end{eqnarray}
For example, $\Gamma_D \coloneqq \partial{\big\{ \mathbf{x} \in D \: \big| \: {\rm dist}(\mathbf{x}, \Gamma_0) \leq \varepsilon \big\}} \cap D$ for some $0 < \varepsilon \ll 1$.
Transforming the pair of Cauchy data $\big( u_0, q_0 \big)$ on $\Gamma_0$ into $\big( u_0, u\big|_{\Gamma_D} \big)$ can be achieved by, e.g., a first-order Taylor expansion inside the domain of the Dirichlet data on $\Gamma_0$ as explained in detail in \cite{CiGrMaI}. 
Moreover, the local stability \eqref{eq:local_stability} provides one with a guarantee that given both the Dirichlet and the Neumann data on $\Gamma_0$, at least in the proximity of $\Gamma_0$, one can stably reconstruct the unknown solution $u$. Furthermore, once this is done, one can discard the Neumann data and regard the problem of reconstructing the solution in the rest of the domain or on $\Gamma_1$ as a particular case of \eqref{eq:insidemeasurements}.

Before proceeding with the details on the approach by C\^{i}mpean et al.~\cite{CiGrMaI}, some important aspects on the notion of the solution to problem \eqref{eq:insidemeasurementsf0fh} need to be clarified. 

\begin{defi}\label{def:local solution}
A function $u \in H^1_{loc}(D)$ is called a local solution of $\div \big(\mathbf{K} \nabla \big)$ in $D$ if
\begin{equation*}
\int_D \nabla{u}(\mathbf{x}) \cdot \big(\mathbf{K}(\mathbf{x}) \nabla{\varphi}(\mathbf{x})\big) \, \mathrm{d}\mathbf{x} = 0, \quad \forall~\varphi \in C^1_c(D),
\end{equation*}
where $C^1_c(D)$ denotes the set of continuously differentiable functions with a compact support in $D$.
\end{defi}

\begin{defi}\label{def:local solution bvp}
A function $u \in H^1_{loc}(D)$ is called a solution of the Dirichlet problem
\begin{equation}\label{eq:dirichlet bvp}
\div \big(\mathbf{K} \nabla{u}\big) = 0  \mbox{ in } D, \quad
u = f  \mbox{ on } \partial D,
\end{equation}
where $f \in C(\partial D)$, if $u$ is a local solution of $\div \big(\mathbf{K} \nabla \big)$ in $D$ and
\begin{equation}\label{eq:limit regular point}
\lim_{\substack{\mathbf{x} \to \mathbf{y} \\ \mathbf{x} \in D}} u (\mathbf{x}) = f(\mathbf{y}),
\quad \forall\; \mathbf{y}\in \partial D \mbox{ regular point}.
\end{equation}
\end{defi}

\begin{defi}\label{def:solution IP}
On assuming that $u_0 \in C(\Gamma_0)$ and $u_D \in C(\Gamma_D)$ are prescribed, a function $u \in H^1_{loc}(D)$ is called a solution to problem \eqref{eq:insidemeasurementsf0fh} if $u$ is a local solution of $\div \big(\mathbf{K} \nabla \big)$ in $D$, $u\big|_{\Gamma_D} = u_D$, and the prescribed Dirichlet condition on $\Gamma_0$ is understood in the sense of relation \eqref{eq:limit regular point}.
\end{defi}

For a precise definition and characterisations of a regular boundary point, we refer the reader to \cite{Wiener24}, \cite[Definition 3.2 and Lemma 3.2]{Stampacchia63}, \cite[Corollary 9.1]{Stampacchia63} or \cite[Sections 2.8--2.9]{GiTr01}.

According to De Giorgi \cite{DeGiorgi57} and Nash \cite{Nash58}, any local solution of $\div \big( \mathbf{K} \nabla \big)$ in $D$ admits a H\"older continuous version in any compact subdomain of $D$ and, in particular, it has a continuous version which will always be further referred to herein.

It is well known that if $D \subset \mathbb{R}^d$ is a bounded domain and $f \in C(\partial D)$, then there exists a unique solution to the Dirichlet problem \eqref{eq:dirichlet bvp}, see \cite[Theorem 1.1]{ChenZhao95}. 
Moreover, for any $\mathbf{x} \in D$, there exists a unique probability measure $\mu_\mathbf{x}$ on $\partial D$ such that 
\begin{equation}\label{eq:representation_C}
u(\mathbf{x}) = 
\int_{\partial D} f(\mathbf{y}) \, \mathrm{d} \mu_{\mathbf{x}}(\mathbf{y}), 
\quad \forall~f \in C(\partial D),
\end{equation}
where $u \in H^1_{loc}(D) \cap C(D)$ is the corresponding solution to the Dirichlet problem \eqref{eq:dirichlet bvp}.
The probability measure $\mu_\mathbf{x}$ is called the elliptic measure associated with $\div \big(\mathbf{K} \nabla \big)$ in $D$ with the pole $\mathbf{x} \in D$. 
When $\mathbf{K} = \mathbf{I}_d$, $\mu_\mathbf{x}$ is referred to as the harmonic measure.

Furthermore, it is also well known that any two measures $\mu_\mathbf{x}$ and $\mu_\mathbf{y}$ are mutually absolutely continuous and if there exists $\mathbf{x}_0 \in D$ such that $f \in L^1(\mu_{\mathbf{x}_0})$, then $f \in L^1(\mu_\mathbf{x})$ for any $\mathbf{x} \in D$. Moreover, if one sets
\begin{equation}\label{eq:L1-representation}
u(\mathbf{x}) \coloneqq \int_{\partial D} f(\mathbf{y}) \, \mathrm{d} \mu_\mathbf{x}(\mathbf{y}), 
\quad \mathbf{x} \in D,
\end{equation}
then $u$ is a local solution of $\div \big(\mathbf{K} \nabla{u}\big) = 0$ in $D$ and, at the same time, is locally H\"older continuous, see \cite{CiGrMaI} and the references therein.

The following uniqueness of solutions to the inverse problem \eqref{eq:insidemeasurementsf0fh} holds. 

\begin{prop}[C\^{i}mpean et al. \cite{CiGrMaI}] \label{prop:uniqueness IP} 
Suppose that $D \subset \mathbb{R}^d$ is a bounded domain and $\mathbf{K} : D \longrightarrow \mathbb{R}^{d \times d}$ is a bounded measurable, symmetric and strictly elliptic (positive definite) matrix-valued function. In addition, if $d \geq 3$, assume that $\mathbf{K}$ is Lipschitz. 
Let $\Gamma_D \subset D$ be a compact set, $\varnothing \neq I \subset D$ be an open set such that $\partial I \subset \Gamma_D \cup \Gamma_0$, $u_0 \in C(\Gamma_0)$ and $u_D \in C(\Gamma_D)$.\\
Then there exists at most one solution $u \in H^1_{loc}(D) \cap C(D)$ to the inverse problem \eqref{eq:insidemeasurementsf0fh}.
\end{prop}

If $u \in C(\overline{D})$ is a solution to \eqref{eq:insidemeasurementsf0fh} and $u_1 := u\big|_{\Gamma_1}$ is the unknown Dirichlet data on $\Gamma_1$, then it follows that
\begin{align}
u(\mathbf{x}) 
&\coloneqq \int_{\Gamma_0} u_0(\mathbf{y}) \, \mathrm{d} \mu_\mathbf{x}(\mathbf{y}) 
+ \int_{\Gamma_1} u_1(\mathbf{y}) \, \mathrm{d} \mu_\mathbf{x}(\mathbf{y}), 
\quad \mathbf{x} \in \Gamma_D.\nonumber 
\end{align}
Without any loss of the generality, if we further assume that $u_0 \equiv 0$, then
\begin{align}
u(\mathbf{x}) 
&= \int_{\Gamma_1} u_1(\mathbf{y}) \, \mathrm{d} \mu_\mathbf{x}(\mathbf{y}) 
\eqqcolon \big( \mathcal{T} u_1 \big)(\mathbf{x}), 
\quad \mathbf{x} \in \Gamma_D.
\label{eq:T_rond}
\end{align}
The operator $\mathcal{T}$ can be uniquely extended to a bounded linear operator, also denoted by $\mathcal{T}$, as
\begin{equation}\label{eq:compact operator}
\mathcal{T} \colon L^1(\Gamma_1;\mu_{\mathbf{z}}) \longrightarrow C(\Gamma_D), \qquad 
\big( \mathcal{T} u_1 \big)(\mathbf{x}) 
= \int_{\Gamma_1} u_1(\mathbf{y}) \, \mathrm{d} \mu_\mathbf{x}(\mathbf{y}), 
\quad \forall~\mathbf{x} \in \Gamma_D, 
\quad \forall~u_1 \in L^1(\Gamma_1;\mu_{\mathbf{z}}).
\end{equation}
Note that definition \eqref{eq:compact operator} is independent of $\mathbf{z} \in D$. 

\begin{prop}[C\^{i}mpean et al. \cite{CiGrMaI}] \label{coro:compact} 
The linear operator $\mathcal{T}$ defined by \eqref{eq:compact operator} is injective and compact. 
\end{prop}
Consequently, $\mathcal{T}$ does not have a bounded inverse even if its domain is chosen to be maximal, in the sense that the integral representation \eqref{eq:T_rond} makes sense. 
However, the representation \eqref{eq:compact operator} is the key ingredient in \cite{CiGrMaI} since it reveals that {\it the (in)stability of the inverse problem \eqref{eq:ICP} is fundamentally related to how much $\Gamma_1$ is charged by the elliptic measures $\mu_{\mathbf{x}}$ with the poles in the proximity of $\Gamma_0$.} 
In other words, the stability issue is strongly determined by the distribution on $\Gamma_1$ of the underlying diffusion process started in the proximity of $\Gamma_0$ at its first exit time from $D$.

Proceeding with the approach of C\^{i}mpean et al. \cite{CiGrMaI}, it is further assumed that only a finite number of measurements are available on both $\Gamma_0$ and $\Gamma_D$, namely the following discrete version of the continuous inverse problem \eqref{eq:ICP} is tackled and addressed, see also \Cref{Fig00b}:
\begin{eqnarray} \label{eq:insidemeasurements}
\left|
\begin{array}{l}
\textrm{Assuming that any boundary point of } D \textrm{ is regular and given } M_0 \textrm{ boundary points }
\big( \mathbf{x}_i^0 \big)_{i= \overline{1,M_0}} \subset \Gamma_0 \\ 
\textrm{and } M_D \textrm{ internal points } \big( \mathbf{x}_i^D \big)_{i= \overline{1,M_D}} \subset D, \textrm{ find } u \in C(\overline{D}) \cap H^1_{\rm loc}(D) \textrm{ such that}\\[5pt]
\begin{array}{lll}
\qquad
\div\big(\mathbf{K} \nabla{u}\big) = 0~\textrm{ in } D, \quad &
u\big( \mathbf{x}^0_i \big) = u^0_i,~~i = \overline{1, M_0}, \quad &
u\big( \mathbf{x}^D_i \big) = u^D_i,~~i = \overline{1, M_D}.
\end{array}
\end{array}
\right.
\end{eqnarray}
Here the partial differential equation above is satisfied in a distributional sense (see \Cref{def:local solution}), $\mathbf{u}^0 \coloneqq \big( u^0_i \big)_{i = \overline{1, M_0}} \in \mathbb{R}^{M_0}$ are boundary measurements in the sense of \Cref{def:local solution bvp}, relation \eqref{eq:limit regular point}, and $\mathbf{u}^D \coloneqq \big( u^D_i \big)_{i = \overline{1, M_D}} \in \mathbb{R}^{M_D}$ are internal measurements.
\begin{figure}[H]
\centering
\subfigure[]{
\includegraphics[scale=0.80]{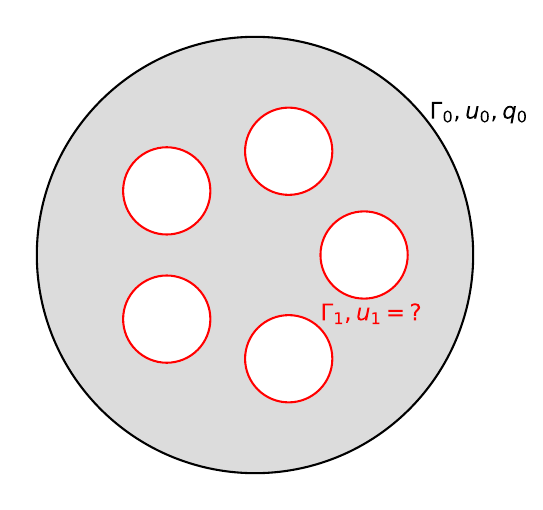}
\label{Fig00a}}
\subfigure[]{
\includegraphics[scale=0.70]{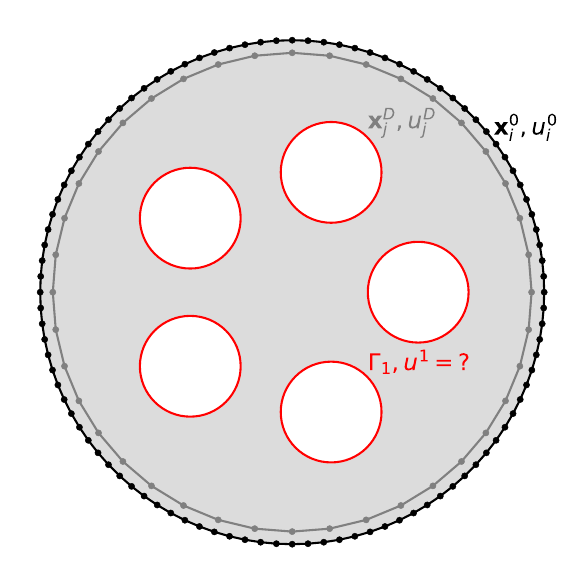}
\label{Fig00b}}
\caption[Figure]{\footnotesize Schematic diagram of \subref{Fig00a}~the continuous inverse Cauchy problem, and \subref{Fig00b}~the discrete version of its generalisation.}
\label{Fig00}
\end{figure}

Although the numerical simulations in \cite{CiGrMaI} have been performed for non-homogeneous Dirichlet data on $\Gamma_0$, for simplicity and without any loss of the generality, the following additional assumption related to the inverse problem \eqref{eq:insidemeasurements} is made 
herein, namely
\begin{equation} \label{eq:u_0=0}
u\big|_{\Gamma_0} = 0~(\textrm{hence}~u^0_i = 0,~i = \overline{1, M_0}).    
\end{equation}
To be more precise wee seek for a solution $u$ to the inverse problem \eqref{eq:insidemeasurements}--\eqref{eq:u_0=0} which admits the representation \eqref{eq:T_rond}. 
It is further assumed that for a point $\mathbf{x} \in D$ (and hence for any point $\mathbf{x} \in D$), the elliptic measure $\mu_\mathbf{x}$ admits a Radon-Nikodym derivative with respect to the surface measure $\sigma$ on $\partial D$
\begin{enumerate}[label=({\bf H.3})]
\item \label{eq:Hq} $\rho_{\mathbf{x}} \coloneqq \dfrac{\mathrm{d} \mu_\mathbf{x}}{\mathrm{d} \sigma} \in L^2(\partial D) 
\textrm{ or merely } 
\rho_{\mathbf{x}} \coloneqq \dfrac{\mathrm{d} \mu_\mathbf{x}}{\mathrm{d} \sigma} \in L^2(\Gamma_1)$,
\end{enumerate}
so that the inverse problem \eqref{eq:insidemeasurements}--\eqref{eq:u_0=0} reduces to solving, for $u_1 \in L^2(\Gamma_1)$, the linear operator equation
\begin{equation}\label{eq:inverse:T0}
Tu_1 = \mathbf{u}^D,
\end{equation}
where
\begin{equation}\label{eq:inverse:T0_def}
T \colon L^2(\Gamma_1) \longrightarrow \mathbb{R}^{M_D}, \qquad 
L^2(\Gamma_1)\ni u_1 \longmapsto Tu_1 = 
\left( 
\mathcal{T}u_1\big( \mathbf{x}_i^D \big) 
\right)_{i = \overline{1,M_D}} \in \mathbb{R}^{M_D}.
\end{equation}

\begin{rem} Condition {\rm \ref{eq:Hq}} is satisfied for a large class of domains $D \subset \mathbb{R}^d$ and conductivity tensors $\mathbf{K} \in \mathbb{R}^{d \times d}$, see, e.g., 
\cite[Theorem 1]{FaJeKe84}, \cite{Dahlberg86}, \cite{Fefferman89}, \cite[Theorem 3.5~(ii)]{GrWi82}), \cite[Theorem 3.1]{JeKe81}. 
\end{rem}
Recall that the internal measurements $\big( u_i^D \big)_{i = \overline{1,M_D}}$ do not simply form a vector in $\mathbb{R}^{M_D}$, but are discrete values of $u \in C(D)$ at $\left(\mathbf{x}_i^D\right)_{i = \overline{1,M_D}}$. 
Consequently, the inner product on $\mathbb{R}^{M_D}$ should be properly scaled as $M_D \to \infty$. 
In fact, $\mathbb{R}^{M_D}$ will be endowed with the Euclidean inner product and the operator $T$ will be rescaled as follows. 
Consider a vector $\boldsymbol{\nu} \coloneqq \big( \nu_i \big)_{i = \overline{1,M_D}} \in (0,1)^{M_D}$ such that 
$\displaystyle \sum_{i=1}^{M_D} \nu_i^2 = 1$. 
The direct operator associated with the inverse problem \eqref{eq:insidemeasurements} is then given by
\begin{equation*}
T_{\boldsymbol{\nu}} \colon \big( L^2(\Gamma_1), \Vert \cdot \Vert_{L^2(\Gamma_1)} \big) \longrightarrow 
\big( \mathbb{R}^{M_D}, \Vert \cdot \Vert \big), \qquad 
u_1 \in L^2(\Gamma_1) \longmapsto T_{\boldsymbol{\nu}} u_1 \coloneqq 
\left( 
\nu_i T u_1\big( \mathbf{x}_i^D \big) 
\right)_{i = \overline{1,M_D}} \in \mathbb{R}^{M_D}.
\end{equation*}
Thus, the operator equation \eqref{eq:inverse:T0} associated with the inverse problem \eqref{eq:insidemeasurements}--\eqref{eq:u_0=0} now becomes
\begin{eqnarray} \label{eq:ip u0=0}
\left|
\begin{array}{l}
\textrm{Given the internal measurements } \mathbf{u}^D \coloneqq \left(u_i^D\right)_{i=\overline{1,M_D}} \in \mathbb{R}^{M_D}, \textrm{ find } u_1 \in L^2(\Gamma_1) \textrm{ such that}\\[5pt]
\hspace*{6.0cm}
T_{\boldsymbol{\nu}} u_1 = \rm{diag}(\boldsymbol{\nu}) \; \mathbf{u}^D.
\end{array}
\right.
\end{eqnarray}
In fact, the main achievements of the approach by C\^{i}mpean et al. \cite{CiGrMaI} consist of the representation and simulation of the spectrum of the symmetrised operator $T_{\boldsymbol{\nu}}^\ast T_{\boldsymbol{\nu}} \colon L^2(\Gamma_1) \longrightarrow L^2(\Gamma_1)$, as well as the determination of the corresponding truncated singular value decomposition (TSVD) solutions.

\begin{rem}
The SVD of operator $T_{\boldsymbol{\nu}}^\ast T_{\boldsymbol{\nu}}$ depends heavily upon the choice of the scaling vector $\boldsymbol{\nu}$. 
A natural choice of the scaling vector is $\boldsymbol{\nu} \coloneqq \big( 1/\sqrt{i} \big)_{i = \overline{1, M_D}}$ and this corresponds to a uniform relevance of the internal measurements taken at $\big( \mathbf{x}_i^D \big)_{i = \overline{1, M_D}}$.
However, $\boldsymbol{\nu}$ may be chosen in many other ways meant to capture the regularity of $u$ on $\Gamma_D$.
\end{rem}
To analyse numerically the spectrum of the symmetrised operator $T_{\boldsymbol{\nu}}^\ast T_{\boldsymbol{\nu}}$, several approximation steps have been employed in \cite{CiGrMaI} and these will be recalled with more details in the following sections. 
At this point, to present concisely the main results, we list below the quantities of main interest and their corresponding estimators. 
The indices $\omega^1$, $\varepsilon$ and $N$ occurring below correspond to the discretisation of $\Gamma_1$ (see \Cref{ss:discret}), the $\varepsilon-$shell of the boundary used to stop the walk-on-elipsoids chain (see \Cref{s:3}), and the number of Monte Carlo samples (see \Cref{s:3}), respectively.
\begin{enumerate}[label={\rm (}\roman*{\rm )}]
\item The elliptic density $\rho_{\mathbf{x}_i^D}$ on the inaccessible boundary $\Gamma_1 \subset \partial D$ is approximated by the Monte Carlo estimator $\rho_{\mathbf{x}_i^D, \omega^1, \varepsilon, N}$ defined by \eqref{eq:rho omega1 MC}.

\item The matrix operator $\Lambda^{\boldsymbol{\nu}} \coloneqq T_{\boldsymbol{\nu}} T_{\boldsymbol{\nu}}^\ast$ is approximated by the Monte Carlo estimator $\boldsymbol{\Lambda}^{\boldsymbol{\nu}}_{\omega^1, \varepsilon, N}$ given by \eqref{eq:Lambda_tilde_nu}.

\item The eigenvalues of the operator $T_{\boldsymbol{\nu}}^\ast T_{\boldsymbol{\nu}}$ defined by \eqref{eq:B Th*Th} are approximated by the estimators $\widetilde{\lambda}_1 \geq \widetilde{\lambda}_2 \geq \ldots \geq \widetilde{\lambda}_{M_D}$ which are the eigenvalues of $\boldsymbol{\Lambda}^{\boldsymbol{\nu}}_{\omega^1, \varepsilon, N}$.

\item The eigenvectors of $T_{\boldsymbol{\nu}}^\ast T_{\boldsymbol{\nu}}$ are estimated by
\begin{equation} \label{eq:eigenfunctions_MC}
\widetilde{u}_j 
\coloneqq \widetilde{\lambda}_j^{-1/2} 
\displaystyle \sum \limits_{i=1}^{M_D} 
\nu_i \big[ \widetilde{\mathbf{u}}_j \big]_i \rho_{\mathbf{x}_i^D, \omega^1, \varepsilon, N}, 
\quad j = \overline{1, k},
\end{equation}
where $k \coloneqq {\rm dim}\left( {\rm Ker}\big( T_{\boldsymbol{\nu}}^\ast T_{\boldsymbol{\nu}} \big) \right)^{\perp} \leq M_D$ and $\big( \widetilde{\mathbf{u}}_j \big)_{j=\overline{1,M_D}}$ is an orthonormal eigenbasis of $\left( {\rm Ker} \big( \boldsymbol{\Lambda}^{\boldsymbol{\nu}}_{\omega^1, \varepsilon, N} \big) \right)^{\perp}$.

\item The $r-$TSVD solution to the equation $T_{\boldsymbol{\nu}} u_1 = {\rm diag}(\boldsymbol{\nu}) \; \mathbf{b}$, for some $\mathbf{b} \in \mathbb{R}^{M_D}$, defined by \eqref{eq:u_r}, is approximated by the Monte Carlo estimator $u_{\omega^1, \varepsilon, N}^{(r)}$ defined by \eqref{eq:u_omega_M1_r_MC}--\eqref{eq:u_omega_M1_r_MC A}, for various levels of the truncation parameter $r$.
\end{enumerate}

\subsection{Brief description of the main results} 
\label{ss:presentation}
The main aim of this study is to provide a comprehensive convergence and stability analysis of the estimators (i)--(v) employed in the approach of \cite{CiGrMaI}. To do so, the following {\it main assumptions} are made:
\begin{enumerate}[label={\rm (}\alph*{\rm )}]
\item Assumptions \ref{Hyp_D}--\ref{eq:Hq} hold together with the additional one, namely
    \begin{enumerate}[label=({\bf H.4})]
    \item \label{eq:H4} $\mathbf{K}$ is homogeneous, i.e. independent of the space variable $\mathbf{x}$.
    \end{enumerate} 
Analogous to \cite{CiGrMaI}, assumption \ref{eq:H4} is considered for two reasons, namely ({\it i})~the presentation of the Monte Carlo estimators and the corresponding algorithm are simpler, and ({\it ii})~for homogeneous conductivity coefficients, one can directly use the existing and very fast walk-on-spheres or, in the present case, walk-on-ellipsoids numerical algorithm. 
\Cref{thm:main1} is nevertheless valid for a general conductivity tensor $\mathbf{K}$. 
Using fundamentally similar steps, one could extend the present analysis to the case of non-homogeneous and sufficiently smooth conductivity coefficients, whilst in the case of piecewise constant conductivity coefficients, one could rely on the walk-on-spheres algorithm derived in \cite{Ta10}. 

\item The assumptions from \Cref{coro:MCmain} and \Cref{coro:tail_lambda} concerning the regularity of the domain $D$ and the conductivity tensor $\mathbf{K}$ are satisfied.

\item The assumptions in \Cref{coro: error estimates} or \Cref{coro:mu - mu eepsilon-Voronoi} regarding the regularity of the interpolation weights $\omega^1$ defined on $\Gamma_1$ are fulfilled.

\item The elliptic densities $\rho_{\mathbf{x}_i^D}$, $i = \overline{1,M_D}$, are continuous on $\overline{\Gamma_1}$. This condition could be, in principle, relaxed as indicated in \Cref{rem:osc regularity}~(iii), however it requires some additional work.  
\end{enumerate}

In the aforementioned setting (a)--(d), the main aim of the present study is achieved in the form of the following four qualitative results. 
Their quantitative versions are more involved and presented in detail in the main body of the paper.
\begin{Thm}[Elliptic densities approximation]
\label{thm:A}
\begin{equation}
\lim_{{\rm diam}(\omega^1) \to 0} 
\lim_{\varepsilon \to 0} 
\lim_{N \to \infty} 
\sup \limits_{\mathbf{x} \in \Gamma_1} 
\Big\vert 
\rho_{\mathbf{x}_i^D}(\mathbf{x})
- \rho_{\mathbf{x}_i^D,\omega^1,\varepsilon, N}(\mathbf{x}) 
\Big\vert = 0~\mathbb{P}\mbox{--a.s. and in } L^1(\mathbb{P}).
\end{equation}
\end{Thm}
\noindent More details and further explicit tail bounds for the convergence error are given by \Cref{coro:rho_MC_est}.

The following two main results are concerned with the approximation of the spectrum of the symmetrised operator $T_{\boldsymbol{\nu}}^\ast T_{\boldsymbol{\nu}}$. 
We further let $k \coloneqq {\rm dim}\left({\rm Ker}\big( T_{\boldsymbol{\nu}}^\ast T_{\boldsymbol{\nu}} \big) \right)^{\perp} \leq M_D$, i.e. the number of nonzero eigenvalues $\lambda_1 \geq \lambda_2 \geq \ldots \geq \lambda_k > 0$ of operator $T_{\boldsymbol{\nu}}^\ast T_{\boldsymbol{\nu}}$, set $\boldsymbol{\lambda}^{\boldsymbol{\nu}} \coloneqq \big( \lambda_1, \lambda_2, \ldots, \lambda_k, 0, \ldots, 0 \big) \in \mathbb{R}^{M_D}$ and let $\widetilde{\boldsymbol{\lambda}}^{\boldsymbol{\nu}} \coloneqq \big( \widetilde{\lambda}_1, \widetilde{\lambda}_2, \ldots, \widetilde{\lambda}_{M_D} \big) \in \mathbb{R}^{M_D}$ be the Monte Carlo estimator of $\boldsymbol{\lambda}^{\boldsymbol{\nu}}$ defined as the vector of non-negative eigenvalues of the random matrix $\boldsymbol{\Lambda}^{\boldsymbol{\nu}}_{\omega^1, \varepsilon, N}$ given by \eqref{eq:Lambda_tilde_nu}. 
Moreover, for a symmetric matrix $\mathbf{Q} \in \mathbb{R}^{n \times n}$, ${\rm gap}_i(\mathbf{Q})$, $i = \overline{1,n}$, denotes its $i-$th gap and is defined by \eqref{eq:gap}, whilst ${\rm \bf gap}(\mathbf{Q}) \coloneqq \big( {\rm gap}_1(\mathbf{Q}), {\rm gap}_2(\mathbf{Q}), \ldots, {\rm gap}_n(\mathbf{Q}) \big) \in \mathbb{R}^n$. 

\begin{Thm}[Eigenvalues approximation]
\label{thm:B}
\begin{equation*}
\lim_{{\rm diam}(\omega^1) \to 0} 
\lim_{\varepsilon \to 0} 
\lim_{N \to \infty} 
\left(
\Vert \boldsymbol{\lambda}^{\boldsymbol{\nu}} - \widetilde{\boldsymbol{\lambda}}^{\boldsymbol{\nu}} \Vert + 
\big\Vert
{\rm \bf gap} \big( \boldsymbol{\Lambda}^{\boldsymbol{\nu}}_{\omega^1, \varepsilon, N} \big) - 
{\rm \bf gap} \big( T_{\boldsymbol{\nu}}^\ast T_{\boldsymbol{\nu}} \big)
\big\Vert
\right) = 0~\mathbb{P}\mbox{-a.s. and in } L^1(\mathbb{P}).
\end{equation*}
\end{Thm}
\noindent More details and further explicit tail bounds for the convergence error are provided by \Cref{coro:tail_lambda}.

\begin{Thm}[Eigenvectors approximation]
\label{thm:C}
Let $\big\{ \widetilde{\mathbf{ u}}_1, \widetilde{\mathbf{ u}}_2, \ldots, \widetilde{\mathbf{ u}}_{M_D} \big\} \subset \mathbb{R}^{M_D}$ be an orthonormal basis that consists of eigenvectors of the random matrix $\boldsymbol{\Lambda}^{\boldsymbol{\nu}}_{\omega^1, \varepsilon, N}$ given by \eqref{eq:Lambda_tilde_nu} with the corresponding eigenvalues $\big( \widetilde{\lambda}_i \big)_{i = \overline{1,M_D}}$ defined above, and set
\begin{equation*}
\widetilde{u}_j(\mathbf{x}) \coloneqq  
\displaystyle \big( \widetilde{\lambda}_j \big)^{-1/2} 
\sum \limits_{i=1}^{M_D} \nu_i 
\big[ \widetilde{\mathbf{u}}_j \big]_{i} \rho_{\mathbf{x}_i^D, \omega^1, \varepsilon, N}(\mathbf{x}), \quad j = \overline{1, {\rm rank}\big( \boldsymbol{\Lambda}^{\boldsymbol{\nu}}_{\omega^1, \varepsilon, N} \big)}, 
\quad \mathbf{x} \in \Gamma_1.    
\end{equation*} 
Then there exists an orthonormal basis $\big\{ u_1, u_2, \ldots, u_k \big \} \subset {\rm Ker}(T_\nu^\ast T_\nu)^{\perp} \subset L^2(\Gamma_1)$ that consists of eigenfunctions of operator $T_{\boldsymbol{\nu}}^\ast T_{\boldsymbol{\nu}}$ such that, for all $j = \overline{1, k}$, where $k \coloneqq {\rm rank}\big( T_{\boldsymbol{\nu}}^\ast T_{\boldsymbol{\nu}} \big)$, the following relation hols
\begin{equation*}
\lim_{{\rm diam}(\omega^1) \to 0} 
\lim_{\varepsilon \to 0} 
\lim_{N \to \infty} 
\Big(
{\rm gap}_j\big( \boldsymbol{\Lambda}^{\boldsymbol{\nu}}_{\omega^1, \varepsilon, N} \big) \widetilde{\lambda}_j^{1/2} 
\sup_{\mathbf{x} \in \Gamma_1} 
\big \vert u_j(\mathbf{x}) - \widetilde{u}_j(\mathbf{x}) \big\vert 
\Big) = 0~\mathbb{P}\mbox{--a.s. and in } L^1(\mathbb{P}).
\end{equation*}
\end{Thm}
\noindent More details and further explicit tail bounds for the convergence error are given by \Cref{coro:tail_eigenvectors}.

The final main result is related to the TSVD of the following problem 
\begin{equation} \label{eq:op_eq}
T_{\boldsymbol{\nu}} u = \mathbf{b}^{\boldsymbol{\nu}} \coloneqq {\rm diag}(\boldsymbol{\nu}) \: \mathbf{b} 
\quad \textrm{or, equivalently,} \quad 
T_{\boldsymbol{\nu}}^\ast T_{\boldsymbol{\nu}} u = T_{\boldsymbol{\nu}}^\ast \mathbf{b}^{\boldsymbol{\nu}},
\end{equation}
for a given $\mathbf{b} \in \mathbb{R}^{M_D}$. 
To this end, consider $r \in \overline{1, {\rm rank}\big( T_{\boldsymbol{\nu}}^\ast T_{\boldsymbol{\nu}} \big)}$ and let $\mathbf{b}\in \mathbb{R}^{M_D}$ be given, $u^{(r)}$ be the {\it $r-$TSVD} solution to problem \eqref{eq:op_eq}, see \eqref{eq:u_r} for a rigorous definition, and $u^{(r)}_{\omega^1, \varepsilon, N}$ be its Monte Carlo estimator defined explicitly by \eqref{eq:u_omega_M1_r_MC} or \eqref{eq:u_omega_M1_r_MC A}. 

\begin{Thm}[Approximation of truncated SVD solutions]
\label{thm:D}
Consider $r \in \overline{1, {\rm rank}\big( T_{\boldsymbol{\nu}}^\ast T_{\boldsymbol{\nu}} \big)}$, whilst $u^{(r)}$ and $u^{(r)}_{\omega^1, \varepsilon, N}$ are given as above. 
If $\lambda_r > \lambda_{r+1}$, then
\begin{equation*}
\lim_{{\rm diam}(\omega^1) \to 0} 
\lim_{\varepsilon \to 0} 
\lim_{N \to \infty} 
\sup \limits_{\mathbf{x} \in \Gamma_1} 
\big\vert u^{(r)}(\mathbf{x})
- u^{(r)}_{\omega^1, \varepsilon, N}(\mathbf{x}) \big\vert = 0~\mathbb{P}\mbox{--a.s. and in } L^1(\mathbb{P}).
\end{equation*}
\end{Thm}
\noindent More details and further explicit tail bounds for the convergence error are provided by \Cref{coro:tail_solutions}.

It should be stressed out that  the four simplified statements listed in Theorems~\ref{thm:A}--\ref{thm:D} claim the convergence of the corresponding numerical schemes only. However, their corresponding detailed versions, namely Theorems~\ref{coro:rho_MC_est}--\ref{coro:tail_solutions}, actually provide one with explicit bounds of the errors which, in particular, show that the limits in Theorems~\ref{thm:A}--\ref{thm:D} are not interchangeable. 
These bounds may seem quite involved, but they are meant to be general enough in order to be applicable in various scenarios. 
To give a glance, a careful inspection of these bounds 
when the domain $D$ is of class $C^{1,\alpha}$ and the conductivity tensor $\mathbf{K}$ is constant, shows that there exists $\beta > 0$ depending on $\alpha$ such that the convergences in Theorems~\ref{thm:A}--\ref{thm:D} are of order
\begin{equation*}
\mathcal{O}\left( 
\big[ {\rm diam}(\omega^1) \big]^\beta + 
\mathcal{O}\left(
\varepsilon^{1/2} \big\vert \log(\varepsilon) \big\vert + 
\mathcal{O}\left( \dfrac{1}{\sqrt{N}} \right) 
\right)
\right),
\end{equation*}
with the mention that one has to ensure that truncation is made such that the $r-$th eigenvalue of $T_{\boldsymbol{\nu}}^\ast T_{\boldsymbol{\nu}}$ is strictly larger then its $(r+1)-$th eigenvalue in case of \Cref{thm:D}. 

Finally, for a merged/unified version of \cite{CiGrMaI} and the present manuscript, we refer the reader to \cite{CiGrMaExtended}.



\section{From spectral representations to Monte Carlo estimators} 
\label{section:formulation}

In this section we describe and  perform a first analysis of the quantities involved in the main results listed briefly in Theorems~\ref{thm:A}--\ref{thm:D}.
\subsection{Spectral analysis}
\label{ss:convergence_spectrum}

Recall that if we consider $M_D \geq 1$ internal points
$\big(\mathbf{x}^D_i\big)_{i = \overline{1,M_D}} \subset D$ and the vector $\boldsymbol{\nu} \coloneqq \big( \nu_i \big)_{i = \overline{1,M_D}} \in (0,1)^{M_D}$ such that 
$\displaystyle \sum_{i=1}^{M_D} \nu_i^2 = 1$, then the operator $T_{\mathbf{\nu}}$ is defined by
\begin{eqnarray} \label{eq:Th}
\left.
\begin{array}{l}
T_{\boldsymbol{\nu}} \colon \big(L^2(\Gamma_1), \Vert \cdot \Vert_{L^2(\Gamma_1)}\big) \longrightarrow \big(\mathbb{R}^{M_D}, \Vert \cdot \Vert\big), \\[4pt]
L^2(\Gamma_1)\ni u \longmapsto T_{\boldsymbol{\nu}} u
= \big(\nu_i \mu_{\mathbf{x}^D_i}(u) \big)_{i = \overline{1,M_D}}
= \displaystyle \bigg( \nu_i \int_{\Gamma_1} u \rho_{\mathbf{x}^D_i} \, \mathrm{d}\sigma \bigg)_{i = \overline{1,M_D}} \in \mathbb{R}^{M_D}.
\end{array}
\right.
\end{eqnarray}
Consequently, it is easy to see that the adjoint operator to operator $T_{\boldsymbol{\nu}}$, denoted by $T_{\boldsymbol{\nu}}^\ast$, is given by
\begin{eqnarray} \label{eq:Th*}
\left.
\begin{array}{l}
T_{\boldsymbol{\nu}}^\ast \colon \big(\mathbb{R}^{M_D}, \Vert \cdot \Vert \big) \longrightarrow \big(L^2(\Gamma_1), \Vert \cdot \Vert_{L^2(\Gamma_1)}\big), \\[4pt]
\mathbb{R}^{M_D}\ni \mathbf{v} = \big( v_i \big)_{i = \overline{1,M_D}} \longmapsto T_{\boldsymbol{\nu}}^\ast \mathbf{v} 
= \displaystyle \sum_{i=1}^{M_D} \nu_i v_i \rho_{\mathbf{x}^D_i} \in L^2(\Gamma_1).
\end{array}
\right.
\end{eqnarray}
Consider now the bounded linear operator
\begin{equation}\label{eq:B Th*Th}
T_{\boldsymbol{\nu}}^\ast T_{\boldsymbol{\nu}} \colon L^2(\Gamma_1) \longrightarrow L^2(\Gamma_1) 
\end{equation}
and introduce the following symmetric and non-negative definite matrices
\begin{equation}
\left.
\begin{aligned} 
\label{eq:Lambda}
&\boldsymbol{\Lambda} = \big( \Lambda_{ij} \big)_{i,j = \overline{1,M_D}} \in \mathbb{R}^{M_D\times M_D}, \qquad 
\Lambda_{ij} \coloneqq \displaystyle \int_{\Gamma_1} \rho_{\mathbf{x}_i^D} \rho_{\mathbf{x}_j^D} \, \mathrm{d}\sigma, \quad i,j = \overline{1,M_D},\\[4pt]
&\boldsymbol{\Lambda}^{\boldsymbol{\nu}} \coloneqq {\rm diag}(\boldsymbol{\nu}) \; \boldsymbol{\Lambda} \; {\rm diag}(\boldsymbol{\nu}) \in \mathbb{R}^{M_D\times M_D}.
\end{aligned}
\right.
\end{equation}

The next result shows the correspondence between the SVD of operator $T_{\boldsymbol{\nu}}^\ast T_{\boldsymbol{\nu}}$ and that of matrix $\boldsymbol{\Lambda}^{\boldsymbol{\nu}}$. 
\begin{prop}[C\^{i}mpean et al.~\cite{CiGrMaI}]
\label{coro:B-Lambda}
The operators $T_{\boldsymbol{\nu}}$ and $T_{\boldsymbol{\nu}}^\ast$, and the matrix $\boldsymbol{\Lambda}^{\boldsymbol{\nu}}$ given by \eqref{eq:Th}, \eqref{eq:Th*} and \eqref{eq:Lambda}, respectively, satisfy the following relation
\begin{equation*}
T_{\boldsymbol{\nu}} T_{\boldsymbol{\nu}}^\ast = \boldsymbol{\Lambda}^{\boldsymbol{\nu}}.
\end{equation*}
In particular, the following spectral correspondence between $T_{\boldsymbol{\nu}}^\ast T_{\boldsymbol{\nu}}$ and $\boldsymbol{\Lambda}^{\boldsymbol{\nu}}$ holds:
\begin{enumerate}[label={\rm (}\roman*{\rm )}]
\item $\boldsymbol{\Lambda}^{\boldsymbol{\nu}}$ and $T_{\boldsymbol{\nu}}^\ast T_{\boldsymbol{\nu}}$ have the same eigenvalues.
\item $u \in L^2(\Gamma_1)$ is an eigenfunction of operator $T_{\boldsymbol{\nu}}^\ast T_{\boldsymbol{\nu}}$ corresponding to the eigenvalue $\lambda \neq 0$ if and only if $\mathbf{u} \in \mathbb{R}^{M_D}$ is an eigenvector of matrix $\boldsymbol{\Lambda}^{\boldsymbol{\nu}}$ corresponding to the eigenvalue $\lambda \neq 0$ and the following relation holds
\begin{equation*}
u = T_{\boldsymbol{\nu}}^\ast \mathbf{u} 
= \displaystyle \sum \limits_{i=1}^{M_D} \nu_i [\mathbf{u}]_i \rho_{\mathbf{x}_i^D}.
\end{equation*}
\item Let $k \coloneqq {\rm dim}~{\rm Ker}~(\boldsymbol{\Lambda}^{\boldsymbol{\nu}})^{\perp} \leq M_D$ and $\big\{ \mathbf{u}_1, \mathbf{u}_2, \ldots, \mathbf{u}_k \big\} \subset \mathbb{R}^{M_D}$ be an orthonormal basis of ${\rm Ker}~(\boldsymbol{\Lambda}^{\boldsymbol{\nu}})^{\perp} \subset \mathbb{R}^{M_D}$ that consists of eigenvectors of $\boldsymbol{\Lambda}^{\boldsymbol{\nu}}$ with the corresponding eigenvalues $\lambda_1 \geq \lambda_2 \geq \ldots \geq \lambda_k > 0$. Define 
\begin{equation}\label{eq:basisB}
u_j 
\coloneqq \lambda_j^{-1/2} T_{\boldsymbol{\nu}}^\ast \mathbf{u}_j 
= \lambda_j^{-1/2} \displaystyle \sum \limits_{i=1}^{M_D} \nu_i \big[ \mathbf{u}_j \big]_{i} \rho_{\mathbf{x}_i^D}, 
\quad j = \overline{1, k},
\end{equation}
Then $\big\{ u_1, u_2, \ldots, u_k \big\}$ is an orthonormal basis of ${\rm Ker}~(T_{\boldsymbol{\nu}}^\ast T_{\boldsymbol{\nu}})^{\perp} \subset L^2(\Gamma_1)$ that consists of eigenfunctions of the operator $T_{\boldsymbol{\nu}}^\ast T_{\boldsymbol{\nu}}$. In particular, ${\rm dim}~{\rm Ker}~(\boldsymbol{\Lambda}^{\boldsymbol{\nu}})^{\perp} = {\rm dim}~{\rm Ker}~(T_{\boldsymbol{\nu}}^\ast T_{\boldsymbol{\nu}})^{\perp}$.
\end{enumerate}
\end{prop}

Given $\mathbf{b} \in \mathbb{R}^{M_D}$ and $\mathbf{b}^{\boldsymbol{\nu}} \coloneqq {\rm diag}(\boldsymbol{\nu}) \; \mathbf{b} \in \mathbb{R}^{M_D}$, we further aim at finding $u \in L^2(\Gamma_1)$ that satisfies the following operator equation
\begin{equation} \label{eq:TTh}
T_{\boldsymbol{\nu}} u = \mathbf{b}^{\boldsymbol{\nu}}
\end{equation}
by employing the SVD analysis.

Using the notations employed in \Cref{coro:B-Lambda}, we let $r \in \overline{1,k}$ be a prescribed truncation parameter and $P_r \colon L^2(\Gamma_1) \longrightarrow \mathrm{span}~\big\{v \in L^2(\Gamma_1) \, \big| \, T_{\boldsymbol{\nu}}^\ast T_{\boldsymbol{\nu}} v = \lambda_i v, \ i = \overline{1,r} \big\}$ be the projection operator onto the linear space generated by the eigenfunctions of $T_{\boldsymbol{\nu}}^\ast T_{\boldsymbol{\nu}}$ corresponding to the largest $r$ non-zero eigenvalues $\lambda_1 \geq \lambda_2 \geq \ldots \geq \lambda_r > 0$. The {\it $r-$TSVD solution} to the operator equation \eqref{eq:TTh} is defined as 
\begin{equation} \label{eq:u_r}
u^{(r)} \coloneqq P_r u \in \mathrm{span}~\big\{v \in L^2(\Gamma_1) \, \big| \, T_{\boldsymbol{\nu}}^\ast T_{\boldsymbol{\nu}} v = \lambda_i v, \ i = \overline{1,r} \big\} \subset L^2(\Gamma_1),
\end{equation} 
where $u \in  L^2(\Gamma_1)$ is any solution to the operator equation $T_{\boldsymbol{\nu}}^\ast T_{\boldsymbol{\nu}} u = T_{\boldsymbol{\nu}}^\ast \mathbf{b}^\nu$.

We further consider the finite-dimensional linear system and its corresponding $r-$TSVD solution, see, e.g. relation \eqref{eq:u1r} in \Cref{appendix},
\begin{equation} \label{eq:syst_uu_r}
\boldsymbol{\Lambda}^\nu \mathbf{u} = \mathbf{b}^\nu , \quad 
\mathbf{u}^{(r)} \coloneqq \left(\boldsymbol{\Lambda}^\nu\right)^\dag_r \mathbf{b}^\nu.
\end{equation} 
Then the following representation has been obtained in \cite{CiGrMaI}.

\begin{prop}[C\^{i}mpean et al.~\cite{CiGrMaI}]
\label{coro:representation_u(r)}
The $r-$TSVD solution to the operator equation \eqref{eq:TTh} given by \eqref{eq:u_r} and the $r-$TSVD solution to the linear system \eqref{eq:syst_uu_r} given by \eqref{eq:u_r} satisfy the following relation
\begin{equation}
\label{eq: truncated solutions relation}
u^{(r)} = \displaystyle \sum \limits_{j=1}^{M_D} \nu_j \big[ \mathbf{u}^{(r)} \big]_j \rho_{\mathbf{x}^D_j}.
\end{equation}
\end{prop}

\subsection{Discretisation of the inaccessible boundary} 
\label{ss:discret}

The main goal of this subsection is to analyse the error between the elliptic densities $\rho_{\mathbf{x}}$, as well as the operator $\Lambda^\nu$s and their approximations entailed by discretising $\Gamma_1$ with a finite number of points. 
In what follows, we consider $M_1$ boundary points $\mathbf{x}^1 = \big( \mathbf{x}^1_i \big)_{i = \overline{1,M_1}} \subset \Gamma_1$ and some generic corresponding boundary weights $\boldsymbol{\omega}^1 \coloneqq \big(\omega^1_i\big)_{i = \overline{1,M_1}}$, namely
\begin{equation}
\begin{aligned}\label{eq:boundary_weights}
&\omega^1_i \colon \Gamma_1 \longrightarrow [0,1] \mbox{ is measurable}, 
\quad \omega^1_i(\mathbf{x}_j) = \delta_{ij},~i, j = \overline{1,M_1}, 
\quad \sum_{i=1}^{M_1} \omega^1_i(\mathbf{x}) = 1,~\mathbf{x} \in \Gamma_1,\\[2pt]
&\sigma(\omega^1_i) \coloneqq \int_{\Gamma_1} \omega^1_i \, \mathrm{d}\sigma > 0,~i = \overline{1,M_1}.
\end{aligned}
\end{equation}
Further requirements will be made on these weights in \Cref{s:3}, namely that they may be extended to $\overline{D}_\varepsilon \coloneqq \big\{ \mathbf{x}\in \overline{D} : {\rm dist}(\mathbf{x}, \partial D) \leq \varepsilon \big\}$ by some $\boldsymbol{\omega}^{1, \varepsilon} \coloneqq \big( \omega^{1 ,\varepsilon}_i \big)_{i = \overline{1,M_1}}$ as
\begin{equation}
\begin{aligned}\label{eq:extended_weights_1}
&\omega_i^{1,\varepsilon} \colon \overline{D}_\varepsilon \longrightarrow [0,1], 
\quad \omega_i^{1, \varepsilon}\big|_{\Gamma_1} = \omega_i^1, \quad i=\overline{1,M_1}.
\end{aligned}
\end{equation}
Special attention is also paid to the so-called {\rm binary weights}, i.e. weights defined as above such that
\begin{equation}\label{eq:binary wieghts}
\exists~\varepsilon_0 > 0, \quad V_i \subset \Gamma_1,\quad V_i^\varepsilon \subset \overline{D}_\varepsilon, \quad 
i = \overline{1,M_1} \colon \quad 
\omega^1_i=1_{V_i}, \quad \omega^{1,\varepsilon}_i=1_{V_i^\varepsilon}, \quad 
i = \overline{1,M_1}, \quad 
\varepsilon \in (0, \varepsilon_0).   
\end{equation}

Next, the $\omega^\varepsilon-$interpolant on $\overline{D}_\varepsilon$ at points $\big( \mathbf{x}_i \big)_{i = \overline{1,n}} \subset \partial D$ of a given function $f \colon \partial D \longrightarrow \mathbb{R}$ is defined by
\begin{align}\label{eq:interpolant}
&f_{\omega^\varepsilon} \colon \overline{D}_\varepsilon \longrightarrow \mathbb{R}, 
\quad f_{\omega^\varepsilon}(\mathbf{x}) 
= \sum_{i=1}^n \omega_i^\varepsilon (\mathbf{x}) f(\mathbf{x}_i), \quad 
\mathbf{x} \in \overline{D}_\varepsilon.
\end{align}

\begin{ex}[{\bf Extrinsic Voronoi weights}]
\label{ex:A.2}
Suppose that $\big( \mathbf{x}_i \big)_{i = \overline{1,n}} \subset \partial D$ are given and consider the Voronoi diagrams on $\mathbb{R}^d$
\begin{equation}\label{eq:voronoi_ext}
\widetilde{V}_i \coloneqq \widetilde{V}(\mathbf{x}_i) 
= \big\{ \mathbf{x} \in \mathbb{R}^d \; \big| \; \Vert \mathbf{x} - \mathbf{x}_i \Vert  
= \displaystyle \inf_{j = \overline{1,n}} \Vert \mathbf{x} - \mathbf{x}_j \Vert \big\}, 
\quad i = \overline{1,n}.
\end{equation}
Since $\widetilde{V}_i$, $i = \overline{1,n}$, may be overlapping, for the sake of rigour, we consider the following partition
\begin{equation} \label{eq:disjoint voronoi_ext}
V_i \coloneqq \widetilde{V}_i \setminus \bigcup\limits_{j = 1}^{i-1} \widetilde{V}_j, 
\quad \bigcup\limits_{i = 1}^n V_i = \mathbb{R}^d.
\end{equation}
The induced (binary) weights on $\overline{D}_\varepsilon$ are given by
\begin{equation}\label{eq:voronoi weights_ext}
\omega_i^{\varepsilon} (\mathbf{x}) = \mathbf{1}_{V_i\cap \overline{D}_\varepsilon}(\mathbf{x}), 
\quad \mathbf{x}\in \overline{D}_{\varepsilon}, 
\quad i = \overline{1,n}.
\end{equation}
Hence the corresponding $\omega^\varepsilon-$interpolant on $\overline{D}_\varepsilon$ at points $(\mathbf{x}_i)_{i = \overline{1,n}} \subset \partial D$ of $f$, defined by \eqref{eq:interpolant}, is just the piecewise constant interpolant of $f$ given by
\begin{equation}\label{eq: f omega eps voronoi_ext}
f_{\omega^\varepsilon}(\mathbf{x}) = f(\mathbf{x_i}), 
\quad \mathbf{x} \in V_i^\varepsilon, 
\quad i = \overline{1,n}. 
\end{equation}
\end{ex}
\noindent For more typical examples of such weights, we refer the reader to \cite[Appendix]{CiGrMaI}.

We further define the linear operator
\begin{equation*}
\mathbf{A}^{\boldsymbol{\nu}}_{\omega^1} \colon \mathbb{R}^{M_1} \longrightarrow \mathbb{R}^{M_D}, \quad 
\Bar{\mathbf{u}} \in \mathbb{R}^{M_1} \longmapsto \mathbf{A}^{\boldsymbol{\nu}}_{\omega^1} \Bar{\mathbf{u}} = T_{\boldsymbol{\nu}} \left( \sum_{i=1}^{M_1} \Bar{u}_i \omega^1_i \right) \in \mathbb{R}^{M_D}.
\end{equation*}
More precisely, $\mathbf{A}^{\boldsymbol{\nu}}_{\omega^1}$ has the following matrix representation in the standard canonical basis
\begin{equation}\label{eq:AM matrix components}
\big[ \mathbf{A}^{\boldsymbol{\nu}}_{\omega^1} \big]_{ij} = 
\nu_i \mu_{\mathbf{x}^D_i}(\omega^1_j) \eqqcolon 
\nu_i \big[ \mathbf{A}_{\omega^1} \big]_{ij}, \quad 
i = \overline{1, M_D}, \quad 
j = \overline{1, M_1}.
\end{equation}
Next, define the approximation $\rho_{\mathbf{x}_i^D,\omega^1}$ of $\rho_{\mathbf{x}_i^D}$, $i = \overline{1, M_D}$, by
\begin{equation}\label{eq:rho omega1}
\rho_{\mathbf{x}_i^D,\omega^1}(\mathbf{x}) \coloneqq 
\displaystyle \sum \limits_{j=1}^{M_1} \sigma\big( \omega^1_j \big)^{-1} \, \mu_{\mathbf{x}_i^D}\big( \omega^1_j \big) \, \omega^1_j(\mathbf{x}), 
\quad \mathbf{x} \in \Gamma_1, 
\quad i = \overline{1, M_D}.
\end{equation}
Note that
\begin{align*}
\displaystyle \int_{\Gamma_1} \rho_{\mathbf{x}_i^D, \omega^1} \rho_{\mathbf{x}_j^D,\omega^1} \, \mathrm{d}\sigma 
= \big\{ \mathbf{A}_{\omega^1} \; 
\big[{\rm diag}\big( \sigma(\omega^1) \big) \big]^{-1} \; 
\mathbf{A}_{\omega^1}^{\sf T} \big\}_{ij}, 
\quad i, j = \overline{1,M_D},
\end{align*}
where ${\rm diag}\big( \sigma(\omega^1) \big) \coloneqq {\rm diag}\big( \sigma(\omega^1_1), \sigma(\omega^1_2), \ldots, \sigma(\omega^1_{M_1}) \big) \in \mathbb{R}^{M_1 \times M_1}$.
We further set
\begin{align}
\nonumber
&\boldsymbol{\Lambda}_{\omega^1} 
\coloneqq \mathbf{A}_{\omega^1} 
\big[ {\rm diag}\big( \sigma(\omega^1) \big) \big]^{-1} 
\mathbf{A}_{\omega^1}^{\sf T},\\[4pt]
\label{eq:Lambda_hat_nu}
&\boldsymbol{\Lambda}^{\boldsymbol{\nu}}_{\omega^1} 
\coloneqq {\rm diag}(\boldsymbol{\nu}) \; \boldsymbol{\Lambda}_{\omega^1} \; 
{\rm diag}(\boldsymbol{\nu}) 
= \mathbf{A}^{\boldsymbol{\nu}}_{\omega^1} \; 
\big[ {\rm diag}\big( \sigma(\omega^1) \big) \big]^{-1} \; 
\big( \mathbf{A}^{\boldsymbol{\nu}}_{\omega^1} \big)^{\sf T}.
\end{align}
As a discretised version of $\mathbf{u}^{(r)}$ given by \eqref{eq:u_r}, consider the $r-$TSVD solution $\mathbf{u}_{\omega^1}^{(r)}$ of the finite-dimensional linear system 
$\boldsymbol{\Lambda}^{\boldsymbol{\nu}}_{\omega_1} \mathbf{u} = \mathbf{b}^{\boldsymbol{\nu}}$
given by
\begin{equation} \label{eq:u_r Lambda omega1}
\mathbf{u}_{\omega^1}^{(r)} \coloneqq 
\big( \boldsymbol{\Lambda}^{\boldsymbol{\nu}}_{\omega^1} \big)^{\dag}_r \mathbf{b}^{\boldsymbol{\nu}} \in \mathbb{R}^{M_1}.
\end{equation}
For a given vector $\mathbf{b} \in \mathbb{R}^{M_D}$, we introduce the {\it $(\omega^1, r)-$TSVD approximate solution} $u_{\omega^1}^{(r)}\in L^2(\Gamma_1)$ of the operator equation \eqref{eq:TTh} by
\begin{equation} \label{eq:u_omega_M1_r}
u_{\omega^1}^{(r)}(\mathbf{x}) 
\coloneqq \sum \limits_{i=1}^{M_1} \nu_i \left[\mathbf{u}_{\omega^1}^{(r)} \right]_i \rho_{\mathbf{x}_i^D, \omega^1}(\mathbf{x}), 
\quad \mathbf{x} \in \Gamma_1.
\end{equation}
Moreover, one can check that relation \eqref{eq:u_omega_M1_r} may be written as
\begin{equation}\label{eq:u_omega_matrix_form}
u_{\omega^1}^{(r)}(\mathbf{x}) = 
\sum \limits_{i=1}^{M_1} 
\left( \mathbf{A}^{\boldsymbol{\nu}, \dag, r}_{\omega^1} \mathbf{b}^{\boldsymbol{\nu}} \right)_i \omega_i^1(\mathbf{x}), 
\quad \mathbf{x} \in \Gamma_1,
\end{equation}
where
\begin{equation*}\label{eq: A_r}
\mathbf{A}^{\boldsymbol{\nu}, \dag, r}_{\omega^1} 
\coloneqq 
\left[ {\rm diag}\big( \sigma(\omega^1) \big) \right]^{-1/2} 
\Big( \mathbf{A}^{\boldsymbol{\nu}}_{\omega^1}  
\left[ {\rm diag}\big( \sigma(\omega^1) \big) \right]^{-1/2} \Big)^{\dag}_r 
\in \mathbb{R}^{M_1 \times M_D}.
\end{equation*}

For $V \subset \partial D$ and $\mathbf{x} \in D$, consider the {\it oscillation} of the elliptic density $\rho_{\mathbf{x}}$ on $V$ as
\begin{align*} 
{\rm osc}_V\big( \rho_{\mathbf{x}} \big) 
\coloneqq \displaystyle \sup \limits_{\mathbf{y}, \mathbf{y}^\prime \in V} \vert \rho_{\mathbf{x}}(\mathbf{y}) - \rho_{\mathbf{x}}(\mathbf{y}^\prime) \vert.
\end{align*}
We further set
\begin{subequations} \label{eq:notation}
\begin{align}
\label{eq:notation.a}
&{\big( {\rm osc}_{M_1,\omega} \big)}_i \coloneqq 
\max \limits_{j = \overline{1,M_1}} {\rm osc}_{{\rm supp}(\omega^1_j)} \big( \rho_{\mathbf{x}_i^D} \big), 
\quad i = \overline{1,M_D}\\[4pt]
\label{eq:notation.b}
&\Upupsilon_{r}\big( \boldsymbol{\Lambda}^{\boldsymbol{\nu}} \big) \coloneqq 
\dfrac{4 \lambda_1\big( \boldsymbol{\Lambda}^{\boldsymbol{\nu}} \big) + 
4 \lambda_r\big( \boldsymbol{\Lambda}^{\boldsymbol{\nu}} \big) - 
2 \lambda_{r+1}\big( \boldsymbol{\Lambda}^{\boldsymbol{\nu}} \big)}{\lambda_r\big( \boldsymbol{\Lambda}^{\boldsymbol{\nu}} \big)^2 \, 
\big[ \lambda_r\big( \boldsymbol{\Lambda}^{\boldsymbol{\nu}} \big) - 
\lambda_{r+1}\big( \boldsymbol{\Lambda}^{\boldsymbol{\nu}} \big) \big]} \\[4pt]
\label{eq:notation.c}
&\delta^{\boldsymbol{\nu}}_{D, M_D, \omega^1} \coloneqq 
\left\langle 
{\rm diag}(\boldsymbol{\nu}) \big( \mu_{\mathbf{x}_i^D}(\Gamma_1) \big)^2_{i = \overline{1,M_D}}, \; 
{\rm diag}(\boldsymbol{\nu}) {\rm osc}^2_{M_1, \omega^1} 
\right\rangle^{1/2}\\[4pt]
\label{eq:notation.d}
&{\big( \overline{\rho}_{\Gamma_1} \big)}_i \coloneqq 
\displaystyle \sup \limits_{\mathbf{x} \in \Gamma_1} \rho_{\mathbf{x}_i^D}(\mathbf{x}), 
\quad i = \overline{1,M_D}.
\end{align}
\end{subequations}

\begin{rem} \label{rem:osc regularity} 
Let $\delta_M \coloneqq \displaystyle \max \limits_{k = \overline{1,M}}{\rm diam}\big( {\rm supp}(\omega_k) \big)$.
\begin{enumerate}[label={\rm (}\roman*{\rm )}]
\item If $D \subset \mathbb{R}^d$, $d \geq 3$, is $C^{1,\alpha}$, $\alpha\in (0,1)$, and the coefficients of the conductivity tensor $\mathbf{K}$ are bounded, strictly elliptic and $\alpha-$H\"older continuous, then
\begin{equation*}\label{eq:osc bound Holder}
{\big( {\rm osc}_{M_1,\omega} \big)}_i \leq 
C \, {\rm dist}\big( \mathbf{x}_i^D, \Gamma_1 \big)^{1 - d - \alpha} \, \delta_M^\alpha,
\end{equation*}
where $C = C(d, \alpha, D, \mathbf{K})$ according to \cite[Theorem 3.5~(ii)]{GrWi82}.
    
\item If $D \subset \mathbb{R}^d$, $d \geq 3$, is $C^{2, {\rm Dini}}$, and $\div (\mathbf{K} \nabla)$ can be represented as a second-order operator in non-divergence form $\displaystyle \sum \limits_{i,j=1}^d A_{ij} \partial_{ij}$ with the coefficients $\mathbf{A} = \big( A_{ij} \big)_{i, j = \overline{1,d}}$ strictly elliptic of Dini mean oscillation, then 
\begin{equation*}\label{eq:osc bound Dini}
{\big( {\rm osc}_{M_1,\omega} \big)}_i \leq 
C \, {\rm dist}\big( \mathbf{x}_i^D, \Gamma_1 \big)^{-d} \, \delta_M,
\end{equation*}
where $C = C(d, D, \mathbf{A})$ according to \cite[Theorem 1.9]{HwKi20}.

\item As already mentioned in \Cref{thm:A}, it would also be of interest to treat the case when the elliptic densities are discontinuous, e.g., their mean oscillation is merely vanishing. 
In this case, the errors computed in the sup-norm in Theorems~\ref{coro:rho_MC_est} and \ref{coro:tail_solutions} below should be measured in an integral sense, whilst the analysis could rely on VMO regularity results, similar to those obtained by \cite{JeKe82VMO,BoToZh23} and the references therein.
\end{enumerate}
\end{rem}

\begin{thm}\label{thm:main1}
The following estimates hold
\begin{align*}
&\displaystyle \sup \limits_{\mathbf{x} \in \Gamma_1} 
\Big\vert \rho_{\mathbf{x}_i^D}(\mathbf{x})
- \rho_{\mathbf{x}_i^D,\omega^1}(\mathbf{x}) \Big\vert \leq {\big( {\rm osc}_{M_1,\omega} \big)}_i, 
\quad i = \overline{1,M_D}, 
\quad \textrm{and} \quad 
\Vert \boldsymbol{\Lambda}^{\boldsymbol{\nu}} - \boldsymbol{\Lambda}^{\boldsymbol{\nu}}_{\omega^1} \Vert_{\rm F} \leq 
\delta^{\boldsymbol{\nu}}_{D, M_D, \omega^1}.
\end{align*}
Moreover, let $u^{(r)}$ be the $r-$TSVD solution to problem \eqref{eq:Th} given by \eqref{eq:u_r} and $u_{\omega^1}^{(r)}$ be the $(\omega^1, r)-$TSVD approximate solution to the same problem given by \eqref{eq:u_omega_M1_r}. 
If 
\begin{align*}
&\delta^{\boldsymbol{\nu}}_{D, M_D, \omega^1} \leq 
\dfrac{1}{2} \big[ 
\lambda_r\big( \boldsymbol{\Lambda}^{\boldsymbol{\nu}} \big) 
- \lambda_{r+1}\big( \boldsymbol{\Lambda}^{\boldsymbol{\nu}} \big) 
\big],
\end{align*}
then
\begin{align}
\label{eq:estimate-TSVD}
&\displaystyle \sup \limits_{\mathbf{x} \in \Gamma_1} 
\Big\vert u^{(r)}(\mathbf{x}) - 
u_{\omega^1}^{(r)}(\mathbf{x}) \Big\vert \leq 
\dfrac{\Vert \mathbf{b}^{\boldsymbol{\nu}} \Vert \; 
\Vert {\rm diag}(\boldsymbol{\nu}) \; {\rm osc}_{M_1, \omega} \Vert}{\lambda_{r}\big( \boldsymbol{\Lambda}^{\boldsymbol{\nu}} \big)} + 
\delta^{\boldsymbol{\nu}}_{D, M_D, \omega^1} \, \Upupsilon_{r}\big( \boldsymbol{\Lambda}^{\boldsymbol{\nu}} \big) \; 
\Vert \mathbf{b}^{\boldsymbol{\nu}} \Vert \; 
\Vert {\rm diag}(\boldsymbol{\nu}) \overline{\rho}_{\Gamma_1} \Vert. 
\end{align}
\end{thm}

\begin{proof}
We proceed in several steps as follows.

\medskip

\noindent \textbf{Step I.} Note that since $\displaystyle \sum \limits_{j=1}^{M_1} \omega_j^1(\mathbf{x}) = 1$, $\mathbf{x} \in \Gamma_1$, it follows that
\begin{align*}
\sup \limits_{\mathbf{x} \in \Gamma_1} 
&\bigg\vert \rho_{\mathbf{x}_i^D}(\mathbf{x})
- \sum \limits_{j=1}^{M_1} \sigma\big( \omega^1_j \big)^{-1} \mu_{\mathbf{x}_i^D}\big( \omega^1_j \big) \omega^1_j(\mathbf{x}) \bigg\vert 
\leq \max \limits_{j = \overline{1,M_1}} 
\sup \limits_{\mathbf{x} \in {\rm supp}\big( \omega^1_j \big)} 
\big\vert \rho_{\mathbf{x}_i^D}(\mathbf{x}) - \sigma\big( \omega^1_j \big)^{-1} \mu_{\mathbf{x}_i^D}\big( \omega^1_j \big) \big\vert\\[4pt]
&= \max \limits_{j = \overline{1,M_1}} 
\sup \limits_{\mathbf{x} \in {\rm supp}(\omega^1_j)} 
\sigma\big( \omega^1_j \big)^{-1} 
\left\vert \int_{\Gamma_1} \big[ \rho_{\mathbf{x}_i^D}(\mathbf{x}) - 
\rho_{\mathbf{x}_i^D}(\mathbf{y}) \big] \, \omega^1_j(\mathbf{y}) \, \mathrm{d}\mathbf{y} \right\vert\\[4pt]
&\leq \max \limits_{j = \overline{1,M_1}}{\rm osc}_{{\rm supp}(\omega^1_j)} \big( \rho_{\mathbf{x}_i^D} \big) 
= {\big( {\rm osc}_{M_1,\omega^1} \big)}_i, 
\quad i = \overline{1,M_D}.
\end{align*}

\medskip

\noindent \textbf{Step II.} The first part of the statement can be proved since the following relations hold
\begin{align*}
&\left\vert \big[ \boldsymbol{\Lambda} \big]_{ij} 
- \big[ \boldsymbol{\Lambda}_{\omega^1} \big]_{ij} \right\vert
= \bigg\vert 
\sigma\big( \rho_{\mathbf{x}_i^D} \rho_{\mathbf{x}_j^D} \big) 
- \sum \limits_{k=1}^{M_1} \sigma\big( \omega^1_k \big)^{-1} \, 
\mu_{\mathbf{x}_i^D}\big (\omega^1_k \big) \, 
\mu_{\mathbf{x}_j^D}\big( \omega^1_k \big) 
\bigg\vert\\[4pt]
&= \bigg\vert
\mu_{\mathbf{x}_i^D} \, 
\bigg( \rho_{\mathbf{x}_j^D}
- \sum \limits_{k=1}^{M_1} \sigma\big( \omega^1_k \big)^{-1} \, 
\mu_{\mathbf{x}_j^D}\big( \omega^1_k \big) \, \omega^1_k \bigg) 
\bigg\vert 
\leq \mu_{\mathbf{x}_i^D}(\Gamma_1) \, \big( {\rm osc}_{M_1,\omega^1} \big)_i, 
\quad i, j = \overline{1,M_D}.
\end{align*}
Therefore,
\begin{equation}\label{eq:norm leq delta}
\begin{split}
\Vert \boldsymbol{\Lambda}^{\boldsymbol{\nu}} - \boldsymbol{\Lambda}^{\boldsymbol{\nu}}_{\omega^1} \Vert_{\rm F}
& = \Vert {\rm diag}(\boldsymbol{\nu}) \, 
\big( \boldsymbol{\Lambda} - \boldsymbol{\Lambda}_{\omega^1} \big) \, 
{\rm diag}(\boldsymbol{\nu}) \Vert_{\rm F}
\leq \Bigg( 
\sum \limits_{i = 1}^{M_D} \sum \limits_{j = 1}^{M_D} 
\nu^2_i \, \nu^2_j \, \mu_{\mathbf{x}_i^D}(\Gamma_1)^2 \,
\big( {\rm osc}_{M_1, \omega^1} \big)_i^2 
\Bigg)^{1/2}\\[4pt]
&= \Bigg( 
\sum \limits_{i = 1}^{M_D} \nu^2_i \, \mu_{\mathbf{x}_i^D}(\Gamma_1)^2 \, 
\big( {\rm osc}_{M_1, \omega^1} \big)_i^2 
\Bigg)^{1/2} 
= \delta^\nu_{D, M_D, \omega^1}.
\end{split}
\end{equation}

\medskip

\noindent \textbf{Step III.} To prove relation \eqref{eq:estimate-TSVD}, we first note that
\begin{align*}
\Big\vert u^{(r)}(\mathbf{x})
- u^{(r)}_{\omega^1}(\mathbf{x}) \Big\vert \leq 
\Bigg\vert u^{(r)}(\mathbf{x}) - 
\sum \limits_{i=1}^{M_D} \nu_i 
\big[ \mathbf{u}^{(r)} \big]_i 
\rho_{\mathbf{x}^D_i,\omega^1}(\mathbf{x}) \Bigg\vert + 
\sum \limits_{i=1}^{M_D} \nu_i 
\left\vert \big[ \mathbf{u}^{(r)} \big]_i - 
\big[ \mathbf{u}^{(r)}_{\omega^1} \big]_i \right\vert \, \Big\vert \rho_{\mathbf{x}^D_i, \omega^1}(\mathbf{x}) \Big\vert 
\end{align*}
and by {\bf Step I}, on the one hand, it follows that
\begin{align*}
&\sup \limits_{\mathbf{x} \in \Gamma_1} 
\bigg\vert u^{(r)}(\mathbf{x}) 
- \sum \limits_{i=1}^{M_D} \nu_i \, 
\big[ \mathbf{u}^{(r)} \big]_i 
\sum \limits_{j=1}^{M_1} \sigma\big( \omega^1_j \big)^{-1} \, 
\mu_{\mathbf{x}_i^D}\big( \omega^1_j \big) \, 
\omega^1_j(\mathbf{x}) \bigg\vert\\[4pt]
&\leq \sum \limits_{i=1}^{M_D} \nu_i \, 
\big\vert \big[ \mathbf{u}^{(r)} ]_i \big\vert \, 
{\big( {\rm osc}_{M_1,\omega} \big)}_i 
\leq \big\Vert \mathbf{u}^{(r)} \big\Vert \, 
\Big\Vert {\rm diag}(\boldsymbol{\nu}) {\rm osc}_{M_1,\omega} \Big\Vert
\leq \dfrac{\big\Vert \mathbf{b}^{\boldsymbol{\nu}} \big\Vert \, \big\Vert {{\rm diag}(\boldsymbol{\nu}) \; \rm osc}_{M_1,\omega} \big\Vert}{\lambda_r\big( \boldsymbol{\Lambda}^{\boldsymbol{\nu}} \big)}.
\end{align*}
Also, in view of relation \eqref{eq:boundary_weights}, one obtains 
\begin{align*}
&\sum \limits_{i=1}^{M_D} \nu_i 
\left\vert \big[ \mathbf{u}^{(r)} \big]_i - 
\big[ \mathbf{u}^{(r)}_{\omega^1} \big]_i \right\vert \, 
\Bigg\vert \sum \limits_{j=1}^{M_1} 
\sigma\big( \omega^1_j \big)^{-1} \, 
\mu_{\mathbf{x}_i^D}\big( \omega^1_j \big) \, \omega^1_j(\mathbf{x}) \Bigg\vert \leq 
\sum \limits_{i=1}^{M_D} \nu_i 
\left\vert \big[ \mathbf{u}^{(r)} \big]_i - 
\big[ \mathbf{u}^{(r)}_{\omega^1} \big]_i \right\vert \, 
\bigg\vert \sup \limits_{\mathbf{x} \in \Gamma_1} \rho_{\mathbf{x}_i^D}(\mathbf{x}) \bigg\vert \\
& \leq \Big\Vert \mathbf{u}^{(r)} - \mathbf{u}_{\omega^1}^{(r)} \Big\Vert \, 
\Bigg\Vert 
{\rm diag}(\boldsymbol{\nu}) 
\Bigg( \sup \limits_{\mathbf{x} \in \Gamma_1} 
\rho_{\mathbf{x}_i^D}(\mathbf{x}) \Bigg)_{i = \overline{1,M_D}} 
\Bigg\Vert,
\end{align*}
and thus
\begin{align*}
\sup \limits_{\mathbf{x} \in \Gamma_1} 
\left\vert u^{(r)}(\mathbf{x})
- u^{(r)}_{\omega^1}(\mathbf{x}) \right\vert \leq 
\dfrac{\big\Vert \mathbf{b}^{\boldsymbol{\nu}} \big\Vert \, \big\Vert {{\rm diag}(\boldsymbol{\nu}) \; \rm osc}_{M_1,\omega} \big\Vert}{\lambda_r\big( \boldsymbol{\Lambda}^{\boldsymbol{\nu}} \big)} + 
\Big\Vert \mathbf{u}^{(r)} - \mathbf{u}_{\omega^1}^{(r)} \Big\Vert \, 
\Bigg\Vert 
{\rm diag}(\boldsymbol{\nu}) 
\Bigg( \sup \limits_{\mathbf{x} \in \Gamma_1} 
\rho_{\mathbf{x}_i^D}(\mathbf{x}) \Bigg)_{i = \overline{1,M_D}} 
\Bigg\Vert.
\end{align*}

\medskip

\noindent \textbf{Step IV.}
Next, we set $\mathbf{E} \coloneqq \boldsymbol{\Lambda}^{\boldsymbol{\nu}} - \boldsymbol{\Lambda}^{\boldsymbol{\nu}}_{\omega^1}$, $\mathbf{e} \coloneqq \mathbf{0}$, $k_r \coloneqq \lambda_1\big( \boldsymbol{\Lambda}^{\boldsymbol{\nu}} \big) / \lambda_r\big( \boldsymbol{\Lambda}^{\boldsymbol{\nu}} \big)$, $\eta_r \coloneqq \vert \mathbf{E} \vert / \lambda_r\big( \boldsymbol{\Lambda}^{\boldsymbol{\nu}} \big)$ and $\gamma_r \coloneqq \lambda_{r+1}\big( \boldsymbol{\Lambda}^{\boldsymbol{\nu}} \big) / \lambda_{r}\big( \boldsymbol{\Lambda}^{\boldsymbol{\nu}} \big)$, 
and since 
\[
\Vert \mathbf{E} \Vert \leq 
\dfrac{1}{2} \big[ \lambda_r\big( \boldsymbol{\Lambda}^{\boldsymbol{\nu}} \big) -\lambda_{r+1}\big( \boldsymbol{\Lambda}^{\boldsymbol{\nu}} \big) \big], 
\]
one can apply \Cref{thm:AHansen} to deduce the bound
\begin{equation*}
\begin{split}
\dfrac{\Big\Vert \mathbf{u}^{(r)} - \mathbf{u}_{\omega^1}^{(r)} \Big\Vert}
{\big\Vert \mathbf{u}^{(r)} \big\Vert}
&\leq \dfrac{k_r}{1 - \eta_r} 
\left( 
\dfrac{\Vert \mathbf{E} \Vert}{\lambda_1\big( \boldsymbol{\Lambda}^{\boldsymbol{\nu}} \big)} 
+ \dfrac{\eta_r}{1 - \eta_r - \gamma_r} 
\dfrac{\Big\Vert \boldsymbol{\Lambda}^{\boldsymbol{\nu}} \mathbf{u}^{(r)} - \mathbf{b}^{\boldsymbol{\nu}} \Big\Vert}
{\Vert \mathbf{b}^{\boldsymbol{\nu}} \Vert} 
\right) 
+ \dfrac{\eta_r}{1 - \eta_r - \gamma_r}\\[4pt]
&\leq \Vert \mathbf{E} \Vert
\left( 
\dfrac{2}{\lambda_r\big( \boldsymbol{\Lambda}^{\boldsymbol{\nu}} \big)} 
+ \dfrac{4\lambda_1\big( \boldsymbol{\Lambda}^{\boldsymbol{\nu}} \big) \, 
\Big\Vert \boldsymbol{\Lambda}^{\boldsymbol{\nu}} \mathbf{u}^{(r)} - \mathbf{b}^{\boldsymbol{\nu}} \Big\Vert}
{\lambda_r\big( \boldsymbol{\Lambda}^{\boldsymbol{\nu}} \big) \, 
\big[ \lambda_r\big( \boldsymbol{\Lambda}^{\boldsymbol{\nu}} \big) 
- \lambda_{r+1}\big( \boldsymbol{\Lambda}^{\boldsymbol{\nu}} \big) \big] \,
\Vert \mathbf{b}^{\boldsymbol{\nu}} \Vert} 
+ \dfrac{2}{\lambda_r\big( \boldsymbol{\Lambda}^{\boldsymbol{\nu}} \big)
- \lambda_{r+1}\big( \boldsymbol{\Lambda}^{\boldsymbol{\nu}} \big)} 
\right).
\end{split}
\end{equation*}
Since $\Big\Vert \boldsymbol{\Lambda}^{\boldsymbol{\nu}} \mathbf{u}^{(r)} - \mathbf{b}^{\boldsymbol{\nu}} \Big\Vert \leq \Vert \mathbf{b}^{\boldsymbol{\nu}} \Vert$ and $\big\Vert \mathbf{u}^{(r)} \big\Vert \leq \Vert \mathbf{b}^{\boldsymbol{\nu}} \Vert / \lambda_r\big( \boldsymbol{\Lambda}^{\boldsymbol{\nu}} \big)$, it follows that
\begin{align*}
\Big\Vert \mathbf{u}^{(r)} - \mathbf{u}_{\omega^1}^{(r)} \Big\Vert 
&\leq \Vert \mathbf{E} \Vert \, 
\dfrac{4 \lambda_1\big( \boldsymbol{\Lambda}^{\boldsymbol{\nu}} \big) 
+ 4 \lambda_r\big( \boldsymbol{\Lambda}^{\boldsymbol{\nu}} \big) 
- 2 \lambda_{r+1}\big( \boldsymbol{\Lambda}^{\boldsymbol{\nu}} \big)}{\lambda_r\big( \boldsymbol{\Lambda}^{\boldsymbol{\nu}} \big)^2 \, 
\big[ \lambda_r\big( \boldsymbol{\Lambda}^{\boldsymbol{\nu}} \big) 
- \lambda_{r+1}\big( \boldsymbol{\Lambda}^{\boldsymbol{\nu}} \big) \big]} \, 
\big\Vert \mathbf{b}^\nu \big\Vert
\end{align*}
and hence one finally obtains estimate \eqref{eq:estimate-TSVD}.
\end{proof}
\subsection{From the probabilistic representation of the elliptic measure to the Monte Carlo TSVD} 
\label{s:3}

In addition to hypotheses \ref{Hyp_D}--\ref{eq:Hq}, we assume that \ref{eq:H4} holds, i.e. the anisotropic material occupying the domain $D \subset \mathbb{R}^d$ is homogeneous. 

Let $(B_t)_{t \geq 0}$ be a standard $d-$dimensional Brownian motion, starting from zero, on a filtered probability space $\big(\Omega, \mathcal{F}, (\mathcal{F}_t)_{t \geq 0}, \mathbb{P}\big)$ satisfying the usual hypotheses, and set
\begin{equation} \label{eq:Y}
Y_t^{\mathbf{x}} \coloneqq \mathbf{K}^{1/2}B_t + \mathbf{x}, \quad t \geq 0.
\end{equation}
Further, consider the $\big(\mathcal{F}_t\big)-$stopping time
\begin{equation}\label{eq:stoppingtime}
\tau_{\partial{D}}^\mathbf{x} \coloneqq
\inf \big\{ t \geq 0 \: \big| \: Y_t^\mathbf{x} \in \partial{D} \big\}, \quad \mathbf{x} \in \mathbb{R}^d.
\end{equation}
It is then well-known (see e.g \cite[Chapter 9]{Oksendal} or \cite{ChenZhao95}) that
\begin{equation} \label{eq:probabilisticrepresentation}
\mu_{\mathbf{x}}(A) = \mathbb{E} \left[ 1_A\left( Y^\mathbf{x}_{\tau^{\mathbf{x}}_{\partial{D}}} \right) \right], 
\quad A \in \mathcal{B}(\partial D), 
\quad \mathbf{x} \in D.
\end{equation}
In particular, using the notations introduced in \Cref{ss:discret}, see \eqref{eq:boundary_weights} and \eqref{eq:AM matrix components}, it follows that
\begin{equation}\label{eq:mc}
\left[ \mathbf{A}_{\omega_1} \right]_{ij} = 
\mathbb{E} \left[ \omega^1_{j} \left( Y^{\mathbf{x}_i^D}_{\tau_{\partial{D}}} \right) \right], 
\quad i = \overline{1,M_D}, 
\quad j = \overline{1,M_1}.
\end{equation}
Due to the homogeneity of $\mathbf{K}$, instead of employing a plain Monte Carlo simulation to approximate the expectation in relation \eqref{eq:mc}, it is much more efficient to rely on the ellipsoid version of the well-known walk-on-spheres algorithm due to \cite{Muller}, which we briefly present in what follows. 
Without any loss of generality, we further assume that $1$ is the largest eigenvalue of matrix $\mathbf{K}$. 
Let $\mathbf{x} \in D$ and consider the walk-on-ellipsoids (WoE) Markov chain on $\overline{D}$ constructed as
\begin{equation} \label{eq:wos chain}
X^\mathbf{x}_0 = \mathbf{x} \in \overline{D}, 
\quad X^\mathbf{x}_{i+1} = X^\mathbf{x}_i 
+ {\rm dist}\big( X^{\mathbf{x}}_i, \partial D \big) \mathbf{K}^{1/2} \mathbf{U}_i, 
\quad i \geq 0,
\end{equation}
where $\big( \mathbf{U}_i \big)_{i \geq 0}$ are i.i.d. random variables on a probability space $\big( \widetilde{\Omega}, \widetilde{\mathcal{F}}, \widetilde{\mathbb{P}} \big)$, uniformly distributed on the $(d-1)-$dimensional unit sphere in $\mathbb{R}^d$ centered at the origin, denoted by $\mathbb{S}^{d-1}(\mathbf{0},1)$. 
Consider the first iteration when the process enters $\overline{D}_\varepsilon \coloneqq \big\{ \mathbf{x} \in \overline{D} \: \big| \: {\rm dist}(x,\partial D) \leq \varepsilon \big\}$,
\begin{equation}\label{eq:N epsilon x}
\mathrm{N}^\mathbf{x}_\varepsilon \coloneqq 
\min \big\{i \geq 0 \: \big| \: {\rm dist}\big( X^\mathbf{x}_i, \partial D \big) \leq \varepsilon \big\}, \quad \varepsilon > 0.
\end{equation}
Further, for $\mathbf{x} \in D$ and $\varepsilon > 0$, we consider
\begin{align}\label{eq:Y_chain}
\tau^\mathbf{x}_1 & 
\coloneqq \inf \big\{t \geq 0 \: \big| \: Y^\mathbf{x}_t \in \mathbf{x} + \mathbf{K}^{1/2} \mathbb{S}^{d-1}\big( \mathbf{0}, {\rm dist}(\mathbf{x}, \partial D)\big) \big\},\\
\tau^\mathbf{x}_n & 
\coloneqq \inf \big\{t \geq \tau^\mathbf{x}_{n-1} \: \big| \: Y^\mathbf{x}_t \in Y^\mathbf{x}_{\tau^\mathbf{x}_{n-1}} + \mathbf{K}^{1/2} \mathbb{S}^{d-1}\big(\mathbf{0}, {\rm dist}(Y^\mathbf{x}_{\tau^\mathbf{x}_{n-1}}, \partial D) \big) \big\}, \quad n \geq 2,\\ 
\widetilde{\mathrm{N}}^\mathbf{x}_\varepsilon  & \coloneqq \inf \big\{n \geq 0 \: \big| \: {\rm dist}(Y^\mathbf{x}_{\tau^\mathbf{x}_n}, \partial{D}) \leq \varepsilon \big\}, \quad
\tau^\mathbf{x}_\varepsilon   
\coloneqq \inf \big\{t \geq 0 \: \big| \: {\rm dist}(Y^\mathbf{x}_t, \partial{D}) \leq \varepsilon \big\},
\end{align}
Since the chains $\big( X^\mathbf{x}_n \big)_{n \geq 0}$ and $\big( Y^\mathbf{x}_{\tau^\mathbf{x}_n} \big)_{n \geq 0}$ have the same law, it follows, in particular, that 
\begin{equation}\label{eq:equal_distribution}
X^\mathbf{x}_{\mathrm{N}^\mathbf{x}_\varepsilon} 
\ \mbox{ and } \  Y^\mathbf{x}_{\tau^\mathbf{x}_{\widetilde{\mathrm{N}}^\mathbf{x}_\varepsilon}} \ \mbox{ are equal in distribution}.
\end{equation}

\subsubsection{Monte Carlo approximation of matrix $\mathbf{A}_{\omega^1}$}
Let $\big( X^\mathbf{x}_n \big)_{n\geq 0}$, $\mathbf{x} \in \overline{D}$, be the WoE constructed by relation \eqref{eq:wos chain}, $\varepsilon > 0$, and also recall relation \eqref{eq:mc}. 
As an approximation for $\mathbf{A}_{\omega^1}$ ,we define the matrix $\mathbf{A}_{\omega^1,\varepsilon}$ as
\begin{equation}\label{eq:WoE_matrix}
\left[ \mathbf{A}_{\omega^1, \varepsilon} \right]_{ij} \coloneqq \mathbb{E} \left[ \omega^{1, \varepsilon}_{j}\left( X^{\mathbf{x}_i^D}_{\mathrm{N}^{\mathbf{x}_i^D}_\varepsilon} \right) \right], 
\quad i = \overline{1,M_D}, 
\quad j = \overline{1,M_1},
\end{equation}
as well as its Monte Carlo estimator $\mathbf{A}_{\omega^1, \varepsilon, N}$ which is the random matrix given by
\begin{equation}\label{eq:WoE_matrix_MC}
\left[ \mathbf{A}_{\omega^1, \varepsilon, N} \right]_{ij} \coloneqq 
\dfrac{1}{N} 
\displaystyle \sum_{\ell = 1}^N 
\omega^{1, \varepsilon}_j\left( X^{\mathbf{x}_i^D, \ell}_{\mathrm{N}^{\mathbf{x}_i^D, \ell}_\varepsilon} \right),
\quad i = \overline{1,M_D}, 
\quad j = \overline{1,M_1},
\end{equation}
where for $\mathbf{x}\in D$, $X^{\mathbf{x}, 1}_{\mathrm{N}^{\mathbf{x}, 1}_\varepsilon}, X^{\mathbf{x}, 2}_{\mathrm{N}^{\mathbf{x}, 2}_\varepsilon}, \ldots, X^{\mathbf{x}, N}_{\mathrm{N}^{\mathbf{x}, N}_\varepsilon}$ are $N \geq 1$ i.i.d. copies of $X^\mathbf{x}_{\mathrm{N}^\mathbf{x}_\varepsilon}$.

Analogous to \eqref{eq:Lambda_hat_nu}, we also set
\begin{equation}
\label{eq:Lambda_hat_eps_nu}
\boldsymbol{\Lambda}_{\omega^1, \varepsilon} 
\coloneqq \mathbf{A}_{\omega^1, \varepsilon} \; 
\left[ {\rm diag}\big( \sigma(\omega^1) \big) \right]^{-1} \; 
\mathbf{A}_{\omega^1, \varepsilon}^{\sf T}, 
\quad \boldsymbol{\Lambda}^{\boldsymbol{\nu}}_{\omega^1, \varepsilon} 
\coloneqq {\rm diag}(\boldsymbol{\nu}) \; \boldsymbol{\Lambda}_{\omega^1, \varepsilon} \; 
{\rm diag}(\boldsymbol{\nu}) 
\end{equation}
and
\begin{equation}
\label{eq:Lambda_tilde_nu}
\boldsymbol{\Lambda}_{\omega^1, \varepsilon, N} 
\coloneqq \mathbf{A}_{\omega^1, \varepsilon, N} \; 
\left[ {\rm diag}\big( \sigma(\omega^1) \big) \right]^{-1} \; 
\mathbf{A}_{\omega^1, \varepsilon, N}^{\sf T}, 
\quad \boldsymbol{\Lambda}^{\boldsymbol{\nu}}_{\omega^1, \varepsilon, N} 
\coloneqq {\rm diag}(\boldsymbol{\nu}) \; 
\boldsymbol{\Lambda}_{\omega^1, \varepsilon, N} \; 
{\rm diag}(\boldsymbol{\nu}).
\end{equation}

\subsubsection{Monte Carlo solution to the inverse problem \eqref{eq:insidemeasurements}}
We further define
\begin{equation}\label{eq:rho omega1 MC}
\rho_{\mathbf{x}_i^D, \omega^1, \varepsilon, N}(\mathbf{x}) 
\coloneqq \sum \limits_{j=1}^{M_1} \sigma\big( \omega^1_j \big)^{-1} \left[ \mathbf{A}_{\omega^1, \varepsilon, N} \right]_{ij} \omega^1_j(\mathbf{x}), 
\quad \mathbf{x}\in \Gamma_1, 
\quad i = \overline{1,M_D}.
\end{equation}
We also consider the {\it $r-$TSVD random solution} $\mathbf{u}^{(r)}_{\omega^1, \varepsilon, N}$ defined as
\begin{equation} \label{eq:u_r Lambda omega1 MC}
\mathbf{u}_{\omega^1,\varepsilon,N}^{(r)} \coloneqq 
\big(\boldsymbol{\Lambda}^\nu_{\omega^1,\varepsilon,N}\big)^{\dag}_r \; \mathbf{b}^\nu \in \mathbb{R}^{M_D},
\end{equation}
which corresponds to the finite-dimensional linear system 
$\boldsymbol{\Lambda}^{\boldsymbol{\nu}}_{\omega^1, \varepsilon, N} \mathbf{u} = \mathbf{b}^\nu$. 
Finally, we define the Monte Carlo version of \eqref{eq:u_omega_M1_r} by introducing the {\it $(\omega^1, \varepsilon, N, r)-$Monte Carlo--TSVD approximate solution} $u^{(r)}_{\omega^1, \varepsilon, N}$ of the operator equation \eqref{eq:TTh} as
\begin{align}\label{eq:u_omega_M1_r_MC}
u_{\omega^1, \varepsilon, N}^{(r)}(\mathbf{x}) 
& \coloneqq \sum \limits_{i=1}^{M_D} \nu_i \left[ \mathbf{u}_{\omega^1, \varepsilon, N}^{(r)} \right]_i \rho_{\mathbf{x}_i^D, \omega^1, \varepsilon, N}(\mathbf{x}), 
\quad \mathbf{x} \in \Gamma_1,
\end{align}
which may also be recast as
\begin{equation}\label{eq:u_omega_M1_r_MC A}
u^{(r)}_{\omega^1, \varepsilon, N}(\mathbf{x}) 
\coloneqq \sum \limits_{i=1}^{M_1} \left[ \mathbf{A}^{\boldsymbol{\nu}, \dag, r}_{\omega^1, \varepsilon, N} \mathbf{b}^{\boldsymbol{\nu}} \right]_i \omega_i^1(\mathbf{x}), 
\quad \mathbf{x} \in \Gamma_1,
\end{equation}
where
\begin{equation}\label{eq: A_r_MC}
\mathbf{A}^{\nu, \dag, r}_{\omega^1,\varepsilon, N} 
\coloneqq \left[ {\rm diag}\big( \sigma(\omega^1) \big) \right]^{-1/2} 
\Big( {\rm diag}(\boldsymbol{\nu}) \; 
\mathbf{A}_{\omega^1, \varepsilon, N} \; 
\left[ {\rm diag}\big( \sigma(\omega^1) \big) \right]^{-1/2} \Big)^{\dag}_r 
\in \mathbb{R}^{M_1 \times M_D}.
\end{equation}






\section{Analysis of Monte Carlo errors: Tail estimates and convergence}\label{section:error}
Throughout this section, we assume that hypotheses \ref{Hyp_D}--\ref{eq:H4} are satisfied. 
For each bounded and measurable boundary data $f$ that is in fact prescribed on or extended to  $\overline{D}_\varepsilon$, we define
\begin{equation}\label{eq:u epsilon}
u^\varepsilon_f (\mathbf{x}) \coloneqq 
{\mathbb{E}} [ f (X^\mathbf{x}_{\mathrm{N}^\mathbf{x}_\varepsilon}) ], 
\quad \mathbf{x} \in D_\varepsilon,
\end{equation}
and this is going to be the main quantity of interest in the first part of the next subsection. 
Moreover, we assume that all points of $\partial D$ are regular, whilst $\pi$ denotes any measurable selection of the projection on $\partial D$, see, e.g., \cite[Appendix]{CiGrMaI} for more details.

\subsection{Error estimates for the direct problem}

Firstly, we clarify some global H\"older estimates for the solution $u$ to the Dirichlet problem \eqref{eq:dirichlet bvp}. 
Such properties have been studied intensively in the literature by various tools, yet our aim here is to derive some explicit estimates that could be used to quantify the (in)stability of the numerical schemes analysed herein.

\paragraph{Regularity estimates for the solution to the Dirichlet problem \eqref{eq:dirichlet bvp}.}
Let $\lambda_{\mathbf{K}} > 0$ be the smallest eigenvalue of $\mathbf{K}$ and recall that, without any loss of generality, the largest one is assumed to be $1$. 

Following \cite{Ai02} (see, e.g., \cite{Su99} for the planar case), for $\alpha \in (0,1]$, we consider the following \textit{local harmonic measure decay condition} for the domain $D$:
\begin{enumerate}[leftmargin=50pt]
\item[\mylabel{LHMD}{$\sf (LHMD_\alpha)$}] 
\begin{equation} \label{eq:LHMD}
\begin{split}
&\exists~M = M(\alpha) > 0,
~~\exists~r_0 = r_0(\alpha) >0,\\
&\mu_\mathbf{x}^{D\cap B(\mathbf{a},r)}\big( D \cap S(\mathbf{a},r) \big) \leq 
M \left( \dfrac{\vert \mathbf{x} - \mathbf{a} \vert}{r} \right)^\alpha, 
~~\forall~\mathbf{a} \in \partial D, 
~~\forall~r \in (0, r_0],
~~\forall~\mathbf{x} \in D \cap B(\mathbf{a},r),
\end{split}
\end{equation}
\end{enumerate}
where $\mu_\mathbf{x}^{D \cap B(\mathbf{a},r)}(\cdot)$ denotes the harmonic measure, i.e. corresponding to $\mathbf{K} = \mathbf{I}_d$, over the domain $D \cap B(\mathbf{a},r)$ with pole at $\mathbf{x} \in D \cap B(\mathbf{a},r)$.

\begin{rem}\label{rem:holder}

\begin{enumerate}[label={\rm (}{\it \roman*}{\rm )}]
\item Note that by \cite[Theorem 3]{Ai02}, if $\sf (LHMD_{\alpha^\prime})$ holds for the domain $D$, then, for all $\alpha \in (0, \alpha^\prime)$, if $f$ is $\alpha-$H\"older on $\partial D$ it follows that the solution $u$ to \eqref{eq:dirichlet bvp} is also $\alpha-$H\"older on $D$. 
This is, in general, not true for $\alpha = \alpha^\prime$, see \cite[Theorem 2]{Ai02} for the case $\alpha^\prime = 1$.

\item According to \cite[Remark 2 and Theorem 2]{Ai02}, if $D$ is a $C^1-$domain, then $\sf (LHMD_{\alpha})$ holds for all $0\alpha \in (0, 1)$ and it is not difficult to check that if $D$ satisfies the uniform exterior ball condition, then $\sf (LHMD_{1})$ holds with a constant $M$ that can be made explicit in terms of the radii of the exterior balls.
\end{enumerate}
\end{rem}

For $\alpha \in (0,1]$ and a subset $E \subset \mathbb{R}^d$, we define
\begin{equation*}
\vert f \vert_{\alpha, E} \coloneqq \sup \limits_{\substack{\mathbf{x} \neq \mathbf{y} \in E\\ \vert \mathbf{x} - \mathbf{y} \vert leq 1}} \dfrac{\vert f(\mathbf{x}) - f(\mathbf{y}) \vert}{\vert \mathbf{x} - \mathbf{y} \vert^\alpha}, \quad 
\Vert f \Vert_{\alpha,E} \coloneqq 
\vert f \vert_{\alpha, \partial D} + 
\sup \limits_{\mathbf{x} \in E} \vert f(\mathbf{x}) \vert.
\end{equation*}

\begin{prop}\label{prop:xlogx}
Assume that \ref{LHMD} holds for the domain $\mathbf{K}^{-1/2} D$ for some $\alpha \in (0, 1]$ and $M_{\mathbf{K}}$ instead of $M$. 
Then for any function $f$ that is $\alpha-$H\"older on $\partial D$ and every $\mathbf{x}, \mathbf{y}\in \overline{D}$, $\vert \mathbf{K}^{-1/2}(\mathbf{x} - \mathbf{y}) \vert < 1$, the solution $u_f$ of the Dirichlet problem \eqref{eq:dirichlet bvp} satisfies the following inequality
\begin{equation*} \label{eq:xlogx}
\vert u_f(\mathbf{x}) - u_f(\mathbf{y}) \vert \leq 
\Vert f \Vert_{\alpha, \partial D_{\mathbf{K}}} \psi(\mathbf{x} - \mathbf{y}),
\end{equation*}
where
$\psi(\mathbf{z}) \coloneqq 
\vert \mathbf{K}^{-1/2}\mathbf{z} \vert^\alpha 
10\left(3 + \dfrac{8M_\mathbf{K}}{\mathrm{e} \log{2}} \big\vert \log{\vert \mathbf{K}^{-1/2} \mathbf{z} \vert} \big\vert \right), 
\quad \mathbf{z} \in \mathbb{R}^d$.
\end{prop}

\begin{rem}
Clearly, if $\vert \mathbf{K}^{-1/2}(\mathbf{x} - \mathbf{y}) \vert \geq 1$, by the maximum principle, we immediately obtain 
\begin{equation*}
\vert u_f(\mathbf{x}) - u_f(\mathbf{y}) \vert \leq 
2 \Vert u \Vert_{\infty,D} \leq 
2 \Vert f \Vert_{\infty,\partial D} 
\vert \mathbf{K}^{-1/2}(\mathbf{x} - \mathbf{y}) \vert, 
\quad \forall~\mathbf{x}, \mathbf{y} \in \overline{D}.
\end{equation*}
\end{rem}

\begin{proof}[Proof of \Cref{prop:xlogx}]
We split the proof in two cases.

\medskip

\noindent{\bf Case~I: $\mathbf{K} = \mathbf{I}_d$.} 
By \cite[Theorem 3]{Ai02}, it follows that $u$ is $\beta-$H\"older for all $\beta \in (0, \alpha)$.
Inspecting the proofs of \cite[Proposition 4 and Proposition 2, (B)]{Ai02}, we obtain 
\begin{equation}\label{eq:holder_Ai02}
\vert u_f(\mathbf{x}) - u_f(\mathbf{y}) \vert \leq 
10 \left(3 + \dfrac{4M}{1 - 2^{\beta - \alpha}}\right) 
\Vert f \Vert_{\alpha, \partial D} 
\vert \mathbf{x} - \mathbf{y} \vert^{\beta}, 
\quad \forall~\vert \mathbf{x} - \mathbf{y} \vert \leq 1, \quad \forall~\beta \in (0, \alpha).
\end{equation}
Now, the idea is to minimise with respect to $\beta$ the right-hand side of \eqref{eq:holder_Ai02}, hence to minimise the function
\begin{align*}
\phi(\beta) 
&\coloneqq \dfrac{1}{1 - 2^{\beta - \alpha}} 
\vert \mathbf{x} - \mathbf{y} \vert^{\beta} 
\leq \dfrac{1}{(\alpha - \beta)\log{2}} 
\vert \mathbf{x} - \mathbf{y} \vert^{\beta} 
\eqqcolon \widetilde{\phi}(\beta), 
\quad \beta \in (0,\alpha).
\end{align*}
Minimising $\widetilde{\phi}$ for $\vert \mathbf{x} - \mathbf{y} \vert < 1$ by looking at its critical points, we deduce that the minimal value of $\widetilde{\phi}$ is attained at $\beta^\ast$ that satisfies
\begin{equation*}
\dfrac{1}{\alpha - \beta^\ast} = 
\big\vert \log{\vert \mathbf{x} - \mathbf{y} \vert} \big\vert 
\quad \mbox{equivalently to} \quad 
\beta^\ast = 
\alpha - \dfrac{1}{\big\vert \log{\vert \mathbf{x} - \mathbf{y} \vert} \big\vert}.
\end{equation*}
Hence $\phi \leq \dfrac{2 \big\vert \log{\vert \mathbf{x} - \mathbf{y} \vert} \big\vert}{\log{2}} 
\vert \mathbf{x} - \mathbf{y} \vert^{\alpha - \dfrac{1}{\big\vert \log{\vert \mathbf{x} - \mathbf{y} \vert} \big\vert}}$ and since $t^{- \dfrac{1}{\big\vert \log{ \vert t \vert} \big\vert}} = \mathrm{e}^{-1}$ for $t \in (0, 1)$, we obtain
\begin{align*}
\phi 
\leq \dfrac{2}{\mathrm{e} \log{2}} 
\vert \mathbf{x} - \mathbf{y} \vert^\alpha 
\big\vert \log{\vert \mathbf{x} - \mathbf{y} \vert} \big\vert.
\end{align*}
Consequently, if $\vert \mathbf{x} - \mathbf{y} \vert < 1$, then
\begin{align}\label{eq:holder_new}
\vert u_f(\mathbf{x}) - u_f(\mathbf{y}) \vert 
&\leq \Vert f \Vert_{\alpha, \partial D}
\vert \mathbf{x} - \mathbf{y} \vert^\alpha 
10\left(3 + \dfrac{8 M}{\mathrm{e} \log{2}} 
\big\vert \log{\vert \mathbf{x} - \mathbf{y} \vert} \big\vert \right).
\end{align}

\medskip

\noindent{\bf Case~II: $\mathbf{K} \in \mathbb{R}^{d \times d}$.} 
The first step is to use the standard change of variables, see, e.g., \cite{BCPZ}, namely
\begin{equation*}\label{eq:change_variables}
D_\mathbf{K} \coloneqq \mathbf{K}^{-1/2} D, 
\quad v(\mathbf{x}) = u_f\big( \mathbf{K}^{1/2} \mathbf{x} \big),~\mathbf{x} \in D_\mathbf{K} \quad 
g(\mathbf{x}) = f\big( \mathbf{K}^{1/2}\mathbf{x} \big),~\mathbf{x}\in \partial D_\mathbf{K},     
\end{equation*}
and note that $v$ solves the following problem
\[
\Delta v(\mathbf{x}) = 0,~\mathbf{x} \in D_\mathbf{K}, 
\quad v(\mathbf{x}) = g(\mathbf{x}),~\mathbf{x} \in \partial D_\mathbf{K}.
\]
Since the \ref{LHMD} condition \eqref{eq:LHMD} holds for $D_\mathbf{K}$ instead of $D$ and the constant $M_{\mathbf{K}}$ instead of $M$, we obtain
\begin{align*}
&\vert u_f(\mathbf{x}) - u_f(\mathbf{y}) \vert 
= \vert v\big( \mathbf{K}^{-1/2} \mathbf{x} \big) - 
v\big( \mathbf{K}^{-1/2} \mathbf{y}) \vert\\
&\leq \Vert f \Vert_{\alpha, \partial D_{\mathbf{K}}} 
\vert \mathbf{K}^{-1/2}(\mathbf{x} - \mathbf{y}) \vert^\alpha 
10\left(3 + \dfrac{8M_\mathbf{K}}{\mathrm{e} \log{2}} 
\big\vert \log{\vert \mathbf{K}^{-1/2 }(\mathbf{x} -  \mathbf{y}) \vert} \big\vert \right), 
\quad \mathbf{x}, \mathbf{y} \in D,
\end{align*}
where the inequality follows by \eqref{eq:holder_new}.
\end{proof}

\paragraph{Estimates for the $\varepsilon-$shell approximation error.}
In this paragraph we analyse the errors in the approximation of the solution $u_f$ to problem \eqref{eq:dirichlet bvp} by the probabilistic representation given by \eqref{eq:u epsilon}.
More precisely, the following key H\"older estimates are obtained.

\begin{thm} \label{thm:error estimates} 
Assume that condition \ref{LHMD} holds for the domain $\mathbf{K}^{-1/2}D$, some $\alpha \in (0, 1)$ and $M_{\mathbf{K}}$ instead of $M$. 
Let $\varepsilon > 0$, $u_f$ be the solution to the Dirichlet problem \eqref{eq:dirichlet bvp} with $f \in C^\alpha(\overline{D}_\varepsilon)$ and $u_g^\varepsilon$ be defined by \eqref{eq:u epsilon} with $g \in L^\infty(\overline{D}_\varepsilon)$. 
Then 
\begin{align*} 
&\vert u_f (\mathbf{x}) - u_g^\varepsilon (\mathbf{x}) \vert 
\leq \varepsilon^\alpha 
\lambda_{\mathbf{K}}^{-\alpha/2} 
\Vert f \Vert_{\alpha, \partial D} 
\left[ 30 + 
\dfrac{80 M_{\mathbf{K}}}{\mathrm{e} \log{2}} 
\left( \dfrac{1}{2} \vert \log{\lambda_{\mathbf{K}}} \vert + 
\vert \log{\varepsilon} \vert \right) \right] + 
\vert g - f \circ \pi \vert_{\infty, \overline{D}_\varepsilon}, 
\quad \mathbf{x} \in D.
\end{align*}
In particular,
\begin{align}\label{eq:main_estimate_f}
&\vert u_f (\mathbf{x}) - u_f^\varepsilon (\mathbf{x}) \vert 
\leq \varepsilon^\alpha 
\lambda_{\mathbf{K}}^{-\alpha/2} 
\Vert f \Vert_{\alpha, D_\varepsilon} 
\left[ 31 + 
\dfrac{80 M_{\mathbf{K}}}{\mathrm{e} \log{2}} 
\left( \dfrac{1}{2} \vert \log{\lambda_{\mathbf{K}}} \vert + 
\vert \log{\varepsilon} \vert \right) \right], 
\quad \mathbf{x} \in D.
\end{align}
\end{thm}
\begin{proof}
The key idea is to employ the fact that the process
$u_f\left(Y_{t\wedge \tau^{\mathbf{x}}_D}^{\mathbf{x}}\right)$, $t \geq 0$, is a continuous and uniformly bounded martingale for every $\mathbf{x} \in D$, where $Y$ is the process given by \eqref{eq:Y}. 
Hence, using the notations introduced in \eqref{eq:Y_chain}, by Doob's stopping theorem, the mean value property is given by
\begin{equation*}
u_f(\mathbf{x}) = 
\mathbb{E}\left\{ u_f\left( Y_{\tau^{\mathbf{x}}_{\widetilde{N}_\varepsilon^\mathbf{x}}}^{\mathbf{x}} \right) \right\}, 
\quad \mathbf{x} \in D, 
\quad \varepsilon > 0.
\end{equation*}
Let $\pi$ be a measurable selection of the projection onto the boundary, see \cite[Appendix]{CiGrMaI} for details. Then using \eqref{eq:equal_distribution} and \Cref{prop:xlogx}, the following estimate is obtained
\begin{align*}
&\vert u_f(\mathbf{x}) - u_g^\varepsilon(\mathbf{x}) \vert 
= \left\vert 
\mathbb{E}\left\{ 
u_f\left( Y_{\tau^{\mathbf{x}}_{\widetilde{N}_\varepsilon^\mathbf{x}}}^{\mathbf{x}} \right) - 
g\left( Y_{\tau^{\mathbf{x}}_{\widetilde{N}_\varepsilon^\mathbf{x}}}^{\mathbf{x}} \right) \right\} 
\right\vert\\
&\leq \mathbb{E}\left\{
\left\vert u_f\left( Y_{\tau^{\mathbf{x}}_{\widetilde{N}_\varepsilon^\mathbf{x}}}^{\mathbf{x}} \right) - 
u_f\left( \pi\left( Y_{\tau^{\mathbf{x}}_{\widetilde{N}_\varepsilon^\mathbf{x}}}^{\mathbf{x}} \right) \right) \right\vert 
\right\} + 
\mathbb{E}\left\{ 
\left\vert g\left( Y_{\tau^{\mathbf{x}}_{\widetilde{N}_\varepsilon^\mathbf{x}}}^{\mathbf{x}} \right) - 
f\left( \pi\left( Y_{\tau^{\mathbf{x}}_{\widetilde{N}_\varepsilon^\mathbf{x}}}^{\mathbf{x}} \right) \right) \right\vert 
\right\}\\
&\leq \varepsilon^\alpha 
\lambda_{\mathbf{K}}^{-\alpha/2} 
\Vert f \Vert_{\alpha, \partial D} 
\left[ 30 + 
\dfrac{80 M_{\mathbf{K}}}{\mathrm{e}\log{2}} 
\left( \dfrac{1}{2} \vert \log{\lambda_{\mathbf{K}}} \vert + 
\vert \log{\varepsilon} \vert \right) 
\right] + 
\vert g - f \circ \pi \vert_{\infty, \overline{D}_\varepsilon}, 
\quad \mathbf{x} \in D, 
\quad \varepsilon > 0.
\end{align*}
In particular, this proves \eqref{eq:main_estimate_f} since $\vert f - f \circ \pi \vert_{\infty, \overline{D}_\varepsilon} \leq \varepsilon^\alpha \vert f \vert_{\alpha, D_\varepsilon}$ and $\lambda_{\mathbf{K}} \leq 1$. 
\end{proof}

A direct consequence of \Cref{thm:error estimates} provides one with an estimate for the error between the matrices $\mathbf{A}_{\omega^1}$ and $\mathbf{A}_{\omega^1, \varepsilon}$ given by relations \eqref{eq:AM matrix components} and \eqref{eq:WoE_matrix}, respectively, when the weight functions $\omega^{1,\varepsilon}_j$ are H\"older continuous, e.g., for linear weights or the inverse distance weights defined in \cite[Appendix]{CiGrMaI}. 
However, it should be emphasised that the regularity of a weight function - more precisely, its H\"older constant - will naturally blow up as its support on $\partial D$ becomes smaller and smaller and this has to be carefully compensated by choosing $\varepsilon$ to be sufficiently small. Nevertheless, note that as discussed in \cite{CiGrMaI}, $\varepsilon$ can be taken extremely small due to the fact that the number of steps in the WoE algorithm typically increases only logarithmically in $1/\varepsilon$.

\begin{coro}\label{coro: error estimates}
Assume that condition \ref{LHMD} holds for the domain $\mathbf{K}^{-1/2}D$, some $\alpha \in (0, 1]$ and $M_{\mathbf{K}}$ instead of $M$. 
Let $\mathbf{A}_{\omega^1}$ and $\mathbf{A}_{\omega^1,\varepsilon}$ be the matrices given by relations \eqref{eq:AM matrix components} and \eqref{eq:WoE_matrix}, respectively. 
Then 
\begin{equation*}
\left\vert 
\left[ \mathbf{A}_{\omega^1} \right]_{ij} - 
\left[ \mathbf{A}_{\omega^1,\varepsilon} \right]_{ij} \right\vert \leq 
\varepsilon^\alpha 
\lambda_{\mathbf{K}}^{-\alpha/2} 
\Vert \omega^{1, \varepsilon}_j \Vert_{\alpha, D_\varepsilon} 
\left[ 31 + 
\dfrac{80 M_{\mathbf{K}}}{\mathrm{e}\log{2}} 
\left( \dfrac{1}{2} \vert \log{\lambda_{\mathbf{K}}} \vert + 
\vert \log{\varepsilon} \vert \right) \right], 
\quad i = \overline{1,M_D}, 
\quad j = \overline{1,M_1}.
\end{equation*}
\end{coro}

\begin{rem}
It is important to mention that, in principle, $\Vert \omega^{1, \varepsilon}_j \Vert_{\alpha, \partial D}$ does not depend on $\varepsilon$. 
For example, this is the case if $\omega^{1,\varepsilon}_j$ is chosen to be the extension of an $\alpha-$H\"older continuous weight function $\omega_j^1 \colon \partial D \longrightarrow [0,1]$ by setting $\omega^{1,\varepsilon}_j = \omega_j^1 \circ \pi$ on some boundary shell $\overline{D}_\varepsilon$, where the projection on the boundary $\pi$ is assumed to be H\"older continuous. 
Another situation when the independence of $\varepsilon$ is easily seen, is if $\omega^{1,\varepsilon}_j$ are the inverse distance weights described in \cite[Appendix]{CiGrMaI}.
\end{rem}

We emphasise that \Cref{coro: error estimates} is useful only when the weights $\omega_j^{1,\varepsilon}$ are H\"older continuous.
Although the choice of these weights is at our disposal, \Cref{coro: error estimates} does not apply in the important case of, e.g., the extrinsic Voronoi weights as defined in \Cref{ex:A.2} since in this situation $\omega_j^{1, \varepsilon}$ is a characteristic function and hence discontinuous. 
We note that the extrinsic Voronoi weights are the ones used for all the simulations performed in \cite{CiGrMaI}. 
To cover this important case, in the remaining part of the paragraph, we thoroughly analyse the general case of weights given by characteristic functions. 
This is done by regularisation and, to this end, the following class of {\it regular} subsets in $\partial D$ needs to be introduced which are, in fact, subsets of $\mathcal{B}(\mathbb{R}^d)$, whose $(d-1)-$dimensional upper Minkowski content can be bounded in terms of their diameter.

\begin{defi}
For a set $B \subset \partial D$ and $r > 0$, define 
\begin{equation}\label{eq:B^c_r}
B^c_r \coloneqq 
\big\{ \mathbf{x} \in \partial D \setminus B \: | \: {\rm dist}(\mathbf{x}, B) \leq r \big\} \cup 
\big\{ \mathbf{x} \in B \: | \: {\rm dist}(\mathbf{x}, \partial D \setminus B) \leq r \big\}.
\end{equation}  
A subset $A \in \mathcal{B}(\partial D)$ is said to have a {\rm bounded $(d-1)-$dimensional upper Minkowski content} if \begin{equation}\label{eq:regular_A}
\exists~C_A > 0: \quad 
\sigma(A^c_r) \leq C_A r, 
\quad r \geq 0,
\end{equation}
\end{defi}

Furthermore, consider the following bound for the elliptic measure with respect to the surface measure:

\medskip

\noindent{$\mathbf{\left(H_{\mu/\sigma}\right)}$} 
$\exists~q \in (0, 1]$, $\forall~\mathbf{x} \in D$, $\forall~A \in \mathcal{B}(\partial D)$, $\exists~C\big( \mathbf{x}, {\rm dist}(\mathbf{x},A) \big) \coloneqq C\big( \mathbf{x}, {\rm dist}(\mathbf{x}, A), d, \mathbf{K}, D \big)$ non-decreasing w.r.t. ${\rm dist}(\mathbf{x},A)$:
\begin{equation}\label{eq:mu-regularity}
\mu_{\mathbf{x}}(A) \leq C\big( \mathbf{x}, {\rm dist}(\mathbf{x}, A) \big) \; \sigma(A)^q. 
\end{equation}

\begin{rem}\label{rem:regular}
If $D \subset \mathbb{R}^d$, $d \geq 3$, is $C^{1,\alpha}$, it follows from \cite[Theorem 3.5]{GrWi82} that the elliptic measure $\mu_\mathbf{x}$ has a density $\rho_\mathbf{x}$ w.r.t. $\sigma$ on $\partial D$ and, moreover,
\begin{equation*}
\exists~c(d, \mathbf{K}, D) > 0: \quad 
\rho_\mathbf{x}(\mathbf{y}) \leq 
c(d,\mathbf{K}, D) \, 
{\rm dist}(\mathbf{x}, \partial D) \, 
\vert \mathbf{x} - \mathbf{y} \vert^{-d}, 
\quad \mathbf{x} \in D, 
\quad \mathbf{y} \in \partial D.
\end{equation*}
Therefore,
\begin{equation*}
\mu_\mathbf{x}(A) 
\leq c(d, \mathbf{K}, D) \, 
{\rm dist}(\mathbf{x}, \partial D) \, 
{\rm dist}(\mathbf{x}, A)^{-d} \, \sigma(A), 
\quad \mathbf{x} \in D,
\end{equation*}
so that condition \eqref{eq:mu-regularity} holds with $C\big( \mathbf{x}, {\rm dist}(\mathbf{x}, A) \big) \coloneqq c(d, \mathbf{K},D) \, {\rm dist}(\mathbf{x}, \partial D) \, {\rm dist}(\mathbf{x}, A)^{-d}$ and $q = 1$.
\end{rem}

\begin{coro}\label{coro:mu - mu epsilon}
Assume condition \ref{LHMD} holds for the domain $\mathbf{K}^{-1/2}D$, some $\alpha \in (0, 1]$ and $M_{\mathbf{K}}$ instead of $M$. 
Then, for any $\mathbf{x} \in D$, $A \in \mathcal{B}(\partial D)$, $\varepsilon > 0$ and $r>0$, the following estimate holds
\begin{equation}\label{eq:error_A}
\left\vert \mu_\mathbf{x} (A) - \mu_\mathbf{x}^\varepsilon (A) \right\vert \leq 
\mu_\mathbf{x}(A^c_r) + 
\varepsilon \lambda_{\mathbf{K}}^{-1/2}(1 + r^{-1}) 
\left[32 + 
\dfrac{80 M_{\mathbf{K}}}{\mathrm{e} \log{2}} 
\left( \dfrac{1}{2} \vert \log{\lambda_{\mathbf{K}}} \vert + 
\vert \log{\varepsilon} \vert \right) \right].
\nonumber
\end{equation}
Furthermore, assume that condition $\mathbf{\left(H_{\mu/\sigma}\right)}$ holds and $A$ has a bounded $(d-1)-$dimensional upper Minkowski content, and set
\begin{align*}\label{eq:tilde}
&\widetilde{\varepsilon} \coloneqq 
\varepsilon \lambda_{\mathbf{K}}^{-1/2}
\left[ 31 + 
\dfrac{80 M_{\mathbf{K}}}{\mathrm{e} \log{2}} 
\left( \dfrac{1}{2} \vert \log{\lambda_{\mathbf{K}}} \vert + 
\vert \log{\varepsilon} \vert \right) \right],\\
&\widetilde{C}(\mathbf{x}, d, \mathbf{K}, D, A, q) \coloneqq 
2 \, C\big( \mathbf{x}, {\rm dist}(\mathbf{x}, A)/2 \big)^{1/2} \, C_A^{q/2},
\end{align*}
where $C_{A}$ is given by \eqref{eq:regular_A}. 
Then, for every $\varepsilon > 0$ sufficiently small such that
\begin{equation}\label{eq:r_ast}
r^\ast \leq {\rm dist}(\mathbf{x}, A)/2,  
\end{equation}
where
\begin{equation*}
r^\ast \coloneqq 
\begin{cases}
\widetilde{\varepsilon}^{1/(2q)} 
C\big( \mathbf{x}, {\rm dist}(\mathbf{x}, A) \big)^{-1/(2q)} 
C_A^{-1/2} & 
\textrm{ if \ } 
\widetilde{\varepsilon}^{1/(2q)} C\big( \mathbf{x}, {\rm dist}(\mathbf{x}, A) \big)^{-1/(2q)} C_A^{-1/2} \leq 1\\
\widetilde{\varepsilon}^{1/2} 
C\big( \mathbf{x}, {\rm dist}(\mathbf{x}, A) \big)^{-1/2} 
C_A^{-q/2} & \textrm{ otherwise,} 
\end{cases}
\end{equation*}
the following estimate holds
\begin{equation}\label{eq:explicit_A}
\left\vert \mu_\mathbf{x} (A) - \mu_\mathbf{x}^\varepsilon (A) \right\vert \leq 
\widetilde{\varepsilon} + 
\widetilde{\varepsilon}^{1/2} 
\widetilde{C}(\mathbf{x}, d, \mathbf{K}, D, A, q) 
\in \mathcal{O}(\varepsilon^{1/2} \vert \log\varepsilon \vert).
\end{equation}
\end{coro}
\begin{proof} 
For every $r > 0$, define the following functions 
\begin{align*}
&\varphi_r \colon \partial D \longrightarrow [0,1], \qquad 
\varphi_r(\mathbf{y}) = \big(r^{-1} \, {\rm dist}(\mathbf{y}, \partial D \setminus A) \big) \wedge 1,\\
&\psi_r \colon \partial D \longrightarrow [0,1], \qquad 
\psi_r(\mathbf{y}) = 1 - \big( r^{-1} \, {\rm dist}(\mathbf{y}, A) \big) \wedge 1.
\end{align*}
For simplicity, whenever required, we tacitly assume $\varphi_r, \psi_r \colon \overline{D} \longrightarrow [0,1]$ defined by the same formulae.

Note that $\varphi_r \leq 1_{A} \leq \psi_r$ and, in particular,
\begin{align}\label{eq:max bounds}
\left\vert \mu_\mathbf{x} (A) - \mu_\mathbf{x}^\varepsilon (A) \right\vert \leq 
\displaystyle \max \big( 
\vert u_{\psi_r}(\mathbf{x}) - u^\varepsilon_{\varphi_r \circ \pi}(\mathbf{x}) \vert, 
\vert u_{\varphi_r}(\mathbf{x}) - u^\varepsilon_{\psi_r \circ \pi}(\mathbf{x}) \vert
\big), \quad r > 0.
\end{align}
The first term of the max function in the right-hand side of \eqref{eq:max bounds} satisfies the following inequality
\begin{equation} \label{eq:triangle ineq}
\vert u_{\psi_r}(\mathbf{x}) - u^\varepsilon_{\varphi_r \circ \pi}(\mathbf{x}) \vert \leq 
\vert u_{\psi_r}(\mathbf{x}) - u_{\varphi_r}(\mathbf{x}) \vert + 
\vert u_{\varphi_r}(\mathbf{x}) - u^\varepsilon_{\varphi_r}(\mathbf{x}) \vert + 
\vert u^\varepsilon_{\varphi_r \circ \pi}(\mathbf{x}) - u^\varepsilon_{\varphi_r}(\mathbf{x}) \vert.
\end{equation}
Since $\varphi_r$ has the Lipschitz constant $r^{-1}$, the third term in the right-hand side of \eqref{eq:triangle ineq} can be estimated as
\begin{equation*}
\vert u^\varepsilon_{\varphi_r \circ \pi}(\mathbf{x}) - u^\varepsilon_{\varphi_r}(\mathbf{x}) \vert 
\leq \varepsilon/r.
\end{equation*}
Since $\varphi_r$ has the Lipschitz constant $r^{-1}$ and using \Cref{thm:error estimates}, the second term in the right-hand side of inequality \eqref{eq:triangle ineq} can be estimated as
\begin{align*}
\vert u_{\varphi_r}(\mathbf{x}) - u^\varepsilon_{\varphi_r}(\mathbf{x}) \vert \leq 
\varepsilon \lambda_{\mathbf{K}}^{-1/2} 
\big( 1 + r^{-1} \big)
\left[ 31 + 
\dfrac{80 M_{\mathbf{K}}}{\mathrm{e} \log{2}} 
\left( \dfrac{1}{2} \vert \log{\lambda_{\mathbf{K}}} \vert +
\vert \log{\varepsilon} \vert \right) 
\right].
\end{align*}
As far as the first term in the right-hand side of inequality \eqref{eq:triangle ineq} is concerned, we set 
\begin{equation*}\label{eq:B}
B_r \coloneqq \big\{ \mathbf{y} \in A \: | \: {\rm dist}(\mathbf{y}, \partial D \setminus A) > r \big\}.
\end{equation*}
Then, 
\begin{align*}
&\vert u_{\psi_r}(\mathbf{x}) - u_{\varphi_r}(\mathbf{x}) \vert = 
\bigg| 
\int_{\partial D \setminus A} 
\big[ \psi_r(\mathbf{y}) - \varphi_r(\mathbf{y}) \big] 
\mu_\mathbf{x}(\mathrm{d} \mathbf{y}) + 
\int_{A} 
\big[ \psi_r(\mathbf{y}) - \varphi_r(\mathbf{y}) \big] 
\mu_\mathbf{x}(\mathrm{d} \mathbf{y}) 
\bigg| \\
&= \bigg| 
\int_{\partial D \setminus A} 
\psi_r(\mathbf{y}) \mu_\mathbf{x}(\mathrm{d} \mathbf{y}) + 
\int_{A \setminus B_r} 
\big[ \psi_r(\mathbf{y}) - \varphi_r(\mathbf{y}) \big] 
\mu_\mathbf{x}(\mathrm{d} \mathbf{y}) + 
\int_{B_r} 
\big[ \psi_r(\mathbf{y}) - \varphi_r(\mathbf{y}) \big] 
\mu_\mathbf{x}(\mathrm{d} \mathbf{y}) 
\bigg|
%
\ \big( \varphi_r = 0 \textrm{ on } \partial D \setminus A \big)\\
&\leq \int_{\partial D \setminus A} 
\psi_r (\mathbf{y}) \mu_\mathbf{x}(\mathrm{d} \mathbf{y}) + 
\mu_\mathbf{x}(A\setminus B_r) 
\quad \big( \psi_r = \varphi_r = 1 \textrm{ on } B_r \big)\\
&= \mu_\mathbf{x}(A^c_r) 
\quad \big( \psi_r = 0 \textrm{ on } \big\{ \mathbf{x} \in \partial D \setminus A \: \big| \: {\rm dist}(\mathbf{x}, A) > r\} \big). 
\end{align*}

The same procedure can be applied to the second term of the max function in the right-hand side of \eqref{eq:max bounds} and, therefore, relation \eqref{eq:error_A} is obtained.

If \eqref{eq:mu-regularity} holds, then by the previous estimates, one obtains
\begin{align*}
\vert \mu_\mathbf{x} (A) - \mu^\varepsilon_\mathbf{x}(A) \vert \leq 
&\varepsilon \lambda_{\mathbf{K}}^{-1/2} 
\big( 1 + r^{-1} \big)
\left[31 + 
\frac{80 M_{\mathbf{K}}}{\mathrm{e} \log{2}} 
\left( \dfrac{1}{2} \vert \log{\lambda_{\mathbf{K}}} \vert + 
\vert \log{\varepsilon} \vert \right) 
\right] \\
&+  C\big( \mathbf{x}, {\rm dist}(\mathbf{x}, A^c_r) \big) \, \sigma(A^c_r)^q, 
\quad \mathbf{x} \in D, 
\quad r > 0.
\end{align*}
Since ${\rm dist}(\mathbf{x}, A^c_r) \geq {\rm dist}(\mathbf{x}, A) - r$, $C(\cdot,\cdot)$ is assumed to be non-decreasing w.r.t. the second variable and $A$ has a bounded $(d-1)-$dimensional upper Minkowski content, i.e. satisfies condition \eqref{eq:regular_A}, the above estimate entails the following one
\begin{align*}
\vert \mu_\mathbf{x} (A) - \mu^\varepsilon_\mathbf{x}(A) \vert 
&\leq \displaystyle \inf \limits_{r > 0} \left[\widetilde{\varepsilon} 
\big( 1 + r^{-1} \big) + 
C\big( \mathbf{x}, {\rm dist}(\mathbf{x}, A) - r \big) \, C_A^qr^q\right].\\
&\leq \min\left\{\inf\limits_{r\leq 1} \left[\widetilde{\varepsilon}(1+r^{-q})+C(\mathbf{x}, {\sf dist}(\mathbf{x},A)-r)C_A^q r^q\right]\right.,\\
&\phantom{\leq \min\;\;\;} \left.\inf\limits_{r>1} \left[\widetilde{\varepsilon}(1+r^{-1})+C(\mathbf{x}, {\sf dist}(\mathbf{x},A)-r)C_A^q r\right]\right\}.
\end{align*}
Further, we simply rely on the upper bound obtained by choosing $r = r^\ast$, where $r^\ast$ is given in \Cref{coro:mu - mu epsilon}. 
This choice together with \eqref{eq:r_ast} lead to
\begin{align*}
\vert \mu_\mathbf{x} (A) - \mu^\varepsilon_\mathbf{x}(A) \vert \leq 
\widetilde{\varepsilon} + 
2 \widetilde{\varepsilon}^{1/2} \, 
C\big( \mathbf{x}, {\rm dist}(\mathbf{x}, A)/2 \big)^{1/2} \, C_A^{q/2}
\end{align*}
and this completes the proof.
\end{proof}

\begin{defi} \label{def:extension_A_epsilon}
Let $A \in \mathcal{B}(\partial D)$ and $A_\varepsilon\subset \mathcal{B}(\overline{D}_\varepsilon)$, $\varepsilon \in (0, \varepsilon_0)$. 
For every $\varepsilon \in (0, \varepsilon_0)$, set $A_\pi^\varepsilon \coloneqq \pi\big( 
\left(A_\varepsilon\setminus \pi^{-1}(A)\right) 
\cup \left(\pi^{-1}(A) \setminus A_\varepsilon\right) \big)$, where $\pi$ is a fixed measurable selection of the projection onto the boundary. 
We say that $(A_\varepsilon)_{\varepsilon>0}$ is an extension of $A$ with a uniformly bounded $(d-1)-$dimensional upper Minkowski content if 
\begin{equation}\label{eq:extension_Minkowski}
\exists~\mathcal{C}_A > 0, \quad 
\forall~\varepsilon \in (0, \varepsilon_0), \quad 
\forall~r>0, \quad 
\sigma(A^c_r) + \sigma\big( \left(A_\pi^\varepsilon\right)^c_r \big) \leq 
\mathcal{C}_A r
\quad \mbox{ and } 
\quad A_\pi^\varepsilon \subset A^c_{\varepsilon\mathcal{C}_A},
\end{equation}
where $A^c_r$ and $\left(A_\pi^\varepsilon\right)^c_r$ are given by \eqref{eq:B^c_r}.
\end{defi}

\begin{coro}\label{coro:mu - mu eepsilon}
Assume that condition \ref{LHMD} holds for the domain $\mathbf{K}^{-1/2}D$, some $\alpha \in (0, 1]$ and $M_{\mathbf{K}}$ instead of $M$, and condition $\mathbf{\left(H_{\mu/\sigma}\right)}$ is also satisfied. 
Let $\mathbf{x} \in D$, $A\in \mathcal{B}(\partial D)$ and $A_\varepsilon \subset \mathcal{B}(\overline{D}_\varepsilon)$, $\varepsilon \in (0, \varepsilon_0)$, be an extension of $A$ with a uniformly bounded $(d-1)-$dimensional upper Minkowski content, see, e.g., \Cref{def:extension_A_epsilon}. 
Let $\widetilde{\varepsilon}$, $\widetilde{C}(\mathbf{x}, d, \mathbf{K}, D, A, q)$ and $r^\ast$ be as in \Cref{coro:mu - mu epsilon}. 
Then, for any $\varepsilon > 0$ sufficiently small such that
\begin{equation*}
{\rm dist}(\mathbf{x}, A) \geq 
2r^\ast + 2 \varepsilon \mathcal{C}_A, 
\end{equation*}
the following estimate holds
\begin{align*}
\Big\vert \mu_\mathbf{x} (A) - 
\mathbb{E}\left[ 1_{A^\varepsilon} \left(
X^\mathbf{x}_{\mathrm{N}_\varepsilon^\mathbf{x}} \right) \right] \Big\vert &\leq 
2 \widetilde{\varepsilon} + 
2 \widetilde{\varepsilon}^{1/2} \, 
\widetilde{C}(\mathbf{x}, d, \mathbf{K}, D, A) + 
\varepsilon^q \, \mathcal{C}_A^{1+q} \, 
C\big(\mathbf{x}, {\rm dist}(\mathbf{x}, A)/2 \big)\nonumber \\
&\in \mathcal{O}\Big( \varepsilon^{(1/2) \wedge q} \vert \log \varepsilon \vert \Big).
\end{align*}
\end{coro}
\begin{proof}
The following relations hold
\begin{align*}
&\Big\vert \mu_\mathbf{x} (A) - 
\mathbb{E}\left[ 1_{A^\varepsilon} \left(
X^\mathbf{x}_{\mathrm{N}_\varepsilon^\mathbf{x}} \right) \right] \Big\vert 
\leq \big\vert \mu_\mathbf{x}(A) - \mu_\mathbf{x}^\varepsilon A) \big\vert + 
\Big\vert \mathbb{E}\left[ 1_{A} \circ \pi \left( X^\mathbf{x}_{\mathrm{N}_\varepsilon^\mathbf{x}} \right) \right] - \mathbb{E}\left[ 1_{A^\varepsilon}\left( X^\mathbf{x}_{\mathrm{N}_\varepsilon^\mathbf{x}} \right) \right] \Big\vert \\
&\leq \big\vert \mu_\mathbf{x}(A) - \mu_\mathbf{x}^\varepsilon(A) \big\vert + 
\mathbb{E}\left[ 1_{\left( A_\varepsilon\setminus \pi^{-1}(A) \right) \cup \left( \pi^{-1}(A) \setminus A_\varepsilon \right)}\left( X^\mathbf{x}_{\mathrm{N}_\varepsilon^\mathbf{x}} \right) \right]\\
&\leq \big\vert \mu_\mathbf{x}(A) - \mu_\mathbf{x}^\varepsilon(A) \big\vert + 
\big\vert \mu_\mathbf{x}(A_\pi^\varepsilon) - \mu_\mathbf{x}^\varepsilon(A_\pi^\varepsilon) \big\vert + 
\mu_\mathbf{x}(A_\pi^\varepsilon) 
\quad \big( \pi^{-1}\left( A_\pi^\varepsilon \right) \supset \left( A_\varepsilon\setminus \pi^{-1}(A) \right) \cup \left( \pi^{-1}(A) \setminus A_\varepsilon \right) \big)\\
&\leq \big\vert \mu_\mathbf{x}(A) - \mu_\mathbf{x}^\varepsilon(A) \big\vert + 
\big\vert \mu_\mathbf{x}(A_\pi^\varepsilon) - \mu_\mathbf{x}^\varepsilon(A_\pi^\varepsilon) \big\vert + \mu_\mathbf{x}(A_{\varepsilon \mathcal{C}_A}^c), 
\quad \big( \textrm{cf.} \eqref{eq:extension_Minkowski} \big).
\end{align*}
Now, \eqref{eq:explicit_A} and \eqref{eq:mu-regularity} can be applied for the first two and the last term above, respectively, to obtain
\begin{align*}
\Big\vert \mu_\mathbf{x} (A) - 
\mathbb{E}\left[ 1_{A^\varepsilon} \left(
X^\mathbf{x}_{\mathrm{N}_\varepsilon^\mathbf{x}} \right) \right] \Big\vert
&\leq  
2 \widetilde{\varepsilon} + 
2 \widetilde{\varepsilon}^{1/2} \, 
\widetilde{C}(\mathbf{x}, d, \mathbf{K}, D, A) + 
C\big( \mathbf{x}, {\rm dist}(\mathbf{x}, A^c_{\varepsilon \mathcal{C}_A}) \big) \, 
\sigma(A^c_{\varepsilon \mathcal{C}_A})^q\\
&\leq 
2 \widetilde{\varepsilon} + 
2 \widetilde{\varepsilon}^{1/2} \, 
\widetilde{C}(\mathbf{x}, d, \mathbf{K}, D, A) + 
\varepsilon^q \, \mathcal{C}_A^{1+q} \, 
C\big( \mathbf{x}, {\rm dist}(\mathbf{x}, A^c_{\varepsilon \mathcal{C}_A}) \big)\\
&\leq 
2 \widetilde{\varepsilon} + 
2 \widetilde{\varepsilon}^{1/2} \, 
\widetilde{C}(\mathbf{x}, d, \mathbf{K}, D, A) + 
\varepsilon^q \, \mathcal{C}_A^{1+q} \, 
C\big( \mathbf{x}, {\rm dist}(\mathbf{x}, A) - 
\varepsilon \mathcal{C}_A \big)\\
&\leq 
2 \widetilde{\varepsilon} + 
2 \widetilde{\varepsilon}^{1/2} \, 
\widetilde{C}(\mathbf{x}, d, \mathbf{K}, D, A) + 
\varepsilon^q \, \mathcal{C}_A^{1+q} \, 
C\big( \mathbf{x}, {\rm dist}(\mathbf{x}, A)/2 \big)
\end{align*}
and this concludes the proof.
\end{proof}

We conclude this paragraph with the following direct consequence of \Cref{coro:mu - mu eepsilon}.
\begin{coro}\label{coro:mu - mu eepsilon-Voronoi}
Assume that condition \ref{LHMD} holds for the domain $\mathbf{K}^{-1/2}D$, some $\alpha \in (0, 1]$ and $M_{\mathbf{K}}$ instead of $M$, and condition $\mathbf{\left(H_{\mu/\sigma}\right)}$ is also satisfied. 
Further, assume that the weight functions $\boldsymbol{\omega}^1$ and their extensions $\boldsymbol{\omega}^{1, \varepsilon}$ given by \eqref{eq:boundary_weights} and \eqref{eq:extended_weights_1}, respectively, are binary weights given by \eqref{eq:binary wieghts}, such that $\big( V_i^\varepsilon \big)_{\varepsilon \in (0, \varepsilon_0)}$ is an extension of $V_i$ with a uniformly bounded $(d-1)-$dimensional upper Minkowski content, for $i = \overline{1,M_1}$ (see, e.g., \Cref{def:extension_A_epsilon} with $A$ replaced by $V_i$). 
Let $\widetilde{\varepsilon}$, $\widetilde{C}_j(\mathbf{x}, d, \mathbf{K}, D, V_j, q)$ and $r^\ast_j$, $j = \overline{1,M_1}$, be as in \Cref{coro:mu - mu epsilon}, with $A$ replaced by $V_j$, $j = \overline{1,M_1}$.
Further, let $\mathbf{A}_{\omega^1}$ and $\mathbf{A}_{\omega^1,\varepsilon}$ be the matrices given by \eqref{eq:AM matrix components} and \eqref{eq:WoE_matrix}, respectively. 
Then, 
\begin{equation*}
\forall~\varepsilon \in (0,\varepsilon_0) \textrm{ such that} \quad 
{\rm dist}(x_i^D,V_j) \geq 
2 r^\ast_j + 2 \varepsilon \mathcal{C}_{V_j}, 
\quad i = \overline{1,M_D}, 
\quad j = \overline{1,M_1}, 
\end{equation*}
the following estimate holds
\begin{align*}
\Big\vert \left[ \mathbf{A}_{\omega^1} \right]_{ij} - \left[ \mathbf{A}_{\omega^1,\varepsilon} \right]_{ij} \Big\vert 
&\leq 
2 \widetilde{\varepsilon} + 
2 \widetilde{\varepsilon}^{1/2} \, 
\widetilde{C}(x_i^D, d, \mathbf{K}, D, V_j) + 
\varepsilon^q \, \mathcal{C}_{V_j}^{1+q} \, 
C\big( x_i^D, {\rm dist}(x_i^D, A)/2 \big)\nonumber \\
&\in \mathcal{O}\Big( \varepsilon^{(1/2) \wedge q} \vert \log \varepsilon \vert \Big), 
\quad i = \overline{1,M_D}, 
\quad j = \overline{1,M_1}.
\end{align*}
\end{coro}

\paragraph{Tail estimates for the Monte Carlo approximation error.}
For $\mathbf{x}_1, \mathbf{x}_2, \ldots, \mathbf{x}_n \in \partial D$ and $(\omega_i^\varepsilon)_{i = \overline{1,n}}$ some associated weight functions defined on $\overline{D}_\varepsilon$, see, e.g., in relations \eqref{eq:boundary_weights} and \eqref{eq:extended_weights_1}, 
and $f \in b\mathcal{B}(\partial D)$, we define, in accordance with \eqref{eq:u epsilon},
\begin{equation*}\label{eq:u epsilon omega}
u^\varepsilon_{f_{\omega^\varepsilon}}(\mathbf{x}) \coloneqq {\mathbb{E}}\big[ f_{\omega^\varepsilon} (X^\mathbf{x}_{\mathrm{N}^\mathbf{x}_\varepsilon}) \big], 
\quad \mathbf{x} \in D,
\end{equation*}
where $f_{\omega^\varepsilon}$ is the $\omega^\varepsilon-$interpolant of $f$ given by \eqref{eq:interpolant}. 
Further, for $\mathbf{x}\in D$ and $N \geq 1$, let $X^{\mathbf{x},1}_{\mathrm{N}^{\mathbf{x},1}_\varepsilon}, X^{\mathbf{x},2}_{\mathrm{N}^{\mathbf{x},2}_\varepsilon}, \ldots, X^{\mathbf{x},N}_{\mathrm{N}^{\mathbf{x},N}_\varepsilon}$ be i.i.d. copies of $X^\mathbf{x}_{\mathrm{N}^\mathbf{x}_\varepsilon}$ and define the {\it Monte Carlo estimator} of $u^\varepsilon_{f_{\omega^\varepsilon}}$ by
\begin{equation*}\label{eq:monte carlo snx}
S^\mathbf{x}_{\omega, \varepsilon, N}(f) \coloneqq 
\dfrac{ f_{\omega^\varepsilon}  \left( X^{\mathbf{x},1}_{\mathrm{N}^{\mathbf{x},1}_\varepsilon} \right) + 
f_{\omega^\varepsilon} \left( X^{\mathbf{x},2}_{\mathrm{N}^{\mathbf{x},2}_\varepsilon} \right) + \ldots + 
f_{\omega^\varepsilon} \left( X^{\mathbf{x},N}_{\mathrm{N}^{\mathbf{x},N}_\varepsilon} \right) }{N}, 
\quad N \geq 1.
\end{equation*}

\begin{coro} \label{coro:final error} 
Assume that condition \ref{LHMD} holds for the domain $\mathbf{K}^{-1/2}D$ for some $\alpha \in (0, 1]$ and $M_{\mathbf{K}}$ instead of $M$, and suppose $f \in C^\alpha(\partial D)$. 
Then, the following estimate holds 
\begin{equation*}\label{eq:final error estimate}
\mathbb{P} \{ \vert u_f (\mathbf{x}) - S^\mathbf{x}_{\omega, \varepsilon, N}(f)  \vert \geq \gamma \} \leq 2 \, \mathrm{e}^{-N \frac{\gamma(f,\varepsilon)^2}{\max\limits_{1 \leq i \leq n} \vert f(\mathbf{x}_i) \vert^2}}, \quad \mathbf{x} \in D,
\end{equation*}
for every $\varepsilon, \gamma > 0$ such that
\begin{equation*}
0 \leq \gamma(f, \varepsilon) \coloneqq 
\gamma - 
\varepsilon^\alpha \lambda_{\mathbf{K}}^{-\alpha/2} 
\Vert f \Vert_{\alpha,\partial D}
\left[ 30 + 
\dfrac{80 M_{\mathbf{K}}}{\mathrm{e} \log{2}} 
\left( \dfrac{1}{2} \vert \log{\lambda_{\mathbf{K}}} \vert + 
\vert \log{\varepsilon} \vert \right) \right] - 
\vert f_{\omega^\varepsilon} - f\circ\pi \vert_{\infty,\overline{D}_\varepsilon}.
\end{equation*} 
\end{coro}
\begin{proof}
By applying Hoeffding's inequality \cite[Theorem 2]{Ho63} for the random variables $f_{\omega^\varepsilon}  \left(X^{\mathbf{x},i}_{\mathrm{N}^{\mathbf{x},i}_\varepsilon}\right)$, it follows that, for every $\mathbf{x} \in D$ and $\eta > 0$,
\begin{equation}\label{eq:aux_tail}
\mathbb{P} \{ \vert u^\varepsilon_{f_{\omega^\varepsilon}}(\mathbf{x}) - S^\mathbf{x}_{\omega, \varepsilon, N} \vert \geq \eta \} \leq 2 \, \mathrm{e}^{\frac{-N \eta^2}{ \max\limits_{i = \overline{1, n}} \vert f(\mathbf{x}_i) \vert^2 }}.
\end{equation}
Since
\begin{equation*}
\mathbb{P} \{ \vert u_f (\mathbf{x}) - S^\mathbf{x}_{\omega,\varepsilon,N} \vert \geq \gamma \} \leq \mathbb{P} \{ \vert u^\varepsilon_{f_{\omega^\varepsilon}}(\mathbf{x}) - S^\mathbf{x}_{\omega,\varepsilon,N} \vert \geq 
\gamma - \vert u_f(\mathbf{x}) - u^\varepsilon_{f_{\omega^\varepsilon}}(\mathbf{x}) \vert \},
\end{equation*}
the result follows immediately  from \eqref{eq:aux_tail} for $\eta \coloneqq \gamma - \vert u_f(\mathbf{x}) - u^\varepsilon_{f_{\omega^\varepsilon}}(\mathbf{x}) \vert$ and \Cref{thm:error estimates}. 
\end{proof}

We further let $\mathbf{A}_{\omega^1}$, $\mathbf{A}_{\omega^1,\varepsilon}$ and $\mathbf{A}_{\omega^1,\varepsilon,N}$ be the matrices given by relations \eqref{eq:AM matrix components}, \eqref{eq:WoE_matrix} and \eqref{eq:WoE_matrix_MC}, respectively. 
The following result regarding the approximation of the elliptic densities $\rho_{\mathbf{x}_i^D}$ by $\rho_{\mathbf{x}_i^D, \omega^1,\varepsilon, N}$ \eqref{eq:rho omega1 MC}, $i = \overline{1, M_D}$, is obtained.
\begin{thm}\label{coro:rho_MC_est} 
The following estimate holds 
\begin{equation}\label{eq:tail_hdensity}
\mathbb{P}\left\{\sup \limits_{\mathbf{x} \in \Gamma_1} 
\bigg| \rho_{\mathbf{x}_i^D}(\mathbf{x})
- \rho_{\mathbf{x}_i^D,\omega^1,\varepsilon, N}(\mathbf{x}) \bigg|\geq \gamma\right\} \leq 2M_1 \mathrm{e}^{-N\left(\delta_{\rho,i}(\gamma)\right)^2}, 
\quad i = \overline{1, M_D}, 
\quad \gamma > 0,
\end{equation}
where
\begin{align}
\delta_{\rho,i}(\gamma) \coloneqq 
\left(\gamma - {\big( {\rm osc}_{M_1, \omega^1} \big)}_i - \max_{j = \overline{1, M_1}} \left\vert 
\left[ \mathbf{A}_{\omega^1} \right]_{ij} -\left[ \mathbf{A}_{\omega^1,\varepsilon} \right]_{ij} \right\vert 
\sigma\big( \omega^1_j\big)^{-1} \right)^{+} 
\min_{j = \overline{1, M_1}} \sigma(\omega_j^1), 
\label{eq:delta_rho_i} 
\end{align}
$\big({\rm osc}_{M_1,\omega^1}\big)_i$ given by \eqref{eq:notation.a} and $x^+ \coloneqq \max(x, 0)$.

In particular, if the elliptic densities $\rho_{\mathbf{x}_i^D}$, $i = \overline{1, M_D}$, are continuous on $\overline{\Gamma_1}$ and the assumptions from \Cref{coro: error estimates} or \Cref{coro:mu - mu eepsilon-Voronoi} are fulfilled, then 
\begin{equation}\label{eq:convergence_hdensity}
\lim_{{\rm diam}(\omega^1) \to 0} 
\lim_{\varepsilon \to 0} 
\lim_{N \to \infty} 
\sup \limits_{\mathbf{x} \in \Gamma_1} 
\left\vert \rho_{\mathbf{x}_i^D}(\mathbf{x}) - 
\rho_{\mathbf{x}_i^D,\omega^1,\varepsilon, N}(\mathbf{x}) \right\vert = 0, 
\quad \mathbb{P}\mbox{--a.s. and in } L^1(\mathbb{P}), 
\end{equation}
where ${\rm diam}(\omega^1) \coloneqq \max \limits_{i=\overline{1, M_1}} {\rm diam}\big({\rm supp}(\omega^1_i) \big)$.
\end{thm}
\begin{proof}
The following relations hold
\begin{align*}
&\sup \limits_{\mathbf{x} \in \Gamma_1} 
\left| \rho_{\mathbf{x}_i^D}(\mathbf{x})-\rho_{\mathbf{x}_i^D,\omega^1,\varepsilon, N}(\mathbf{x}) \right|\\
&\leq \sup \limits_{\mathbf{x} \in \Gamma_1} 
\left| \rho_{\mathbf{x}_i^D}(\mathbf{x})
- \sum \limits_{j=1}^{M_1} \sigma\big(\omega^1_j\big)^{-1} \mu_{\mathbf{x}_i^D}\big(\omega^1_j\big) \omega^1_j(\mathbf{x}) \right| 
+ \sup \limits_{\mathbf{x} \in \Gamma_1} 
\sum \limits_{j=1}^{M_1} \sigma\big(\omega^1_j\big)^{-1} \left|\left[\mathbf{A}_{\omega^{1}}\right]_{ij}-\left[\mathbf{A}_{\omega^1,\varepsilon}\right]_{ij}\right| \omega^1_j(\mathbf{x})\\
 &\quad + \sup \limits_{\mathbf{x} \in \Gamma_1} 
\bigg|\sum \limits_{j=1}^{M_1} \sigma\big(\omega^1_j\big)^{-1} \left[\mathbf{A}_{\omega^1,\varepsilon}\right]_{ij} \omega^1_j(\mathbf{x}) - \rho_{\mathbf{x}_i^D,\omega^1,\varepsilon, N}(\mathbf{x})
 \bigg|\\
&\leq {\big({\rm osc}_{M_1,\omega^1}\big)}_i + 
\max_{j=\overline{1, M_1}} \left|\left[\mathbf{A}_{\omega^1}\right]_{ij } -\left[\mathbf{A}_{\omega^1,\varepsilon}\right]_{ij}\right|\sigma\big(\omega^1_j\big)^{-1} + 
\max_{j=\overline{1, M_1}} \left|\left[\mathbf{A}_{\omega^1,\varepsilon}\right]_{ij } -\left[\mathbf{A}_{\omega^1,\varepsilon, N}\right]_{ij}\right|\sigma\big(\omega^1_j\big)^{-1}.
\end{align*}
Then, using the union bound inequality and then Hoeffding's inequality \cite[Theorem 2]{Ho63}, one obtains
\begin{align*}
\mathbb{P}\left\{\sup \limits_{\mathbf{x} \in \Gamma_1} 
\bigg| \rho_{\mathbf{x}_i^D}(\mathbf{x})
- \rho_{\mathbf{x}_i^D,\omega^1,\varepsilon, N}(\mathbf{x}) \bigg|\geq \gamma\right\} &\leq  \mathbb{P}\left\{\max_{1\leq j\leq M_1}\left|\left[\mathbf{A}_{\omega^1,\varepsilon}\right]_{ij } -\left[\mathbf{A}_{\omega^1,\varepsilon, N}\right]_{ij}\right|\geq \delta_{\rho,i}(\gamma)\right\}\\
&\leq 2M_1\mathrm{e}^{-N\left(\delta_{\rho,i}(\gamma)\right)^2}.
\end{align*}
Finally, the $\mathbb{P}-$a.s. convergence \eqref{eq:convergence_hdensity} follows by the law of large numbers, whilst the $L^1(\mathbb{P})-$convergence can be proved as follows. If we denote
\begin{equation*}
\mathcal{E} \coloneqq \sup \limits_{\mathbf{x} \in \Gamma_1} 
\bigg| \rho_{\mathbf{x}_i^D}(\mathbf{x})
- \rho_{\mathbf{x}_i^D,\omega^1,\varepsilon, N}(\mathbf{x}) \bigg|,
\end{equation*}
then, by \eqref{eq:tail_hdensity}, it follows that
\begin{align*}
\mathbb{E}\left\{\mathcal{E}\right\}=\int_0^\infty \mathbb{P}\left\{\mathcal{E}\geq t\right\}\;\mathrm{d}t\leq \gamma + \int_\gamma^\infty \mathbb{P}\left\{\mathcal{E}\geq t\right\}\;\mathrm{d}t\leq \gamma + 2M_1\int_\gamma^\infty  \mathrm{e}^{-N\left(\delta_{\rho,i}(t)\right)^2} \;\mathrm{d}t.
\end{align*}
Without writing down all of the details, under the assumptions made, it can easily be seen that fixing and arbitrary $\gamma > 0$, i.e., by making sufficiently small first ${\rm diam}(\omega^1)$ and then $\varepsilon$, one can ensure that, for some constant $c(\gamma) >$, the following relation holds
$
\left( \delta_{\rho,i}(t) \right)^2 \geq c(\gamma) t^2, \quad t\geq \gamma, 
$
and, for such choices, one obtains
\begin{align*}
\mathbb{E}\left\{\mathcal{E}\right\}\leq \gamma + 2M_1\int_\gamma^\infty  \mathrm{e}^{-Nc(\gamma) t^2} \;\mathrm{d}t.
\end{align*}
Letting $N \to \infty$, it follows that $\limsup\limits_N\mathbb{E}\left\{ \mathcal{E}\right\}\leq \gamma$ and hence 
$
\limsup\limits_{{\rm diam}(\omega^1)\to 0 }\limsup\limits_{\varepsilon\to 0}\limsup\limits_N\mathbb{E}\left\{ \mathcal{E}\right\}=0.$ 

\qedhere
\end{proof}

Let $\boldsymbol{\Lambda}^{\boldsymbol{\nu}}$, $\boldsymbol{\Lambda}^{\boldsymbol{\nu}}_{\omega^1}$, $\boldsymbol{\Lambda}^{\boldsymbol{\nu}}_{\omega^1,\varepsilon}$, $\boldsymbol{\Lambda}^{\boldsymbol{\nu}}_{\omega^1,\varepsilon, N}$ be defined by \eqref{eq:Lambda}, \eqref{eq:Lambda_hat_nu}, \eqref{eq:Lambda_hat_eps_nu}, \eqref{eq:Lambda_tilde_nu}, respectively. The following result regarding the matrix $\boldsymbol{\Lambda}^{\boldsymbol{\nu}}$ is obtained.

\begin{coro}\label{coro:MCmain}
Assume that condition \ref{LHMD} holds for the domain $\mathbf{K}^{-1/2}D$, some $\alpha \in (0, 1]$ and $M_{\mathbf{K}}$ instead of $M$, and condition $\mathbf{\left(H_{\mu/\sigma}\right)}$ is also satisfied. 
For each $\varepsilon > 0$, set 
\begin{equation*}\label{eq:Gamma_omega1_varepsilon}
\Gamma_{\omega^{1,\varepsilon}} \coloneqq 
\bigcup \limits_{i = 1}^{M_1} {\rm supp} \big( \omega^{1, \varepsilon}_i \big),    
\end{equation*}
and suppose that $\big( \Gamma_{\omega^{1, \varepsilon}} \big)_{\varepsilon>0}$ is an extension of $\Gamma_1$ with a uniform bounded $(d-1)-$dimensional upper Minkowski content with the constant $\mathcal{C}_{\Gamma_1} < \infty$, i.e. \eqref{eq:extension_Minkowski} holds for $A = \Gamma_1$.
Let $\widetilde{\varepsilon}$, $\widetilde{C}(\mathbf{x},d,\mathbf{K},D,\Gamma_1,q)$ and $r^\ast$ be as given in \Cref{coro:mu - mu epsilon} and set
\begin{align*}
&\mathcal{C}_1(\mathbf{x}) \coloneqq 
2 \, \widetilde{\varepsilon} +
2 \, \widetilde{\varepsilon}^{1/2} \, 
\widetilde{C}(\mathbf{x},d,\mathbf{K},D,\Gamma_1) + 
\varepsilon^q \, 
\mathcal{C}_{\Gamma_1}^{1+q} \, 
C\big( \mathbf{x}, {\rm dist}(\mathbf{x}, \Gamma_1)/2 \big)\\
&\mathcal{C}_2 \coloneqq 
\left[ \sum_{i=1}^{M_D} \nu_i^2 \, 
\max_{j = \overline{1,M_1}} \left|\left[\mathbf{A}_{\omega^1}\right]_{ij} -\left[\mathbf{A}_{\omega^1,\varepsilon}\right]_{ij}\right|^2\sigma(\omega_j^1)^{-2}\bigg( \mu_{\mathbf{x}^D_i}(\Gamma_1) + \sum_{k=1}^{M_D} \nu_k^2 \, \mu_{\mathbf{x}^D_k}(\Gamma_1) + \sum_{k=1}^{M_D} \nu_k^2 \, \mathcal{C}_1(\mathbf{x}^D_k)\bigg)^2 \right]^{1/2}\\
&{\gamma_\Lambda}=\gamma_\Lambda\big( \varepsilon,M_1,M_D,\omega^1,\delta^{\boldsymbol{\nu}}_{D,M_D,\omega^1} \big) \coloneqq \big( \gamma-\delta^{\boldsymbol{\nu}}_{D,M_D,\omega^1}-\mathcal{C}_2 \big)^{+}.
\end{align*}
Then, for any $\varepsilon > 0$ sufficiently small such that
\begin{equation*}
{\rm dist}(\mathbf{x}_i^D,\Gamma_1) \geq 
2 r^\ast + 2 \varepsilon\mathcal{C}_{\Gamma_1}, 
\quad i = \overline{1,M_D},
\end{equation*}
and any $\gamma > 0$, the following estimate holds
\begin{equation*}
\mathbb{P}\left\{ 
\Vert \boldsymbol{\Lambda}^{\boldsymbol{\nu}} - 
\boldsymbol{\Lambda}^{\boldsymbol{\nu}}_{\omega^1,\varepsilon, N} \Vert_{\rm F} \geq \gamma \right\} \leq \delta_\Lambda(\gamma),
\end{equation*}
where
\begin{align}\label{eq:delta_Lambda}
\delta_\Lambda(\gamma) & \coloneqq 2M_1M_D \bigg[ \mathrm{e}^{-N {\gamma_\Lambda}^2/9 \min\limits_{j=\overline{1,M_1}} \sigma(\omega_j^1)^2 
\left(2\max\limits_{i=\overline{1,M_D}} \mu_{\mathbf{x}^D_i}(\Gamma_1) + 
2\max\limits_{i=\overline{1,M_D}}
\mathcal{C}_1(\mathbf{x}^D_i)\right)^{-2}} + \mathrm{e}^{-2N {\gamma_\Lambda}/3 \, 
\min\limits_{j=\overline{1,M_1}} 
\sigma(\omega_j^1)^2}\bigg] \nonumber \\
&\quad + 2 M_1 \mathrm{e}^{-2N{\gamma_\Lambda}/3}.
\end{align}

In particular, if the elliptic densities $\rho_{\mathbf{x}_i^D}$, $i=\overline{1,M_D}$, are continuous on $\overline{\Gamma_1}$ and the assumptions from \Cref{coro: error estimates} or \Cref{coro:mu - mu eepsilon-Voronoi} are fulfilled, then
\begin{equation}\label{eq:convergence_Lambda}
\lim_{{\rm diam}(\omega^1) \to 0} 
\lim_{\varepsilon \to 0} 
\lim_{N \to \infty} 
\Vert \boldsymbol{\Lambda}^{\boldsymbol{\nu}} - 
\boldsymbol{\Lambda}^{\boldsymbol{\nu}}_{\omega^1,\varepsilon, N} \Vert_{\rm F} = 0, 
\quad \mathbb{P}\mbox{--a.s. and in } L^1(\mathbb{P}).
\end{equation}
\end{coro}
\begin{proof}
Note that by \eqref{eq:norm leq delta}, 
\begin{equation}\label{eq:0bound}
\Vert \boldsymbol{\Lambda}^{\boldsymbol{\nu}} - 
\boldsymbol{\Lambda}^{\boldsymbol{\nu}}_{\omega^{1}} \Vert_{\rm F} \leq \delta^{\boldsymbol{\nu}}_{D,M_D,\omega^1}.
\end{equation}
Furthermore, the following identity holds
\begin{align*}
&\left\Vert \boldsymbol{\Lambda}^{\boldsymbol{\nu}}_{\omega^{1}} - 
\boldsymbol{\Lambda}^{\boldsymbol{\nu}}_{\omega^1, \varepsilon} \right\Vert_{\rm F}
= \Bigg( 
\sum \limits_{i = 1}^{M_D} \sum \limits_{j = 1}^{M_D} 
\nu^2_i\nu^2_j \, \left|\left[\mathbf{\Lambda}_{\omega^1}\right]_{ij}-\left[\mathbf{\Lambda}_{\omega^1, \varepsilon}\right]_{ij}\right|^2
\Bigg)^{1/2}.
\end{align*}
On the other hand, the following relations hold
\begin{align*}
& \left\vert 
\left[\mathbf{\Lambda}_{\omega^1}\right]_{ij} - 
\left[\mathbf{\Lambda}_{\omega^1, \varepsilon}\right]_{ij} \right\vert = 
\left\vert 
\sum_{k = 1}^{M_1} 
\left[ \mathbf{A}_{\omega^1} \right]_{ik} 
\sigma(\omega_k^1)^{-1} 
\left[ \mathbf{A}_{\omega^1} \right]_{jk } - 
\sum_{k = 1}^{M_1}  
\left[ \mathbf{A}_{\omega^1,\varepsilon} \right]_{ik} 
\sigma(\omega_k^1)^{-1} 
\left[ \mathbf{A}_{\omega^1,\varepsilon} \right]_{jk} 
\right\vert\\
&\leq \left\vert 
\sum_{k = 1}^{M_1} \left[\mathbf{A}_{\omega^1}\right]_{ik} 
\sigma(\omega_k^1)^{-1} 
\left( \left[\mathbf{A}_{\omega^1}\right]_{jk } -\left[\mathbf{A}_{\omega^1,\varepsilon}\right]_{jk} \right) \right\vert + 
\left\vert 
\sum_{k = 1}^{M_1}  
\left( \left[\mathbf{A}_{\omega^1}\right]_{ik} - \left[\mathbf{A}_{\omega^1,\varepsilon}\right]_{ik} \right) 
\sigma(\omega_k^1)^{-1} 
\left[\mathbf{A}_{\omega^1,\varepsilon}\right]_{jk} 
\right\vert\\
&\leq \max_{k=\overline{1,M_1}} 
\left|\left[\mathbf{A}_{\omega^1}\right]_{jk } -\left[\mathbf{A}_{\omega^1,\varepsilon}\right]_{jk} \right|\sigma(\omega_k^1)^{-1} 
\sum_{k=1}^{M_1} 
\left[\mathbf{A}_{\omega^1}\right]_{ik} + 
\max_{k=\overline{1,M_1}}
\left|\left[\mathbf{A}_{\omega^1}\right]_{ik } -\left[\mathbf{A}_{\omega^1,\varepsilon}\right]_{ik} \right|\sigma(\omega_k^1)^{-1} 
\sum_{k=1}^{M_1} 
\left[\mathbf{A}_{\omega^1,\varepsilon}\right]_{jk}\\
&\leq \max_{k=\overline{1,M_1}} 
\left|\left[\mathbf{A}_{\omega^1}\right]_{jk } -\left[\mathbf{A}_{\omega^1,\varepsilon}\right]_{jk} \right|\sigma(\omega_k^1)^{-1}\mu_{\mathbf{x}^D_i}(\Gamma_1)\\
&+ \max_{k=\overline{1,M_1}} 
\left|\left[\mathbf{A}_{\omega^1}\right]_{ik } -\left[\mathbf{A}_{\omega^1,\varepsilon}\right]_{ik} \right| \sigma(\omega_k^1)^{-1}\mathbb{E}
\left[ \sum_{k=1}^{M_1} \omega_k^{1,\varepsilon}\left(X^{\mathbf{x}_j^D}_{N_{\varepsilon}^{\mathbf{x}_j^D}}\right)\right]\\
&\leq \max_{k=\overline{1,M_1}} 
\left|\left[\mathbf{A}_{\omega^1}\right]_{jk } -\left[\mathbf{A}_{\omega^1,\varepsilon}\right]_{jk} \right|\sigma(\omega_k^1)^{-1}\left\{\mu_{\mathbf{x}^D_i}(\Gamma_1)+\mathbb{E}\left[1_{\Gamma_{\omega^{1,\varepsilon}}}\left(X^{\mathbf{x}_j^D}_{N_{\varepsilon}^{\mathbf{x}_j^D}}\right)\right]\right\}\\
&\leq \max_{k=\overline{1,M_1}} 
\left|\left[\mathbf{A}_{\omega^1}\right]_{jk } -\left[\mathbf{A}_{\omega^1,\varepsilon}\right]_{jk} \right|\sigma(\omega_k^1)^{-1}\left[\mu_{\mathbf{x}^D_i}(\Gamma_1)+\mu_{\mathbf{x}^D_j}(\Gamma_1)+\mathcal{C}_1(\mathbf{x}^D_j)\right],
\end{align*}
where the last inequality follows from \Cref{coro:mu - mu eepsilon}.
Therefore, since $\displaystyle \sum_{i=1}^{M_D} \nu_i^2 = 1$,
\begin{equation}\label{eq:first bound}
\left\Vert
\boldsymbol{\Lambda}^{\boldsymbol{\nu}}_{\omega^{1}} - \boldsymbol{\Lambda}^{\boldsymbol{\nu}}_{\omega^1, \varepsilon} \right\Vert_{\rm F} 
\leq \mathcal{C}_2.
\end{equation}
By similar computations to the ones above, one obtains
\begin{align*}
\left\Vert 
\boldsymbol{\Lambda}^{\boldsymbol{\nu}}_{\omega^1, \varepsilon} - \boldsymbol{\Lambda}^{\boldsymbol{\nu}}_{\omega^1,\varepsilon, N} 
\right\Vert_{\rm F}
\leq \max\limits_{i,j=\overline{1,M_D}} \left|\left[\mathbf{\Lambda}_{\omega^1, \varepsilon}\right]_{ij}-\left[\mathbf{\Lambda}_{\omega^1, \varepsilon,N}\right]_{ij}\right|,
\end{align*}
as well as
\begin{align*}
\left\vert 
\left[ \boldsymbol{\Lambda}_{\omega^1,\varepsilon} \right]_{ij} - 
\left[ \boldsymbol{\Lambda}_{\omega^1, \varepsilon,N} \right]_{ij} 
\right\vert
&\leq \max_{k=\overline{1,M_1}} 
\left\vert 
\left[ \mathbf{A}_{\omega^1,\varepsilon} \right]_{jk } - 
\left[ \mathbf{A}_{\omega^1,\varepsilon, N} \right]_{jk} 
\right\vert 
\sigma(\omega_k^1)^{-1} 
\sum_{k=1}^{M_1}
\mathbb{E}\left[ 1_{\Gamma_{\omega^{1,\varepsilon}}}\left(X^{\mathbf{x}_i^D}_{N_{\varepsilon}^{\mathbf{x}_i^D}}\right) \right]\\
& \quad  + \max_{k=\overline{1,M_1}} 
\left\vert 
\left[ \mathbf{A}_{\omega^1,\varepsilon} \right]_{ik } - 
\left[ \mathbf{A}_{\omega^1,\varepsilon, N} \right]_{ik} 
\right\vert 
\sigma(\omega_k^1)^{-1} S^{\mathbf{x}^D_j}_{\varepsilon,N}\left(1_{\Gamma_{\omega^{1,\varepsilon}}}\right)\\
&\leq \max_{k=\overline{1,M_1}} 
\left\vert 
\left[ \mathbf{A}_{\omega^1,\varepsilon} \right]_{jk } - 
\left[ \mathbf{A}_{\omega^1,\varepsilon, N} \right]_{jk} 
\right\vert 
\sigma(\omega_k^1)^{-1} \mathbb{E}\left[1_{\Gamma_{\omega^{1,\varepsilon}}}\left(X^{\mathbf{x}_i^D}_{N_{\varepsilon}^{\mathbf{x}_i^D}}\right)\right]\\
& \quad  + \max_{k=\overline{1,M_1}} 
\left\vert 
\left[ \mathbf{A}_{\omega^1,\varepsilon} \right]_{ik} - 
\left[ \mathbf{A}_{\omega^1,\varepsilon, N} \right]_{ik} 
\right\vert 
\sigma(\omega_k^1)^{-1} \mathbb{E}\left[1_{\Gamma_{\omega^{1,\varepsilon}}}\left(X^{\mathbf{x}_j^D}_{N_{\varepsilon}^{\mathbf{x}_j^D}}\right)\right]\\
& \quad  + \max_{k=\overline{1,M_1}} 
\left\vert 
\left[ \mathbf{A}_{\omega^1,\varepsilon} \right]_{ik} - 
\left[ \mathbf{A}_{\omega^1,\varepsilon, N} \right]_{ik} 
\right\vert 
\sigma(\omega_k^1)^{-1} 
\left\vert S^{\mathbf{x}^D_j}_{\varepsilon,N}\left(1_{\Gamma_{\omega^{1,\varepsilon}}}\right)-\mathbb{E}\left[1_{\Gamma_{\omega^{1,\varepsilon}}}\left(X^{\mathbf{x}_j^D}_{N_{\varepsilon}^{\mathbf{x}_j^D}}\right)\right] \right\vert,
\end{align*}
where $S^\mathbf{x}_{\varepsilon, N}(f) \coloneqq \dfrac{1}{N} \sum\limits_{i=1}^N f \left(X^{\mathbf{x},i}_{\mathrm{N}^{\mathbf{x},i}_\varepsilon} \right)$ and $X^{\mathbf{x},i}_{\mathrm{N}^{\mathbf{x},i}_\varepsilon}$, $i=\overline{1,N}$, are independent random variables identically distributed with $X^\mathbf{x}_{\mathrm{N}^\mathbf{x}_\varepsilon}$. Therefore, one obtains
\begin{align*}
&\left\|\boldsymbol{\Lambda}^{\boldsymbol{\nu}}_{\omega^1, \varepsilon} - \boldsymbol{\Lambda}^{\boldsymbol{\nu}}_{\omega^1,\varepsilon, N}\right\|_{\rm F} \\
& \leq \max_{\substack{i=\overline{1,M_D}\\ j=\overline{1,M_1}}} 
\left|\left[\mathbf{A}_{\omega^1,\varepsilon}\right]_{ij } -\left[\mathbf{A}_{\omega^1,\varepsilon, N}\right]_{ij} \right|\sigma(\omega_j^1)^{-1}\\
&\quad \times \left(2\max_{i=\overline{1,M_D}} 
\mu_{\mathbf{x}^D_i}(\Gamma_1) + 
2 \max_{i=\overline{1,M_D}} \mathcal{C}_1(\mathbf{x}^D_i) + 
\max_{i=\overline{1,M_D}} 
\left|S^{\mathbf{x}^D_i}_{\varepsilon,N}\left(1_{\Gamma_{\omega^{1,\varepsilon}}}\right)-\mathbb{E}\left[1_{\Gamma_{\omega^{1,\varepsilon}}}\left(X^{\mathbf{x}_i^D}_{N_{\varepsilon}^{\mathbf{x}_i^D}}\right)\right]\right|\right)\\
&\leq \max_{\substack{i=\overline{1,M_D}\\ j=\overline{1,M_1}}} 
\left|\left[\mathbf{A}_{\omega^1,\varepsilon}\right]_{ij } -\left[\mathbf{A}_{\omega^1,\varepsilon, N}\right]_{ij} \right|\sigma(\omega_j^1)^{-1} 
\left(2 \max_{i=\overline{1,M_D}} \mu_{\mathbf{x}^D_i}(\Gamma_1) + 
2\max_{i=\overline{1,M_D}} \mathcal{C}_1(\mathbf{x}^D_i)\right)\\
&\quad +\frac{1}{2} 
\max_{\substack{i=\overline{1,M_D}\\ j=\overline{1,M_1}}}
\left|\left[\mathbf{A}_{\omega^1,\varepsilon}\right]_{ij } -\left[\mathbf{A}_{\omega^1,\varepsilon, N}\right]_{ij} \right|^2\sigma(\omega_j^1)^{-2}
+\frac{1}{2}\max_{i=\overline{1,M_D}} 
\left|S^{\mathbf{x}^D_i}_{\varepsilon,N}\left(1_{\Gamma_{\omega^{1,\varepsilon}}}\right)-\mathbb{E}\left[1_{\Gamma_{\omega^{1,\varepsilon}}}\left(X^{\mathbf{x}_i^D}_{N_{\varepsilon}^{\mathbf{x}_i^D}}\right)\right]\right|^2.
\end{align*}
Coupling the previous bound with relations \eqref{eq:0bound} and \eqref{eq:first bound} yields
\begin{align*}
&\left\Vert
\boldsymbol{\Lambda}^{\boldsymbol{\nu}} - \boldsymbol{\Lambda}^{\boldsymbol{\nu}}_{\omega^1,\varepsilon, N}\right \Vert_{\rm F} 
\leq \delta^{\boldsymbol{\nu}}_{D,M_D,\omega^1} + \mathcal{C}_2\\
&\quad + \max_{\substack{i=\overline{1,M_D}\\ j=\overline{1,M_1}}} \left|\left[\mathbf{A}_{\omega^1,\varepsilon}\right]_{ij } -\left[\mathbf{A}_{\omega^1,\varepsilon, N}\right]_{ij} \right|\sigma(\omega_j^1)^{-1}
\left(2 \max_{i=\overline{1,M_D}} \mu_{\mathbf{x}^D_i}(\Gamma_1) + 
2 \max_{i=\overline{1,M_D}} \mathcal{C}_1(\mathbf{x}^D_i)\right)\\
&\quad + \dfrac{1}{2} 
\max_{\substack{i=\overline{1,M_D}\\ j=\overline{1,M_1}}} 
\left|\left[\mathbf{A}_{\omega^1,\varepsilon}\right]_{ij } -\left[\mathbf{A}_{\omega^1,\varepsilon, N}\right]_{ij} \right|^2\sigma(\omega_j^1)^{-2}
+\frac{1}{2} \max_{i=\overline{1,M_D}} \left|S^{\mathbf{x}^D_i}_{\varepsilon,N}\left(1_{\Gamma_{\omega^{1,\varepsilon}}}\right)-\mathbb{E}\left[1_{\Gamma_{\omega^{1,\varepsilon}}}\left(X^{\mathbf{x}_i^D}_{N_{\varepsilon}^{\mathbf{x}_i^D}}\right)\right]\right|^2.
\end{align*}
Finally, using the union bound inequality, Hoeffding's inequality \cite[Theorem 2]{Ho63} and \eqref{eq:extended_weights_1}, one obtains 
\begin{align*}
&\mathbb{P}\left\{\left\|\boldsymbol{\Lambda}^{\boldsymbol{\nu}} - 
\boldsymbol{\Lambda}^{\boldsymbol{\nu}}_{\omega^1,\varepsilon, N}\right\|_{\rm F}\geq \gamma\right\}\\
&\leq \mathbb{P}\left\{
\max_{\substack{i=\overline{1,M_D}\\ j=\overline{1,M_1}}} \left|\left[\mathbf{A}_{\omega^1,\varepsilon}\right]_{ij } -\left[\mathbf{A}_{\omega^1,\varepsilon, N}\right]_{ij} \right|\geq \frac{\gamma_\Lambda}{3} 
\min_{j=\overline{1,M_1}} 
\sigma(\omega_j^1)
\left(2 \max_{i=\overline{1,M_D}} \mu_{\mathbf{x}^D_i}(\Gamma_1) + 
2\max_{i=\overline{1,M_D}} \mathcal{C}_1(\mathbf{x}^D_i)\right)^{-1}\right\}\\
&\quad +\mathbb{P}\left\{
\max_{\substack{i=\overline{1,M_D}\\ j=\overline{1,M_1}}} \left|\left[\mathbf{A}_{\omega^1,\varepsilon}\right]_{ij } -\left[\mathbf{A}_{\omega^1,\varepsilon, N}\right]_{ij} \right|\geq \sqrt{2{\gamma_\Lambda}/3} 
\min_{j=\overline{1,M_1}} 
\sigma(\omega_j^1)\right\}\\
&\quad + \mathbb{P}\left\{
\max_{i=\overline{1,M_D}} \left|S^{\mathbf{x}^D_i}_{\varepsilon,N}\left(1_{\Gamma_{\omega^{1,\varepsilon}}}\right)-\mathbb{E}\left[1_{\Gamma_{\omega^{1,\varepsilon}}}\left(X^{\mathbf{x}_i^D}_{N_{\varepsilon}^{\mathbf{x}_i^D}}\right)\right]\right|\geq \sqrt{2{\gamma_\Lambda}/3}\right\}\\
&\leq M_1M_D 
\max\limits_{\substack{i=\overline{1,M_D}\\ j=\overline{1,M_1}}} 
\mathbb{P}\left\{ \left|\left[\mathbf{A}_{\omega^1,\varepsilon}\right]_{ij } -\left[\mathbf{A}_{\omega^1,\varepsilon, N}\right]_{ij} \right|\geq \dfrac{\gamma_\Lambda}{3} 
\min_{j=\overline{1,M_1}} 
\sigma(\omega_j^1) 
\left(2 \max_{i=\overline{1,M_D}} \mu_{\mathbf{x}^D_i}(\Gamma_1) + 
2 \max_{i=\overline{1,M_D}} 
\mathcal{C}_1(\mathbf{x}^D_i)\right)^{-1}\right\}\\
&\quad +M_1 M_D 
\max\limits_{\substack{i=\overline{1,M_D}\\ j=\overline{1,M_1}}} 
\mathbb{P}\left\{ \left|\left[\mathbf{A}_{\omega^1,\varepsilon}\right]_{ij } -\left[\mathbf{A}_{\omega^1,\varepsilon, N}\right]_{ij} \right|\geq \sqrt{2{\gamma_\Lambda}/3} 
\min_{j=\overline{1,M_1}} \sigma(\omega_j^1)\right\}\\
&\quad + M_1 \max_{i=\overline{1,M_D}} 
\mathbb{P}\left\{ \left|S^{\mathbf{x}^D_i}_{\varepsilon,N}\left(1_{\Gamma_{\omega^{1,\varepsilon}}}\right)-\mathbb{E}\left[1_{\Gamma_{\omega^{1,\varepsilon}}}\left(X^{\mathbf{x}_i^D}_{N_{\varepsilon}^{\mathbf{x}_i^D}}\right)\right]\right|\geq \sqrt{2{\gamma_\Lambda}/3}\right\}\\
&\leq 2M_1M_D\bigg[ \mathrm{e}^{-N {\gamma_\Lambda}^2/9 
\min\limits_{j=\overline{1,M_1}} 
\sigma(\omega_j^1)^2 
\left(2 \max\limits_{i=\overline{1,M_D}} 
\mu_{\mathbf{x}^D_i}(\Gamma_1) + 
2 \max\limits_{i=\overline{1,M_D}} 
\mathcal{C}_1(\mathbf{x}^D_i)\right)^{-2}} + \mathrm{e}^{-2N {\gamma_\Lambda}/3 \, 
\min\limits_{j=\overline{1,M_1}} 
\sigma(\omega_j^1)^2}\bigg] + 
2 M_1 \mathrm{e}^{-2N{\gamma_\Lambda}/3}.
\end{align*}
We complete the proof by noting that the convergence \eqref{eq:convergence_Lambda} follows similarly to the corresponding one from \Cref{coro:rho_MC_est}.

\end{proof}

\begin{thm}\label{coro:tail_lambda}
Under the assumptions from \Cref{coro:MCmain}, we let $k \coloneqq {\rm dim}\big( {\rm Ker}(T_{\boldsymbol{\nu}}^\ast T_{\boldsymbol{\nu}}) \big)^{\perp} \leq M_D$, hence $k$ is the number of nonzero eigenvalues $\lambda_1 \geq \lambda_2 \geq \ldots \geq \lambda_k > 0$ of the operator $T_{\boldsymbol{\nu}}^\ast T_{\boldsymbol{\nu}}$ given by \eqref{eq:Th}-\eqref{eq:Th*} and set $\boldsymbol\lambda^{\boldsymbol{\nu}} \coloneqq (\lambda_1, \dots, \lambda_2, \dots, \lambda_k, 0, \ldots, 0) \in \mathbb{R}^{M_D}$. 
Furthermore, let $\widetilde{\lambda}_1 \geq \widetilde{\lambda}_2 \geq \ldots \geq \widetilde{\lambda}_{M_D}$ be the (possibly zero) eigenvalues of the matrix $\boldsymbol{\Lambda}^{\boldsymbol{\nu}}_{\omega^1,\varepsilon, N}$ given by \eqref{eq:Lambda_tilde_nu} and set $\widetilde{\boldsymbol\lambda}^{\boldsymbol{\nu}} \coloneqq \big( \widetilde{\lambda}_1, \widetilde{\lambda}_2, \ldots, \widetilde{\lambda}_{M_D} \big) \in \mathbb{R}^{M_D}$. 
For a symmetric matrix $\mathbf{Q}$ its $i-$th gap ${\rm gap}_i(\mathbf{Q})$ is defined by \eqref{eq:gap}, whilst ${\rm\bf gap}(\mathbf{Q})$ denotes the vector whose entries are ${\rm gap}_i(Q)$, $i = \overline{1,k}$.

Then, for all $\gamma > 0$, the following estimates hold
\begin{align}
&\mathbb{P}\left\{ 
\vert \boldsymbol\lambda^{\boldsymbol{\nu}} - \widetilde{\boldsymbol\lambda}^{\boldsymbol{\nu}} \vert \geq \gamma \right\} 
\leq \delta_\Lambda(\gamma),
\label{eq:tail_lambda}\\
&\mathbb{P}\left\{ 
\left\vert {\rm gap}_i \big( \boldsymbol{\Lambda}^{\boldsymbol{\nu}}_{\omega^1, \varepsilon, N}) - 
{\rm gap}_i \big( T_{\boldsymbol{\nu}}^\ast T_{\boldsymbol{\nu}} \big) \right\vert \geq \gamma \right\} 
\leq \delta_\Lambda(\gamma/\sqrt{3}), \quad i=\overline{1,k},
\label{eq:tail_gap_i}\\
&\mathbb{P}\left\{ 
\left\vert {\rm \bf gap} \big( \boldsymbol{\Lambda}^{\boldsymbol{\nu}}_{\omega^1,\varepsilon, N} \big) - 
{\rm \bf gap} \big( T_{\boldsymbol{\nu}}^\ast T_{\boldsymbol{\nu}} \big) \right\vert \geq \gamma \right\} 
\leq \delta_\Lambda(\gamma/3),
\label{eq:tail_gap},
\end{align}
where $\delta_\Lambda$ is given by \eqref{eq:delta_Lambda}.

In particular, if the elliptic densities $\rho_{\mathbf{x}_i^D}$, $i = \overline{1,M_D}$ are continuous on $\overline{\Gamma_1}$ and the assumptions from \Cref{coro: error estimates} or \Cref{coro:mu - mu eepsilon-Voronoi} are fulfilled, then  
\begin{equation}\label{eq:convergence_eigenvalues}
\lim_{{\rm diam}(\omega^1) \to 0} 
\lim_{\varepsilon \to 0} 
\lim_{N \to\infty} \left\{ 
\vert \boldsymbol\lambda^{\boldsymbol{\nu}} - 
\widetilde{\boldsymbol\lambda}^{\boldsymbol{\nu}} \vert + 
\left\vert 
{\rm \bf gap} \big( \boldsymbol{\Lambda}^{\boldsymbol{\nu}}_{\omega^1,\varepsilon, N} \big) - 
{\rm \bf gap} \big( T_{\boldsymbol{\nu}}^\ast T_{\boldsymbol{\nu}} \big) \right\vert 
\right\} = 0 
\quad \mathbb{P}\mbox{--a.s. and in } L^1(\mathbb{P}).
\end{equation}
\end{thm}
\begin{proof}
According to \eqref{coro:B-Lambda}, $\mathbf{\Lambda}^{\boldsymbol{\nu}}$ and $T_{\boldsymbol{\nu}}^\ast T_{\boldsymbol{\nu}}$ share the same eigenvalues. 
Therefore, the estimate \eqref{eq:tail_lambda} follows directly from \Cref{thm:Hoffman-Wielandt} and \Cref{coro:MCmain}, and hence \eqref{eq:convergence_eigenvalues} as well.
Further, for $i = \overline{1,k}$, we let $i_\ast, \widetilde{i}_\ast \in \overline{1,M_D}$ be such that 
\begin{equation*}
{\rm gap}_i\big( \boldsymbol{\Lambda}^{\boldsymbol{\nu}}_{\omega^1,\varepsilon, N} \big) = 
\left\vert \widetilde{\lambda}_i - \widetilde{\lambda}_{\widetilde{i}_\ast} \right\vert 
\quad \mbox{and} \quad 
{\rm gap}_i\big( T_{\boldsymbol{\nu}}^\ast T_{\boldsymbol{\nu}} \big) = 
\big\vert \lambda_i - \lambda_{i_\ast} \big\vert.
\end{equation*}
Note that $\widetilde{i}_\ast$ is random. 
Nevertheless, since $\left\vert \widetilde{\lambda}_i - \widetilde{\lambda}_{\widetilde{i}_\ast} \right\vert \leq 
\left\vert \widetilde{\lambda}_i - \widetilde{\lambda}_{i_\ast} \right\vert$ and $\big\vert \lambda_i - \lambda_{i_\ast} \big\vert \leq \big\vert \lambda_i - \lambda_{\widetilde{i}_\ast} \big\vert$, it follows easily that
\begin{align*}
\left\vert {\rm gap}_i\big( \boldsymbol{\Lambda}^{\boldsymbol{\nu}}_{\omega^1,\varepsilon, N} \big) - 
{\rm gap}_i\big (T_{\boldsymbol{\nu}}^\ast T_{\boldsymbol{\nu}} \big) \right\vert 
&\leq \left\vert \widetilde{\lambda}_i-\lambda_i \right\vert + 
\left\vert \widetilde{\lambda}_{\widetilde{i}_\ast} - \widetilde{\lambda}_{i_\ast} \right\vert + 
\left\vert \widetilde{\lambda}_{i_\ast} - \lambda_{i_\ast} \right\vert\\
&\leq\sqrt{3} \left( 
\left\vert \widetilde{\lambda}_i-\lambda_i \right\vert^2 + 
\left\vert \widetilde{\lambda}_{\widetilde{i}_\ast} - \widetilde{\lambda}_{i_\ast} \right\vert^2 + 
\left\vert \widetilde{\lambda}_{i_\ast} - \lambda_{i_\ast} \right\vert^2
\right)^{1/2}
\end{align*}
and hence relations \eqref{eq:tail_gap_i} and \eqref{eq:tail_gap} follow from \eqref{eq:tail_lambda}. 

Finally, the convergence \eqref{eq:convergence_eigenvalues} follows analogously to that corresponding to \Cref{coro:rho_MC_est}.

\end{proof}

\begin{thm}\label{coro:tail_eigenvectors}
Under the assumptions from \Cref{coro:MCmain} and \Cref{coro:tail_lambda}, we let $\widetilde{\lambda}_1 \geq \widetilde{\lambda}_2 \geq \ldots \geq \widetilde{\lambda}_{M_D}$ be the (possibly zero) eigenvalues of the random matrix $\boldsymbol{\Lambda}^{\boldsymbol{\nu}}_{\omega^1,\varepsilon, N}$ given by \eqref{eq:Lambda_tilde_nu}, $\big\{ \mathbf{\widetilde{u}}_1, \mathbf{\widetilde{u}}_2, \ldots, \mathbf{\widetilde{u}}_{M_D} \big\} \subset \mathbb{R}^{M_D}$ be a corresponding orthonormal basis consisting of eigenvectors of $\boldsymbol{\Lambda}^{\boldsymbol{\nu}}_{\omega^1,\varepsilon, N}$ and define
\begin{equation}\label{eq:eigenfunctions_MC}
\widetilde{u}_j(\mathbf{x}) \coloneqq  
\displaystyle \big( \widetilde{\lambda}_j \big)^{-1/2} 
\sum \limits_{i=1}^{M_D} \nu_i 
\big[ \mathbf{\widetilde{u}}_j \big]_i 
\rho_{\mathbf{x}_i^D, \omega^1, \varepsilon, N}(\mathbf{x}), 
\quad j = \overline{1, k},    
\end{equation} 
Then there exists an orthonormal basis $\big\{ u_1, u_2, \ldots, u_k \big\} \subset {\rm Ker}\big( T_{\boldsymbol{\nu}}^\ast T_{\boldsymbol{\nu}} \big)^{\perp} \subset L^2(\Gamma_1)$ consisting of eigenfunctions of $T_{\boldsymbol{\nu}}^\ast T_{\boldsymbol{\nu}}$ such that 
\begin{equation*}
\mathbb{P}\left\{ {\rm gap}_j\big( \boldsymbol{\Lambda}^{\boldsymbol{\nu}}_{\omega^1,\varepsilon, N} \big) 
\widetilde{\lambda}_j^{1/2} 
\displaystyle \sup_{\mathbf{x} \in \Gamma_1} 
\big\vert u_j(\mathbf{x}) - \widetilde{u}_j(\mathbf{x}) \big\vert \geq \gamma \right\} 
\leq \widetilde{\delta}_{u_j}(\gamma), 
\quad j = \overline{1, k}, 
\quad \gamma > 0,
\end{equation*}
where 
\begin{align*}
\widetilde{\delta}_{u_j}(\gamma) \coloneqq  
\delta_\Lambda(\gamma) 
&+ \delta_\Lambda \big( \gamma \Vert {\rm diag}(\boldsymbol{\nu} \big) \overline{\rho}_{\Gamma_1} \Vert^{-1} 
\big( 4 \sqrt{2} \big)^{-1/2}) + 
\delta_\Lambda\Big( \gamma^2 
\big( \gamma + 
\Vert \boldsymbol{\Lambda}^{\boldsymbol{\nu}} \Vert_{\rm F} \big)^{-2} 
\Vert {\rm diag}(\boldsymbol{\nu}) \overline{\rho}_{\Gamma_1} \Vert^{-2} \lambda_j \Big) 
\nonumber \\
&+ 2M_1 \displaystyle \sum_{i=1}^{M_D} 
\mathrm{e}^{-N \big[ \delta_{\rho,i} 
\big( \gamma (\gamma + 
\Vert \boldsymbol{\Lambda}^{\boldsymbol{\nu}} \Vert_{\rm F} \big)^{-1}) \big]^2},
\end{align*}
whilst $\delta_\lambda$ and $\delta_{\rho,i}$ are defined by relations \eqref{eq:delta_Lambda} and \eqref{eq:delta_rho_i}, respectively.

In particular, if the elliptic densities $\rho_{\mathbf{x}_i^D}$, $i = \overline{1,M_D}$, are continuous on $\overline{\Gamma_1}$,and the assumptions from \Cref{coro: error estimates} or \Cref{coro:mu - mu eepsilon-Voronoi} are fulfilled, then 
\begin{equation*}
\lim_{{\rm diam}(\omega^1) \to 0} 
\lim_{\varepsilon \to 0} 
\lim_{N \to \infty}\left\{ 
{\rm gap}_j\left( \boldsymbol{\Lambda}^{\boldsymbol{\nu}}_{\omega^1,\varepsilon, N}\right) \widetilde{\lambda}_j^{1/2} 
\displaystyle \sup_{\mathbf{x}\in \Gamma_1} 
\big\vert u_j(\mathbf{x}) - \widetilde{u}_j(\mathbf{x}) \big\vert| \right\} = 0 
\quad \mathbb{P}\mbox{--a.s. and in } L^1(\mathbb{P}), 
\quad j = \overline{1,k}.
\end{equation*}
\end{thm}
\begin{proof}
According to \Cref{thm:Davis-Kahan}, there exists $\big\{ \mathbf{u}_1, \mathbf{u}_2, \ldots, \mathbf{u}_k\big\} \subset {\rm Ker}\big( \boldsymbol{\Lambda}^{\boldsymbol{\nu}} \big)^\perp \subset \mathbb{R}^{M_D}$ an orthonormal basis consisting of eigenvectors of $\boldsymbol{\Lambda}^{\boldsymbol{\nu}}$ such that
\begin{equation*}
\Vert \mathbf{\widetilde u}_i - \mathbf{u}_i \Vert \leq 
\sqrt{4 \sqrt{2}} 
\dfrac{\big\Vert \boldsymbol{\Lambda}^{\boldsymbol{\nu}}_{\omega^1,\varepsilon, N} - 
\boldsymbol{\Lambda}^{\boldsymbol{\nu}} \big\Vert_{\rm F}} {{\rm gap}_i \left( \boldsymbol{\Lambda}^{\boldsymbol{\nu}}_{\omega^1,\varepsilon, N} \right)}, 
\quad i = \overline{1,k}.
\end{equation*}
Further, in view of \eqref{eq:Th*}, define
\begin{equation*}
u_j \coloneqq T_{\boldsymbol{\nu}}^\ast \mathbf{u}_j 
= \displaystyle \left( \lambda_j \right)^{-1/2} 
\sum \limits_{i=1}^{M_D} \nu_i 
\big[ \mathbf{u}_j \big]_i \rho_{\mathbf{x}_i^D}, 
\quad j = \overline{1, k},
\end{equation*}
so that, by \Cref{coro:B-Lambda}, relation \eqref{eq:basisB}, $\big( u_i \big)_{i=\overline{1,k}}$ forms an orthonormal basis of ${\rm Ker}\big( T_{\boldsymbol{\nu}}^\ast T_{\boldsymbol{\nu}} \big)^{\perp} \subset L^2(\Gamma_1)$ consisting of eigenfunctions of $T_{\boldsymbol{\nu}}^\ast T_{\boldsymbol{\nu}}$.

On the one hand, for $\mathbf{x} \in \Gamma_1$ and $\overline{\rho}_{\Gamma_1}$ given by \eqref{eq:notation.d}, the following relations hold
\begin{align*}
&\big( \widetilde{\lambda}_j \big)^{1/2} 
\big\vert u_j(\mathbf{x}) - \widetilde{u}_j(\mathbf{x}) \big\vert\\
&\leq 
\left\vert \big( \widetilde{\lambda}_j \big)^{1/2} 
\big( \lambda_j \big)^{-1/2} - 1 \right\vert 
\left\vert \displaystyle \sum \limits_{i=1}^{M_D} \nu_i 
\big[ \mathbf{u}_j \big]_i 
\rho_{\mathbf{x}_i^D}(\mathbf{x}) \right\vert 
+ \left\vert 
\displaystyle \sum \limits_{i=1}^{M_D} \nu_i 
\big[ \mathbf{u}_j \big]_i 
\rho_{\mathbf{x}_i^D}(\mathbf{x}) - 
\displaystyle \sum \limits_{i=1}^{M_D} \nu_i 
\big[ \mathbf{\widetilde u}_j \big]_i \rho_{\mathbf{x}_i^D,\omega^1,\varepsilon, N}(\mathbf{x}) 
\right\vert\\
&\leq \Vert {\rm diag}(\boldsymbol{\nu}) \overline{\rho}_{\Gamma_1} \Vert 
\left\vert \big( \widetilde{\lambda}_j \big)^{1/2} 
\big( \lambda_j \big)^{-1/2} - 1 \right\vert\\
&\quad + \sum \limits_{i=1}^{M_D} \nu_i \left|[\mathbf{u}_j]_{i} -[\mathbf{\tilde u}_j]_{i}\right|\rho_{\mathbf{x}_i^D}(\mathbf{x})+\sum \limits_{i=1}^{M_D} \nu_i \left| [\mathbf{\tilde u}_j]_{i} \right| \left|\rho_{\mathbf{x}_i^D}(\mathbf{x})-\rho_{\mathbf{x}_i^D,\omega^1,\varepsilon, N}(\mathbf{x})\right|\\
&\leq \Vert {\rm diag}(\boldsymbol{\nu}) \overline{\rho}_{\Gamma_1} \Vert 
\left( 
\left\vert \big( \widetilde{\lambda}_j \big)^{1/2} 
\big( \lambda_j \big)^{-1/2} - 1 \right\vert + 
\left\Vert \mathbf{u}_j - \mathbf{\widetilde{u}}_j \right\Vert 
\right)
+ displaystyle \sup_{i = \overline{1,M_D}} 
\displaystyle \sup_{\mathbf{x} \in \Gamma_1} 
\left\vert \rho_{\mathbf{x}_i^D}(\mathbf{x}) - 
\rho_{\mathbf{x}_i^D,\omega^1,\varepsilon, N}(\mathbf{x}) \right\vert.
\end{align*}
On the other hand, according to \Cref{thm:Weyl} and relation \eqref{eq:gap},
\begin{equation*}
{\rm gap}_j\left( \boldsymbol{\Lambda}_{\omega^1,\varepsilon, N}\right) \leq 
\widetilde\lambda_1 \leq 
\vert \tilde\lambda_1 - \lambda_1 \vert + \lambda_1 \leq 
\Vert \boldsymbol{\Lambda}^{\boldsymbol{\nu}} - \boldsymbol{\Lambda}^{\boldsymbol{\nu}}_{\omega^1,\varepsilon, N} \Vert_{\rm F} + 
\Vert \boldsymbol{\Lambda}^{\boldsymbol{\nu}} \Vert_{\rm F}
\end{equation*}
and hence, using \Cref{thm:Davis-Kahan}, the following relations are obtained on the event $\left[ \Vert \boldsymbol{\Lambda}^{\boldsymbol{\nu}} - \boldsymbol{\Lambda}^{\boldsymbol{\nu}}_{\omega^1,\varepsilon, N} \Vert_{\rm F} \leq \gamma \right]$
\begin{align*}
&{\rm gap}_j\left( 
\boldsymbol{\Lambda}^{\boldsymbol{\nu}}_{\omega^1,\varepsilon, N} \right) 
\big( \widetilde{\lambda}_j \big)^{1/2} 
\displaystyle \sup_{\mathbf{x} \in \Gamma_1} 
\big\vert u_j(\mathbf{x}) - \widetilde{u}_j(\mathbf{x}) \big\vert\\
&\quad \leq \big( \gamma + 
\Vert \boldsymbol{\Lambda}^{\boldsymbol{\nu}} \Vert_{\rm F} \big) 
\left( \Vert {\rm diag}(\boldsymbol{\nu}) \overline{\rho}_{\Gamma_1} \Vert 
\left\vert \big( \widetilde{\lambda}_j \big)^{1/2} 
\big( \lambda_j \big)^{-1/2} - 1 \right\vert + 
\displaystyle \sup_{i=\overline{1,M_D}} 
\displaystyle \sup_{\mathbf{x} \in \Gamma_1} 
\left\vert \rho_{\mathbf{x}_i^D}(\mathbf{x}) - \widetilde\rho_{\mathbf{x}_i^D,\omega^1,\varepsilon,N}(\mathbf{x}) \right\vert \right)\\
&\qquad + {\rm gap}_j \left(
\boldsymbol{\Lambda}^{\boldsymbol{\nu}}_{\omega^1,\varepsilon, N}\right) \, 
\Vert {\rm diag}(\boldsymbol{\nu}) \overline{\rho}_{\Gamma_1} \Vert 
\left\Vert \mathbf{u}_j - \mathbf{\widetilde u}_j \right\Vert\\
&\quad \leq \big( \gamma + 
\Vert \boldsymbol{\Lambda}^{\boldsymbol{\nu}} \Vert_{\rm F} \big) 
\left( \Vert {\rm diag}(\boldsymbol{\nu}) \overline{\rho}_{\Gamma_1} \Vert 
\left\vert \big( \widetilde{\lambda}_j \big)^{1/2} 
\big( \lambda_j \big)^{-1/2} - 1 \right\vert + 
\displaystyle \sup_{i=\overline{1,M_D}} 
\displaystyle \sup_{\mathbf{x} \in \Gamma_1} 
\left\vert \rho_{\mathbf{x}_i^D}(\mathbf{x}) - \widetilde\rho_{\mathbf{x}_i^D,\omega^1,\varepsilon,N}(\mathbf{x}) \right\vert \right)\\
&\qquad + \sqrt{4\sqrt{2}} \, 
\Vert {\rm diag}(\boldsymbol{\nu}) \overline{\rho}_{\Gamma_1} \Vert \,  
\big\Vert \boldsymbol{\Lambda}^\nu - \boldsymbol{\Lambda}^\nu_{\omega^1,\varepsilon,N} \big\Vert_{\rm F}.
\end{align*}
By employing \Cref{coro:rho_MC_est}, \Cref{coro:MCmain}, and \Cref{coro:tail_lambda}, it follows that
\begin{align*}
\mathbb{P}&\left\{
{\rm gap}_j\left( 
\boldsymbol{\Lambda}^{\boldsymbol{\nu}}_{\omega^1,\varepsilon, N} \right) 
\big( \widetilde{\lambda}_j \big)^{1/2} 
\displaystyle \sup_{\mathbf{x} \in \Gamma_1} 
\big\vert u_j(\mathbf{x}) - \widetilde{u}_j(\mathbf{x}) \big\vert \geq \gamma \right\}\\
&\leq \mathbb{P}\left\{ 
\big\Vert \boldsymbol{\Lambda}^{\boldsymbol{\nu}} - \boldsymbol{\Lambda}^{\boldsymbol{\nu}}_{\omega^1,\varepsilon,N} \big\Vert_{\rm F} > \gamma \right\} + 
\mathbb{P}\left\{ 
\left\vert \big( \widetilde{\lambda}_j \big)^{1/2} 
\big( \lambda_j \big)^{-1/2} - 1 \right\vert 
\geq \gamma \left(\gamma + 
\big\Vert \boldsymbol{\Lambda}^{\boldsymbol{\nu}} \big\Vert_{\rm F} \right)^{-1} 
\Vert {\rm diag}(\boldsymbol{\nu}) \overline{\rho}_{\Gamma_1} \Vert^{-1} \right\}\\
&\quad +\mathbb{P}\left\{ 
\big\Vert \boldsymbol{\Lambda}^{\boldsymbol{\nu}} - \boldsymbol{\Lambda}^{\boldsymbol{\nu}}_{\omega^1,\varepsilon,N} \big\Vert_{\rm F} \geq 
\gamma \Big( 4\sqrt{2} \Big)^{-1/2} \, 
\Vert {\rm diag}(\boldsymbol{\nu}) \overline{\rho}_{\Gamma_1} \Vert^{-1} \right\}\\
&\quad + \displaystyle \sum_{i=1}^{M_D} 
\mathbb{P}\left\{ 
\displaystyle \sup_{\mathbf{x} \in \Gamma_1} 
\left\vert \rho_{\mathbf{x}_i^D}(\mathbf{x}) - \widetilde\rho_{\mathbf{x}_i^D,\omega^1,\varepsilon,N}(\mathbf{x}) \right\vert \geq 
\gamma \Big( \gamma + 
\big\Vert \boldsymbol{\Lambda}^{\boldsymbol{\nu}} \big\Vert_{\rm F} \Big)^{-1} \right\}\\
&\leq \delta_\Lambda(\gamma) + 
\mathbb{P}\left\{ 
\left\vert \big( \widetilde{\lambda}_j \big)^{1/2} - 
\big( \lambda_j \big)^{1/2} \right\vert \geq 
\gamma \Big( \gamma + 
\big\Vert \boldsymbol{\Lambda}^{\boldsymbol{\nu}} \big\Vert_{\rm F} \Big)^{-1} 
\Vert {\rm diag}(\boldsymbol{\nu}) \overline{\rho}_{\Gamma_1} \Vert^{-1} 
\big( \lambda_j \big)^{1/2} \right\}\\
& \quad + \delta_\Lambda 
\left[ \gamma \Big( 4\sqrt{2} \Big)^{-1/2} \, 
\Vert {\rm diag}(\boldsymbol{\nu}) \overline{\rho}_{\Gamma_1} \Vert^{-1} \right] + 
2 M_1 \displaystyle \sum_{i=1}^{M_D} 
\mathrm{e}^{-N \{ \delta_{\rho,i} 
[\gamma (\gamma + 
\Vert \boldsymbol{\Lambda}^{\boldsymbol{\nu}} \Vert_{\rm F})^{-1}] \}^2}.
\end{align*}
Finally, relation \eqref{eq:tail_lambda} from \Cref{coro:tail_lambda} yields
\begin{align*}
&\mathbb{P}\left\{ 
\left\vert \big( \widetilde{\lambda}_j \big)^{1/2} - 
\big( \lambda_j \big)^{1/2} \right\vert \geq 
\gamma \Big( \gamma + 
\big\Vert \boldsymbol{\Lambda}^{\boldsymbol{\nu}} \big\Vert_{\rm F} \Big)^{-1} 
\Vert {\rm diag}(\boldsymbol{\nu}) \overline{\rho}_{\Gamma_1} \Vert^{-1} 
\big( \lambda_j \big)^{1/2} \right\}\\
&\leq \mathbb{P}\left\{ 
\big\vert \widetilde{\lambda}_j - \lambda_j \big\vert \geq 
\gamma^2 \Big( \gamma + 
\big\Vert \boldsymbol{\Lambda}^{\boldsymbol{\nu}} \big\Vert_{\rm F} \Big)^{-2} 
\Vert {\rm diag}(\boldsymbol{\nu}) \overline{\rho}_{\Gamma_1} \Vert^{-2} 
\lambda_j \right\}
\leq 
\delta_\Lambda \left[ \gamma^2 \Big( \gamma + 
\big\Vert \boldsymbol{\Lambda}^{\boldsymbol{\nu}} \big\Vert_{\rm F} \Big)^{-2} 
\Vert {\rm diag}(\boldsymbol{\nu}) \overline{\rho}_{\Gamma_1} \Vert^{-2} 
\lambda_j \right]
\end{align*}
which proves the claimed tail bound. The convergence in expectation \eqref{eq:convergence_eigenvalues} is proven analogously to that of the corresponding one from \Cref{coro:rho_MC_est}.

\end{proof}

\subsection{Tail error estimates for the inverse problem}
Now, we are in the position to tackle the ultimate theoretical aim of this work, namely to analyse the reconstruction of a truncated spectrum of solutions to the inverse problem \eqref{eq:TTh}.
To do this, we let $u^{(r)}$, $u^{(r)}_{\omega^1,\varepsilon, N}$ and $\Upupsilon_{r}\big( \boldsymbol{\Lambda}^{\boldsymbol{\nu}} \big)$ be given by \eqref{eq: truncated solutions relation}, \eqref{eq:u_omega_M1_r_MC} and \eqref{eq:notation.b}, respectively, and set
\begin{equation*}
\widetilde{\Upupsilon}_{r}\big( \boldsymbol{\Lambda}^{\boldsymbol{\nu}} \big) \coloneqq 
\dfrac{2 \lambda_1\big( \boldsymbol{\Lambda}^{\boldsymbol{\nu}} \big) + 
2 \lambda_r\big( \boldsymbol{\Lambda}^{\boldsymbol{\nu}} \big) - 
\lambda_{r+1}\big( \boldsymbol{\Lambda}^{\boldsymbol{\nu}} \big)}{\lambda_r\big( \boldsymbol{\Lambda}^{\boldsymbol{\nu}} \big)^2}.
\end{equation*}

\begin{thm}\label{coro:tail_solutions}
Under the assumptions from \Cref{coro:MCmain} and \Cref{coro:tail_lambda}, Further, for a fixed $1\leq $ we set
\begin{align*}
&\gamma_1 \coloneqq \gamma \, \dfrac{\lambda_r\big( \boldsymbol{\Lambda}^\nu{\boldsymbol{\nu}} \big)}{4 \, \Vert \mathbf{b}^{\boldsymbol{\nu}} \Vert}, \quad 
\gamma_2 \coloneqq \dfrac{\lambda_r\big( \boldsymbol{\Lambda}^{\boldsymbol{\nu}}) - \lambda_{r+1}(\boldsymbol{\Lambda}^{\boldsymbol{\nu}} \big)}{2} \wedge  \dfrac{\gamma}{2 \Upupsilon_{r}\big( \boldsymbol{\Lambda}^{\boldsymbol{\nu}} \big) \, 
\big\Vert \mathbf{b}^{\boldsymbol{\nu}} \big\Vert \, 
\big\Vert {\rm diag}(\boldsymbol{\nu}) \overline{\rho}_{\Gamma_1} \big\Vert}, \quad 
\gamma_3 \coloneqq \dfrac{\lambda_{r}\big( \mathbf{\Lambda}^{\boldsymbol{\nu}} \big)}{2} \wedge \gamma, 
\quad r = \overline{1,k}.
\end{align*}
Then, the following estimate holds
\begin{equation*}
\mathbb{P}\left\{ 
\displaystyle \sup \limits_{\mathbf{x} \in \Gamma_1} 
\Big\vert u^{(r)}(\mathbf{x}) - u^{(r)}_{\omega^1,\varepsilon, N}(\mathbf{x}) \Big\vert 
\geq \gamma \right\} \leq 
2 M_1 \displaystyle \sum_{i=1}^{M_D} 
\mathrm{e}^{-N \left[\delta_{\rho,i} \left(\gamma_1\right)\right]^2} +  \delta_\Lambda\left(\gamma_2\right) + 
\delta_\Lambda\left(\gamma_3\right), 
\quad r = \gamma \in (0, 1),
\end{equation*}
where $\delta_{\rho,i}$ and $\delta_\Lambda$ are given by relations \eqref{eq:delta_rho_i} and \eqref{eq:delta_Lambda}, respectively.

In particular, if the elliptic densities $\rho_{\mathbf{x}_i^D}$, $i = \overline{1,M_D}$, are continuous on $\overline{\Gamma_1}$, the assumptions in \Cref{coro: error estimates} or \Cref{coro:mu - mu eepsilon-Voronoi} are fulfilled and $\gamma_2 > 0$, then
\begin{equation}\label{eq:convergence truncated inverse}
\lim_{{\rm diam}(\omega^1) \to 0} 
\lim_{\varepsilon \to 0} 
\lim_{N \to \infty} 
\left\{ 
\sup \limits_{\mathbf{x} \in \Gamma_1} 
\Big\vert u^{(r)}(\mathbf{x})
- u^{(r)}_{\omega^1,\varepsilon, N}(\mathbf{x}) \Big\vert 
\right\} = 0 
\quad \mathbb{P}\mbox{--a.s. and in } L^1(\mathbb{P}).
\end{equation}
\end{thm}
\begin{proof}
The following identities hold
\begin{align*}
&\Big\vert u^{(r)}(\mathbf{x})
- u^{(r)}_{\omega^1,\varepsilon, N}(\mathbf{x}) \Big\vert = 
\Bigg\vert 
\displaystyle \sum \limits_{i=1}^{M_D} \nu_i 
\left[ \mathbf{u}^{(r)} \right]_i 
\rho_{\mathbf{x}^D_i}(\mathbf{x}) - 
\displaystyle \sum \limits_{i=1}^{M_D} \nu_i 
\left[ \mathbf{u}^{(r)}_{\omega^1,\varepsilon, N} \right]_i 
\rho_{\mathbf{x}^D_i,\omega^1,\varepsilon, N}(\mathbf{x}) 
\Bigg\vert \\
&\leq \displaystyle \sum \limits_{i=1}^{M_D} \nu_i 
\left\vert \left[ \mathbf{u}^{(r)} \right]_i - 
\left[ \mathbf{u}^{(r)}_{\omega^1,\varepsilon, N} \right]_i \right\vert \, 
\left\vert \rho_{\mathbf{x}^D_i}(\mathbf{x}) \right\vert + 
\displaystyle \sum \limits_{i=1}^{M_D} \nu_i 
\left\vert \left[ \mathbf{u}^{(r)}_{\omega^1,\varepsilon, N} \right]_i \right\vert \, 
\left\vert \rho_{\mathbf{x}^D_i}(\mathbf{x}) - \rho_{\mathbf{x}^D_i,\omega^1,\varepsilon, N}(\mathbf{x}) \right\vert.
\end{align*}
Hence, with $\overline{\rho}_{\Gamma_1}$ given by \eqref{eq:notation.d}, one obtains
\begin{align*}
\displaystyle \sup \limits_{\mathbf{x} \in \Gamma_1} 
\Big\vert u^{(r)}(\mathbf{x})
- u^{(r)}_{\omega^1,\varepsilon, N}(\mathbf{x}) \Big\vert 
&\leq \big\Vert \mathbf{u}^{(r)} -  \mathbf{u}^{(r)}_{\omega^1,\varepsilon, N} \big\Vert \, 
\big\Vert {\rm diag}(\boldsymbol{\nu}) \overline{\rho}_{\Gamma_1} \big\Vert\\
&+ \big\Vert \mathbf{u}^{(r)}_{\omega^1,\varepsilon, N} \big\Vert \, 
\displaystyle \sup \limits_{i = \overline{1,M_D}} 
\displaystyle \sup \limits_{\mathbf{x} \in \Gamma_1}  
\left\vert \rho_{\mathbf{x}^D_i}(\mathbf{x}) - 
\rho_{\mathbf{x}^D_i,\omega^1,\varepsilon, N}(\mathbf{x}) \right\vert.
\end{align*}
Note that $\big\Vert \mathbf{u}^{(r)}_{\omega^1,\varepsilon, N} \big\Vert \leq \dfrac{\big\Vert \mathbf{b}^{\boldsymbol{\nu}} \big\Vert}{\lambda_r\big( \boldsymbol{\Lambda}^{\boldsymbol{\nu}}_{\omega^1,\varepsilon, N} \big)}$. 

On the event 
\begin{equation*}
E_1 \coloneqq \left[ 
\Big\Vert \boldsymbol{\Lambda}^{\boldsymbol{\nu}} - \boldsymbol{\Lambda}^{\boldsymbol{\nu}}_{\omega^1,\varepsilon, N} \Big\Vert_{\rm F} \leq 
\dfrac{\lambda_r\big( \boldsymbol{\Lambda}^{\boldsymbol{\nu}}) - \lambda_{r+1}(\boldsymbol{\Lambda}^{\boldsymbol{\nu}} \big)}{2} \wedge  \dfrac{\gamma}{2 \Upupsilon_{r}\big( \boldsymbol{\Lambda}^{\boldsymbol{\nu}} \big) \, 
\big\Vert \mathbf{b}^{\boldsymbol{\nu}} \big\Vert \, 
\big\Vert {\rm diag}(\boldsymbol{\nu}) \overline{\rho}_{\Gamma_1} \big\Vert}
\right],    
\end{equation*}
similarly to {\bf Step IV} in the proof of \Cref{thm:main1}, the following relations hold
\begin{equation*}
\big\Vert \mathbf{u}^{(r)} -  \mathbf{u}^{(r)}_{\omega^1,\varepsilon, N} \big\Vert \leq 
\Upupsilon_{r}\big( \boldsymbol{\Lambda}^{\boldsymbol{\nu}} \big) 
\left\Vert \boldsymbol{\Lambda}^{\boldsymbol{\nu}} - \boldsymbol{\Lambda}^{\boldsymbol{\nu}}_{\omega^1,\varepsilon, N} \right\Vert \, 
\Vert \mathbf{b}^{\boldsymbol{\nu}} \Vert \leq 
\dfrac{\gamma}{2 \big\Vert {\rm diag}(\boldsymbol{\nu}) \overline{\rho}_{\Gamma_1} \big\Vert},
\end{equation*}
and, consequently,
\begin{align*}
&\displaystyle \sup \limits_{\mathbf{x} \in \Gamma_1} 
\Big\vert u^{(r)}(\mathbf{x})
- u^{(r)}_{\omega^1,\varepsilon, N}(\mathbf{x}) \Big\vert 
\leq \dfrac{\gamma}{2} + 
\dfrac{\big\Vert \mathbf{b}^{\boldsymbol{\nu}} \big\Vert}{\lambda_r\big( \boldsymbol{\Lambda}^{\boldsymbol{\nu}}_{\omega^1,\varepsilon, N} \big)} 
\displaystyle \sup \limits_{\mathbf{x} \in \Gamma_1}  
\left\vert \rho_{\mathbf{x}^D_i}(\mathbf{x}) - 
\rho_{\mathbf{x}^D_i,\omega^1,\varepsilon, N}(\mathbf{x}) \right\vert.
\end{align*}
Also, note that 
$\lambda_r\big( \boldsymbol{\Lambda}^{\boldsymbol{\nu}})  \leq 
\big\vert \lambda_{r}\big( \mathbf{\Lambda}^{\boldsymbol{\nu}} c\big) - 
\lambda_{r}\big( \mathbf{\Lambda}^{\boldsymbol{\nu}}_{\omega^1,\varepsilon, N} \big) \big\vert + 
\lambda_{r}\big( \mathbf{\Lambda}^{\boldsymbol{\nu}}_{\omega^1,\varepsilon, N} \big)$ and with the notations used in \Cref{coro:tail_lambda}, on the event
\[ 
E_2 \coloneqq \left[ 
\big\Vert \boldsymbol\lambda^{\boldsymbol{\nu}} - \widetilde{\boldsymbol\lambda}^{\boldsymbol{\nu}} \big\Vert \leq 
\dfrac{\lambda_{r}\big( \mathbf{\Lambda}^{\boldsymbol{\nu}} \big)}{2} \wedge \gamma \right],
\]
the following inequality holds 
\[
\lambda_{r}\big( \mathbf{\Lambda}^{\boldsymbol{\nu}}_{\omega^1,\varepsilon, N} \big) \geq 
\lambda_{r}\big( \mathbf{\Lambda}^{\boldsymbol{\nu}} \big) - 
\big\Vert \boldsymbol\lambda^{\boldsymbol{\nu}} - \widetilde{\boldsymbol\lambda}^{\boldsymbol{\nu}} \big\Vert.
\]
Therefore, on the event $E_1 \cap E_2$, one obtains 
\begin{equation*}
\sup \limits_{\mathbf{x} \in \Gamma_1} 
\left\vert u^{(r)}(\mathbf{x})
- u^{(r)}_{\omega^1,\varepsilon, N}(\mathbf{x}) \right\vert 
\leq \dfrac{\gamma}{2} + 
\dfrac{2 \, \big\Vert \mathbf{b}^{\boldsymbol{\nu}} \big\Vert}{\lambda_r\big( \boldsymbol{\Lambda}^{\boldsymbol{\nu}} \big)} 
\displaystyle \sup \limits_{i = \overline{1,M_D}} 
\displaystyle \sup \limits_{\mathbf{x} \in \Gamma_1}  
\left\vert \rho_{\mathbf{x}^D_i}(\mathbf{x}) - 
\rho_{\mathbf{x}^D_i,\omega^1,\varepsilon, N}(\mathbf{x}) \right\vert,
\end{equation*}
hence
\begin{align*}
& \mathbb{P}\left\{ 
\sup \limits_{\mathbf{x} \in \Gamma_1} 
\left\vert u^{(r)}(\mathbf{x}) - 
u^{(r)}_{\omega^1,\varepsilon, N}(\mathbf{x}) \right\vert \geq \gamma \right\} 
\leq \mathbb{P}\left\{ 
\displaystyle \sup \limits_{i = \overline{1,M_D}} 
\displaystyle \sup \limits_{\mathbf{x} \in \Gamma_1}  
\left\vert \rho_{\mathbf{x}^D_i}(\mathbf{x}) - 
\rho_{\mathbf{x}^D_i,\omega^1,\varepsilon, N}(\mathbf{x}) \right\vert 
\geq \gamma \, 
\dfrac{\lambda_r\big( \boldsymbol{\Lambda}^{\boldsymbol{\nu}} \big)}{4 \, \big\Vert \mathbf{b}^{\boldsymbol{\nu}} \big\Vert} 
\right\}\\
& \quad + \mathbb{P}\left\{ 
\Big\Vert \boldsymbol{\Lambda}^{\boldsymbol{\nu}} - \boldsymbol{\Lambda}^{\boldsymbol{\nu}}_{\omega^1,\varepsilon, N} \Big\Vert_{\rm F} \geq 
\dfrac{\lambda_r\big( \boldsymbol{\Lambda}^{\boldsymbol{\nu}}) - \lambda_{r+1}(\boldsymbol{\Lambda}^{\boldsymbol{\nu}} \big)}{2} \wedge  \dfrac{\gamma}{2 \Upupsilon_{r}\big( \boldsymbol{\Lambda}^{\boldsymbol{\nu}} \big) \, 
\big\Vert \mathbf{b}^{\boldsymbol{\nu}} \big\Vert \, 
\big\Vert {\rm diag}(\boldsymbol{\nu}) \overline{\rho}_{\Gamma_1} \big\Vert}
\right\}\\
&\quad + \mathbb{P}\left\{ 
\big\Vert \boldsymbol\lambda^{\boldsymbol{\nu}} - \widetilde{\boldsymbol\lambda}^{\boldsymbol{\nu}} \big\Vert \geq 
\dfrac{\lambda_{r}\big( \mathbf{\Lambda}^{\boldsymbol{\nu}} \big)}{2} \wedge \gamma 
\right\},
\end{align*}
and the desired result follows by \Cref{coro:rho_MC_est} and \Cref{coro:MCmain}.

The proof of convergence in expectation \eqref{eq:convergence truncated inverse} follows analogously to that corresponding to \Cref{coro:rho_MC_est}.

\end{proof}

\section{Conclusions} \label{section:conclusions}
In this paper, a comprehensive and systematic convergence and stability analysis of the numerical methods and the corresponding algorithm presented in \cite{CiGrMaI} has been carried out.
The main results obtained herein provides one with explicit tail bounds for the errors between the exact quantities of interest and their corresponding Monte Carlo estimators as follows.
\begin{enumerate}[label={\rm (}{\it \alph*}{\rm )}]
\setlength\itemsep{1pt}
\item \Cref{coro:rho_MC_est} gives explicit upper estimates for the probability that the approximation error between the exact elliptic density $\rho_{\mathbf{x}_i^D}$ on the inaccessible boundary $\Gamma_1 \subset \partial D$ and its Monte Carlo estimator $\rho_{\mathbf{x}_i^D, \omega^1, \varepsilon, N}$ defined by \eqref{eq:rho omega1 MC}, exceeds a given threshold.

\item Further, by considering the matrix operator $\Lambda^{\boldsymbol{\nu}} \coloneqq T_{\boldsymbol{\nu}} T_{\boldsymbol{\nu}}^\ast$ and its Monte Carlo estimator $\boldsymbol{\Lambda}^{\boldsymbol{\nu}}_{\omega^1, \varepsilon, N}$ given by \eqref{eq:Lambda_tilde_nu}, explicit upper estimates for the probability that the approximation error between the exact eigenvalues of the operator $T_{\boldsymbol{\nu}}^\ast T_{\boldsymbol{\nu}}$, defined by \eqref{eq:B Th*Th}, and their Monte Carlo estimators, which are the (random) eigenvalues of $\boldsymbol{\Lambda}^{\boldsymbol{\nu}}_{\omega^1, \varepsilon, N}$, exceeds a given threshold have been derived in \Cref{coro:tail_lambda}.

\item Explicit upper estimates for the probability that the approximation error between the eigenvectors of $T_{\boldsymbol{\nu}}^\ast T_{\boldsymbol{\nu}}$ and their Monte Carlo estimators, expressed in terms of an orthonormal eigenbasis of $\left( {\rm Ker} \big( \boldsymbol{\Lambda}^{\boldsymbol{\nu}}_{\omega^1, \varepsilon, N} \big) \right)^{\perp}$, exceeds a given threshold have been obtained in \Cref{coro:tail_eigenvectors}.

\item Finally, by considering the $r-$TSVD solution to the (direct) equation $T_{\boldsymbol{\nu}} u_1 = {\rm diag}(\boldsymbol{\nu}) \; \mathbf{b}$, $\mathbf{b} \in \mathbb{R}^{M_D}$, defined by \eqref{eq:u_r}, as well as its Monte Carlo estimator $u_{\omega^1, \varepsilon, N}^{(r)}$ defined by \eqref{eq:u_omega_M1_r_MC}--\eqref{eq:u_omega_M1_r_MC A}, for various levels of the truncation parameter $r$, explicit upper estimates for the probability that the approximation error between the exact $r-$TSVD solution and its estimator exceeds a given threshold have been derived in \Cref{coro:tail_solutions}.
\end{enumerate}

The results listed above represent quantitative guarantees that the spectral structure (rank, eigenvalues and eigenvectors) of the direct operator associated with the inverse problem \eqref{eq:ICP} can be accurately captured by the spectrum of the random matrix $\boldsymbol{\Lambda}^{\boldsymbol{\nu}}_{\omega^1, \varepsilon, N}$. 
Moreover, a set of TSVD solutions to the original inverse problem can be obtained, whilst the choice of the truncation parameters can be properly retrieved by computing the spectrum of $\boldsymbol{\Lambda}^{\boldsymbol{\nu}}_{\omega^1, \varepsilon, N}$.

\medskip

Further work related to and developments of the convergence and stability analysis provided in this study are the following:
\begin{enumerate}[label={\rm (}{\it \roman*}{\rm )}]
\setlength\itemsep{1pt}
\item The extension the error analysis to less regular domains, e.g., domains with corners and/or boundary or internal cracks. 
In our opinion, a subtle point here is that it would be very interesting not only to extend the analysis to non-smooth domains, but also to understand the impact of the regularity of the boundary on the stability of the inverse problem. According to our understanding, it is expected that the inverse problem \eqref{eq:ICP} becomes less unstable as the boundary where the measurements are taken is more irregular, whilst the inaccessible portion of the boundary is smoother.
This is justified by the heuristic remark that the heat particle starting in the proximity of an irregular boundary has more chances to move away from it and, eventually, hit the (possibly more regular) inaccessible part of the boundary.

\item The extension of the Monte Carlo schemes and the corresponding error analysis to more general operators, including the case of non-homogeneous or non-symmetric coefficients.
\end{enumerate}

\appendix
\section{Spectral and stability inequalities for linear systems} \label{appendix}
\subsection{Spectral inequalities for matrices}
\begin{thm}{(Weyl's inequality for singular values \cite{Weyl12}, \cite[Corollary 8.6.2]{Golub} )}
\label{thm:Weyl} 
Let $\mathbf{A}, \mathbf{B} \in \mathbb{R}^{m\times n}$ with singular values $\sigma_1 \geq \sigma_2 \geq \ldots \geq \sigma_k \geq 0$ and $\widetilde{\sigma}_1 \geq \widetilde{\sigma}_2 \geq \ldots \geq \widetilde{\sigma}_k \geq 0$, respectively, with $k=\min \{m,n \}$. Then
\begin{equation*}
\vert \sigma_i - \widetilde{\sigma}_i \vert \leq 
\Vert \mathbf{A} - \mathbf{B} \Vert, \quad 
i = \overline{1,k}.
\end{equation*}
\end{thm}

\begin{thm}[Hoffman-Wielandt inequality \cite{HoWi53}] \label{thm:Hoffman-Wielandt} 
Let $\mathbf{A}, \mathbf{B} \in \mathbb{R}^{n\times n}$ be two normal matrices with the eigenvalues $\alpha_1, \alpha_2, \ldots, \alpha_n$ and $\beta_1, \beta_2, \ldots, \beta_n$, respectively. Then
\begin{equation*}
\exists~\pi \in S_n \colon \quad 
\displaystyle \sum_{i = 1}^{n}  
\vert \alpha_i - \beta_{\pi(i)} \vert^2 \leq 
\Vert \mathbf{A} - \mathbf{B} \Vert_{\rm F}^2.
\end{equation*}
If $\mathbf{A}, \mathbf{B} \in \mathbb{R}^{n\times n}$ are symmetric and their corresponding eigenvalues are indexed in decreasing order, then by the rearrangement inequality, the permutation $\pi \in S_n$ can be taken to be the identity.
\end{thm}

The following result is essential \cite[Lemma 4]{ElBeWa18} and is a version of the well-known Davis-Kahan inequality concerning eigenspaces perturbation for symmetric matrices.
\begin{thm}[A Davis-Kahan-type inequality \cite{ElBeWa18}]
\label{thm:Davis-Kahan} 
Let $\mathbf{A}, \mathbf{B} \in \mathbb{R}^{n\times n}$ be two symmetric matrices with the eigenvalues $\alpha_1 \geq \alpha_2 \geq \ldots \geq \alpha_n$ and $\beta_1 \geq \beta_2 \geq \ldots \geq \beta_n$, respectively, and $\big\{ \mathbf{v}_1, \mathbf{v}_2, \ldots, \mathbf{v}_n \big\} \subset \mathbb{R}^n$ be an orthonormal basis such that $A \mathbf{v}_i = \alpha_i \mathbf{v}_i$, $i = \overline{1,n}$.
We further set 
\begin{equation}\label{eq:gap}
{\rm gap}_i(\mathbf{A}) \coloneqq 
\displaystyle \min_{\alpha_j \neq \alpha_i} 
\vert \alpha_i - \alpha_j \vert, 
\quad i = \overline{1,n}.    
\end{equation}
Then there exists an orthonormal basis $\big\{ \mathbf{w}_1, \mathbf{w}_2, \ldots, \mathbf{w}_n \big\} \subset \mathbb{R}^n$ such that $\mathbf{B} \mathbf{w}_i = \beta_i \mathbf{w}_i$, $i = \overline{1,n}$, and
\begin{equation*}\label{eq:v-w}
\Vert \mathbf{v}_i - \mathbf{w}_i \Vert \leq 
\sqrt{4 \sqrt{2}} \dfrac{\Vert \mathbf{A} - \mathbf{B} \Vert_{\rm F}}{{\rm gap}_i(\mathbf{A})}, 
\quad i = \overline{1,n}.
\end{equation*}
\end{thm}
\subsection{Inequalities for the truncated generalized inverse}
Herein some well-known results on the perturbations of the TSVD are recalled following mainly \cite{Ha87}. We also refer the reader to the seminal works \cite{St69} and \cite{We73}, as well as \cite{Me10} for some recent improvements.

Let $\mathbf{A} \in \mathbb{R}^{m \times n}$ and $\mathbf{b} \in \mathbb{R}^m$, and consider the linear system associated with these data, namely
\begin{equation} 
\label{eq:syst_A_b}
\mathbf{A} \mathbf{u} = \mathbf{b}.
\end{equation}

We further consider the SVD of matrix $\mathbf{A}$,
\begin{align*}
\nonumber
&\mathbf{A} = \mathbf{U} \mathbf{\Sigma} \mathbf{V}^{\sf T},
\end{align*}
where $\mathbf{U} \in \mathbb{R}^{m \times m}$ and $\mathbf{V}\in \mathbb{R}^{n \times n}$ are unitary matrices, and $\mathbf{\Sigma} \coloneqq \mathrm{diag}\big( \sigma_1(\mathbf{A}), \sigma_2(\mathbf{A}), \ldots, \sigma_{\min\{m,n\}}(\mathbf{A}) \big) \in \mathbb{R}^{m \times n}$ is a diagonal matrix containing the singular values of $\mathbf{A}$ in descending order. 
For any $r = \overline{1, \mathrm{rank}(\mathbf{A}))}$ a regularisation (truncation) parameter, we consider the $r-$TSVD of matrix $\mathbf{A}$
\begin{align*}
\nonumber
&\mathbf{A}_r \coloneqq \mathbf{U} \mathbf{\Sigma}_r \mathbf{V}^{\sf T}, 
\end{align*}
where $\mathbf{\Sigma}_r \coloneqq \mathrm{diag}\big( \sigma_1(\mathbf{A}), \sigma_2(\mathbf{A}), \ldots, \sigma_r(\mathbf{A}) \big) \in \mathbb{R}^{m \times n}$ is the diagonal matrix containing the first $r$ singular values of $\mathbf{A}$ in descending order. 
The {\it $r-$TSVD solution} to the linear system \eqref{eq:syst_A_b}
\begin{align} 
\label{eq:u1r}
&\mathbf{u}_r \coloneqq 
\mathbf{A}^\dag_r \mathbf{b},
\end{align}
where $\mathbf{A}^\dag_r \coloneqq \mathbf{V} \mathbf{\Sigma}_r^{-1} \mathbf{U}^{\sf T}$ is the so-called {\it $r-$TSVD pseudoinverse {\rm /}Moore–Penrose inverse} of $\mathbf{A}$, whilst $\mathbf{\Sigma}_r^{-1} \coloneqq \mathrm{diag}\big( 1/\sigma_1(\mathbf{A}), 1/\sigma_2(\mathbf{A}), \ldots, 1/\sigma_r(\mathbf{A}) \big) \in \mathbb{R}^{n \times m}$.

\begin{thm}{(\cite[Theorems 3.2 and 3.4]{Ha87})}
\label{thm:AHansen}
Let $\mathbf{A}, \widetilde{\mathbf{A}} \in \mathbb{R}^{m \times n}$, $\mathbf{b}, \widetilde{\mathbf{b}} \in \mathbb{R}^m$, and consider $\mathbf{E} \coloneqq \mathbf{A} - \widetilde{\mathbf{A}}$ and $\mathbf{e} \coloneqq \widetilde{\mathbf{b}}-\mathbf{b}$. Further, let $r \in \overline{1,  \mathrm{rank}(\mathbf{A})}$ be a regularisation (truncation) parameter and define 
\begin{equation*}
    k_r \coloneqq \sigma_1(\mathbf{A})/\sigma_r(\mathbf{A}), \quad \eta_r \coloneqq |\mathbf{E}|/\sigma_r(\mathbf{A}), \quad \gamma_r \coloneqq \sigma_{r+1}(\mathbf{A})/\sigma_{r}(\mathbf{A})
\end{equation*}

If $\vert \mathbf{E} \vert < \sigma_r(\mathbf{A}) - \sigma_{r+1}(\mathbf{A})$, then the following estimates hold
\begin{align}
&\dfrac{\big\vert \mathbf{A}^\dag_r - \widetilde{\mathbf{A}}^\dag_r \big\vert}
{\big\vert \mathbf{A}^\dag_r \big\vert} \leq 
\dfrac{3 k_r}{(1 - \eta_r) (1 - \eta_r -\gamma_r)} \dfrac{\vert \mathbf{E} \vert}{\sigma_1(\mathbf{A})}\\
&\dfrac{\vert \widetilde{\mathbf{u}}_r - \mathbf{u}_r \vert}{\vert \mathbf{u}_r \vert} \leq 
\dfrac{k_r}{1 - \eta_r} 
\left( 
\dfrac{\vert \mathbf{E} \vert}{\sigma_1(\mathbf{A})} + \dfrac{\vert \mathbf{e} \vert}{\vert \mathbf{b} \vert} + \dfrac{\eta_r}{1 - \eta_r - \gamma_r} 
\dfrac{\vert \mathbf{A} \mathbf{u}_r - \mathbf{b} \vert}{\vert \mathbf{b} \vert} 
\right) + 
\dfrac{\eta_r}{1 - \eta_r - \gamma_r},
\end{align}
where ${\mathbf{A}}^\dag_r, \widetilde{\mathbf{A}}^\dag_r $ are the  $r-$TSVD of ${\mathbf{A}}$ and $\widetilde{\mathbf{A}}$, respectively, and ${\mathbf{u}}_r, \widetilde{\mathbf{u}}_r$ are the {\it $r-$TSVD solutions} to the linear systems ${\mathbf{A}} {\mathbf{u}} = {\mathbf{b}}$ and $\widetilde{\mathbf{A}} \widetilde{\mathbf{u}} = \widetilde{\mathbf{b}}$, respectively.
\end{thm}

Note that if $\vert \mathbf{E} \vert < \sigma_r(\mathbf{A}) - \sigma_{r+1}(\mathbf{A})$, then, in particular, $\vert \mathbf{E} \vert < \sigma_r(\mathbf{A})$, and by \Cref{thm:Weyl}, we have that $\mathrm{rank}(\widetilde{\mathbf{A}}) \geq r$.

\bigskip
\noindent \textbf{Acknowledgements.} I.C. acknowledges support from the Ministry of Research, Innovation and Digitization (Romania), grant CF-194-PNRR-III-C9-2023.

\addcontentsline{toc}{section}{\refname}
\bibliographystyle{abbrv}

\end{document}